\documentclass{amsart}

\usepackage{kotex}
\usepackage{amsmath}
\usepackage{amsfonts}
\usepackage{amscd}
\usepackage{slashed}

\usepackage{amssymb}
\usepackage{graphicx}
\usepackage{caption}
\usepackage{subcaption}
\usepackage{dsfont}
\usepackage{color}
\usepackage{hyperref}
\usepackage{bbm}
\usepackage{a4wide}
\usepackage{sseq}
\usepackage{tikz-cd}
\usepackage[abs]{overpic} 

\newtheorem{theorem}{Theorem}[section]
\newtheorem{lemma}[theorem]{Lemma}
\newtheorem{proposition}[theorem]{Proposition}
\newtheorem{corollary}[theorem]{Corollary}

\newtheorem{example}[theorem]{Example}

\theoremstyle{definition}
\newtheorem*{definition*}{Definition}
\newtheorem{definition}[theorem]{Definition}

\theoremstyle{remark}
\newtheorem*{remark*}{Remark}
\newtheorem{remark}[theorem]{Remark}

\theoremstyle{question}
\newtheorem*{question*}{Question}

\numberwithin{equation}{section}

\newcommand{\myproof}[2]{Proof of {#1} {#2}}

\begin{document}

\title[Quantization of K\"ahler Manifolds via Brane Quantization]{Quantization of K\"ahler Manifolds via Brane Quantization}

\author{YuTung Yau}
\address{Kavli Institute for the Physics and Mathematics of the Universe (WPI), The University of Tokyo Institutes for Advanced Study, The University of Tokyo, Kashiwa, Chiba 277-8583, Japan}
\email{yu-tung.yau@ipmu.jp}

\thanks{}

\maketitle


\begin{abstract}
	In their physical proposal for quantization \cite{GukWit2009}, Gukov-Witten suggested that, given a symplectic manifold $M$ with a complexification $X$, the A-model morphism spaces $\operatorname{Hom}(\mathcal{B}_{\operatorname{cc}}, \mathcal{B}_{\operatorname{cc}})$ and $\operatorname{Hom}(\mathcal{B}, \mathcal{B}_{\operatorname{cc}})$ should recover holomorphic deformation quantization of $X$ and geometric quantization of $M$ respectively, where $\mathcal{B}_{\operatorname{cc}}$ is a canonical coisotropic A-brane on $X$ and $\mathcal{B}$ is a Lagrangian A-brane supported on $M$.\par
	Assuming $M$ is spin and K\"ahler with a prequantum line bundle $L$, Chan-Leung-Li \cite{ChaLeuLi2023} constructed a subsheaf $\mathcal{O}_{\operatorname{qu}}^{(k)}$ of smooth functions on $M$ with a non-formal star product and a left $\mathcal{O}_{\operatorname{qu}}^{(k)}$-module structure on the sheaf of holomorphic sections of $L^{\otimes k} \otimes \sqrt{K}$.\par
	In this paper, we give a careful treatment of the relation between (holomorphic) deformation quantizations of $M$ and $X$. As a result, Chan-Leung-Li's work \cite{ChaLeuLi2023} provides a mathematical realization of the action of $\operatorname{Hom}(\mathcal{B}_{\operatorname{cc}}, \mathcal{B}_{\operatorname{cc}})$ on $\operatorname{Hom}(\mathcal{B}, \mathcal{B}_{\operatorname{cc}})$. By Fedosov's gluing arguments, we also construct a $\mathcal{O}_{\operatorname{qu}}^{(k)}$-$\overline{\mathcal{O}}_{\operatorname{qu}}^{(k)}$-bimodule structure on the sheaf of smooth sections of $L^{\otimes 2k}$ to realize the actions of $\operatorname{Hom}(\mathcal{B}_{\operatorname{cc}}, \mathcal{B}_{\operatorname{cc}})$ and $\operatorname{Hom}(\overline{\mathcal{B}}_{\operatorname{cc}}, \overline{\mathcal{B}}_{\operatorname{cc}})$ on $\operatorname{Hom}(\overline{\mathcal{B}}_{\operatorname{cc}}, \mathcal{B}_{\operatorname{cc}})$, which is related to the analytic geometric Langlands program.
\end{abstract}

\section{Introduction}
\label{Section 1}
In geometric quantization on a compact, spin and prequantizable K\"ahler manifold $(M, \omega)$, a prequantum line bundle $L$ and the K\"ahler polarization $T^{0, 1}M$ are used to obtain a quantum Hilbert space $\mathcal{H}^{(k)} := H^0(M, L^{\otimes k} \otimes \sqrt{K})$, where $\sqrt{K}$ is a square root of the canonical bundle of $M$. A deformation quantization $(\mathcal{C}^\infty(M, \mathbb{C})[[\hbar]], \star)$ of $(M, \omega)$ is specified by its asymptotic action on $\mathcal{H}^{(k)}$ via Toeplitz operators as $\tfrac{\hbar}{\sqrt{-1}} = \tfrac{1}{k} \to 0$ \cite{Cha2006, Sch2000}. The prototype is the \emph{Bargmann-Fock action}, i.e. the action of $\mathbb{C}[[z^1, ..., z^n, \hbar]][\overline{z}^1, ..., \overline{z}^n]$ on $\mathbb{C}[[z^1, ..., z^n, \hbar]]$ given by $z^i \mapsto z^i \cdot$ and $\overline{z}^j \mapsto \hbar \tfrac{\partial}{\partial z^j}$.\par
In \cite{ChaLeuLi2023}, Chan-Leung-Li used Fedosov connections to glue the Bargmann-Fock action as a local model. First, they obtained a sheaf of algebras $\mathcal{O}_{\operatorname{qu}}^{(k)}$, named the \emph{sheaf of quantizable functions of level $k$}. It is a subsheaf of smooth functions on $M$ equipped with a non-formal star product for which $\mathcal{O}_{\operatorname{qu}}^{(k)}(M)$ is dense in $\mathcal{C}^\infty(M, \mathbb{C})$ as $k \to \infty$. Then, they constructed a left $\mathcal{O}_{\operatorname{qu}}^{(k)}$-module structure on the sheaf $\mathcal{L}^{(k)}$ of holomorphic sections of $L^{\otimes k} \otimes \sqrt{K}$ \footnote{Indeed, here we have modified Chan-Leung-Li's original construction by metaplectic correction and taking the anti-Wick ordering of operators. Their method of construction is still valid for this modification.}. It naturally induces a right $\mathcal{O}_{\operatorname{qu}}^{(k)}$-module structure on the sheaf of holomorphic sections of $(L^\vee)^{\otimes k} \otimes \sqrt{K}$ as well.\par
The conjugate complex manifold $\overline{M}$ has a K\"ahler structure given by $- \omega$. In this paper, we further consider the $(\overline{M}, -\omega)$-counterpart $\overline{\mathcal{O}}_{\operatorname{qu}}^{(k)}$ of $\mathcal{O}_{\operatorname{qu}}^{(k)}$ and the sheaf $\overline{\mathcal{L}}^{(-k)}$ of antiholomorphic sections of $L^{\otimes k} \otimes (\sqrt{K})^\vee$ as a right $\overline{\mathcal{O}}_{\operatorname{qu}}^{(k)}$-module. The main result is as follows.

\begin{theorem}
	\label{Theorem 1.1}
	The sheaf $\mathcal{L}_{\operatorname{sm}}^{(2k)}$ of smooth sections of $L^{\otimes 2k}$ has an $\mathcal{O}_{\operatorname{qu}}^{(k)}$-$\overline{\mathcal{O}}_{\operatorname{qu}}^{(k)}$ bimodule structure such that the natural pairing $\mathcal{L}^{(k)} \otimes \overline{\mathcal{L}}^{(-k)} \to \mathcal{L}_{\operatorname{sm}}^{(2k)}$ is a morphism of $\mathcal{O}_{\operatorname{qu}}^{(k)}$-$\overline{\mathcal{O}}_{\operatorname{qu}}^{(k)}$ bimodules.
\end{theorem}

Our main result can be interpreted as realizing the $\operatorname{Hom}(\mathcal{B}_{\operatorname{cc}}, \mathcal{B}_{\operatorname{cc}})$-$\operatorname{Hom}(\overline{\mathcal{B}}_{\operatorname{cc}}, \overline{\mathcal{B}}_{\operatorname{cc}})$ bimodule structure on the morphism space $\operatorname{Hom}(\overline{\mathcal{B}}_{\operatorname{cc}}, \mathcal{B}_{\operatorname{cc}})$ for the brane-conjugate brane system $(\overline{\mathcal{B}}_{\operatorname{cc}}, \mathcal{B}_{\operatorname{cc}})$ on a complexification $(X, \Omega)$ of $(M, \omega)$. Here, $\mathcal{B}_{\operatorname{cc}}$ and $\overline{\mathcal{B}}_{\operatorname{cc}}$ are important examples of \emph{coisotropic A-branes} on $(X, \omega_X)$, where $\omega_X := \operatorname{Im} \Omega$.\par
Coisotropic A-branes were discovered by Kapustin-Orlov  \cite{KapOrl2003} as A-model boundary conditions of pseudoholomorphic maps from a string world-sheet to $X$. They argued that these objects should be in an enlargement of the Fukaya category of $(X, \omega_X)$ for the sake of mirror symmetry. While general morphism spaces and compositions among them remain mysterious in both physics and mathematics, \emph{brane quantization}, the quantization paradigm developed by Gukov-Witten in \cite{GukWit2009}, makes the study on them a fascinating subject. Attempts to describe these morphism spaces mathematically have been emerging in the literatures, e.g. \cite{AldZas2005, BisGua2022, GukKorNawPeiSab2023, KapLi2005, KapOrl2004, LeuYau2023, Qin2020, Qiu2023, Qui2012}.\par
A beautiful idea in brane quantization is the recovery of geometric quantization of $(M, \omega)$ and holomorphic deformation quantization of $(X, \Omega)$ by the morphism spaces $\operatorname{Hom}(\mathcal{B}, \mathcal{B}_{\operatorname{cc}})$ and $\operatorname{Hom}(\mathcal{B}_{\operatorname{cc}}, \mathcal{B}_{\operatorname{cc}})$ respectively, where $\mathcal{B}$ is a Lagrangian A-brane supported on $M$ and $\mathcal{B}_{\operatorname{cc}}$ is a canonical coisotropic A-brane. Chan-Leung-Li's work \cite{ChaLeuLi2023} can be interpreted as realizing the left action of $\operatorname{Hom}(\mathcal{B}_{\operatorname{cc}}, \mathcal{B}_{\operatorname{cc}})$ on $\operatorname{Hom}(\mathcal{B}, \mathcal{B}_{\operatorname{cc}})$, as long as one can extend the involved deformation quantization of $(M, \omega)$ to a holomorphic deformation quantization of $(X, \Omega)$. We will show in Section \ref{Section 3} that such an extension always exists on a neighbourhood of $M$ in $X$. The precise statement is as follows.

\begin{theorem}
	\label{Theorem 1.2}
	Every symplectic manifold $(M, \omega)$ admits a complexification $\iota: (M, \omega) \hookrightarrow (X, \Omega)$ which satisfies the following condition. For any deformation quantization $(\mathcal{C}^\infty(M, \mathbb{C})[[\hbar]], \star)$ of $(M, \omega)$, there exists a holomorphic deformation quantization $(\mathcal{A}, \psi)$ of $(X, \Omega)$ and a morphism of sheaves of $\mathbb{C}[[\hbar]]$-algebras $\tau: \mathcal{A} \to \iota_*( \mathcal{C}_M^\infty[[\hbar]], \star)$, such that the following diagram is commutative:
	\begin{center}
		\begin{tikzcd}
			\mathcal{A} \ar[r, "\tau"] \ar[d] & \iota_*\mathcal{C}_M^\infty[[\hbar]] \ar[d]\\
			\mathcal{O}_X \ar[r, "\iota^*"'] & \iota_*\mathcal{C}_M^\infty
		\end{tikzcd}
	\end{center}
	where $\mathcal{A} \to \mathcal{O}_X$ is the composition of $\psi: \mathcal{A}/\hbar \mathcal{A} \to \mathcal{O}_X$ with the canonical projection $\mathcal{A} \to \mathcal{A}/\hbar \mathcal{A}$, and $\iota_*\mathcal{C}_M^\infty[[\hbar]] \to \iota_*\mathcal{C}_M^\infty$ is the pushforward of the canonical projection $\mathcal{C}_M^\infty[[\hbar]] \to \mathcal{C}_M^\infty$ by $\iota$.
\end{theorem}

Note that Theorem \ref{Theorem 1.2} holds for \emph{arbitrary} symplectic manifolds. When $(M, \omega)$ is real analytic K\"ahler and $(\mathcal{C}^\infty(M, \mathbb{C})[[\hbar]], \star)$ is given by the modified Chan-Leung-Li's construction, we can choose $(X, \Omega)$ to admit a pair of holomorphic polarizations $\mathcal{P}^{\operatorname{c}}, \check{\mathcal{P}}^{\operatorname{c}}$ (in the sense of Definition \ref{Definition 4.4}) which are holomorphic extensions of $T^{1, 0}M$ and $T^{0, 1}M$ respectively, and choose $\mathcal{A}$ to be of the form $(\mathcal{O}_X[[\hbar]], \widetilde{\star})$, where $\widetilde{\star}$ is a holomorphic star product which is `with separation of variables' with respect to the decomposition $T^{1, 0}X = \mathcal{P}^{\operatorname{c}} \oplus \check{\mathcal{P}}^{\operatorname{c}}$. Details can be found in Subsection \ref{Subsection 5.2}.\par
An intriguing by-product of brane quantization is the morphism space $\operatorname{Hom}(\overline{\mathcal{B}}_{\operatorname{cc}}, \mathcal{B}_{\operatorname{cc}})$, which was studied in \cite{AldZas2005, EtiFreKaz2021, GaiWit2022, LeuYau2023}, as it is related to the analytic geometric Langlands program. In \cite{LeuYau2023}, the author of this paper and Leung proposed a definition of $\operatorname{Hom}(\overline{\mathcal{B}}_{\operatorname{cc}}, \mathcal{B}_{\operatorname{cc}})$ when $X$ is compact, due to an $\operatorname{Sp}(1)$-symmetry in hyperK\"ahler geometry. The main result of this paper, however, not only provides a mathematical description of the space $\operatorname{Hom}(\overline{\mathcal{B}}_{\operatorname{cc}}, \mathcal{B}_{\operatorname{cc}})$ itself in a generally non-compact case, but also verifies actions on it due to the categorical structure of coisotropic A-branes.\par
Indeed, the above realizations only apply to a sufficiently small complexification $(X, \Omega)$ of $(M, \omega)$ in general. Thus, we only focus on local models of the data $(X, \Omega)$, $\mathcal{B}$, $\mathcal{B}_{\operatorname{cc}}$ and $\overline{\mathcal{B}}_{\operatorname{cc}}$ near $M$ for brane quantization. We will also verify in Section \ref{Section 4} that real analyticity of $M$ and its prequantum line bundle $L$ implies the uniqueness of such a local model. More precisely, we call the quadruple $(X, \Omega, \iota, \widehat{L})$ a \emph{brane extension} of $(M, \omega, L)$, where $\widehat{L}$ is the Chan-Paton bundle of $\mathcal{B}_{\operatorname{cc}}$, and we have the following result.

\begin{theorem}
	\label{Theorem 1.3}
	Let $L$ be a real analytic prequantum line bundle of $(M, \omega)$. Then $(M, \omega, L)$ admits a brane extension $(X, \Omega, \iota, \widehat{L})$. Moreover, if $(X', \Omega', \iota', \widehat{L}')$ is another brane extension of $(M, \omega, L)$, then there exists a diffeomorphism $\phi: U \to U'$ from a neighbourhood $U$ of $\iota(M)$ in $X$ onto a neighbourhood $U'$ of $\iota'(M)$ in $X'$ and an isomorphism
	\begin{equation*}
		\Phi: \widehat{L} \vert_U \to \phi^*(\widehat{L}' \vert_{U'})
	\end{equation*}
	of Hermitian line bundles with unitary connections over $U$ such that $\phi \circ \iota = \iota'$, $\phi^*(\Omega' \vert_{U'}) = \Omega \vert_U$ and $\Phi$ restricts to the identity map $\operatorname{Id}_L$.
\end{theorem}

For the local model, our work suggests a correspondence between certain morphism spaces involving coisotropic A-branes $\mathcal{B}, \mathcal{B}_{\operatorname{cc}}, \overline{\mathcal{B}}_{\operatorname{cc}}$ and sheaves on $M$ with algebraic structures which give the corresponding mathematical realizations. For the reader's convenience, we list out this correspondence in the following table.

\begin{center}
	\begin{tabular}{|c|c|c|c|}
		\hline
			Morphism space & Sheaf & Algebraic structure & Description of sections of the sheaf\\
		\hline
			$\operatorname{Hom}(\mathcal{B}_{\operatorname{cc}}, \mathcal{B}_{\operatorname{cc}})$ & $\mathcal{O}_{\operatorname{qu}}^{(k)}$ & Algebra & Quantizable functions of level $k$ on $(M, \omega)$\\
			$\operatorname{Hom}(\overline{\mathcal{B}}_{\operatorname{cc}}, \overline{\mathcal{B}}_{\operatorname{cc}})$ & $\overline{\mathcal{O}}_{\operatorname{qu}}^{(k)}$ & Algebra & Quantizable functions of level $k$ on $(\overline{M}, -\omega)$\\
			$\operatorname{Hom}(\mathcal{B}, \mathcal{B}_{\operatorname{cc}})$ & $\mathcal{L}^{(k)}$ & Left module & Holomorphic sections of $L^{\otimes k} \otimes \sqrt{K}$\\
			$\operatorname{Hom}(\mathcal{B}, \overline{\mathcal{B}}_{\operatorname{cc}})$ &  $\overline{\mathcal{L}}^{(k)}$ & Left module & Antiholomorphic sections of $(L^\vee)^{\otimes k} \otimes \sqrt{K}^\vee$\\
			$\operatorname{Hom}(\mathcal{B}_{\operatorname{cc}}, \mathcal{B})$ & $\mathcal{L}^{(-k)}$ & Right module & Holomorphic sections of $(L^\vee)^{\otimes k} \otimes \sqrt{K}$\\
			$\operatorname{Hom}(\overline{\mathcal{B}}_{\operatorname{cc}}, \mathcal{B})$ & $\overline{\mathcal{L}}^{(-k)}$ & Right module & Antiholomorphic sections of $L^{\otimes k} \otimes \sqrt{K}^\vee$\\
			$\operatorname{Hom}(\overline{\mathcal{B}}_{\operatorname{cc}}, \mathcal{B}_{\operatorname{cc}})$ & $\mathcal{L}_{\operatorname{sm}}^{(2k)}$ & Bimodule & Smooth sections of $L^{\otimes 2k}$\\
		\hline
	\end{tabular}
\end{center}

As a final remark, bimodules over deformation quantizations were studied in \cite{BurWal2004, NeuWal2003, Wal2002}, etc., so as to understand Morita equivalence of star products. However, the $\mathcal{O}_{\operatorname{qu}}^{(k)}$-$\overline{\mathcal{O}}_{\operatorname{qu}}^{(k)}$ bimodule $\mathcal{L}_{\operatorname{sm}}^{(2k)}$ constructed in this paper does not describe a Morita equivalence bimodule in the context of deformation quantization. See Subsection \ref{Subsection 5.4} and Remark \ref{Remark 5.24} for more details.\par
The proof of Theorem \ref{Theorem 1.1} relies on Fedosov type gluing arguments. Our construction of the desired bimodule structure on $\mathcal{L}_{\operatorname{sm}}^{(k)}$ is based on Chan-Leung-Li's construction \cite{ChaLeuLi2023} of the left module structure on $\mathcal{L}^{(k)}$. We hope that, in the future, techniques involved in our proof will be applied to study other A-model morphism spaces, especially $\operatorname{Hom}(\mathcal{B}, \mathcal{B}_{\operatorname{cc}})$ when $\omega = \operatorname{Re} \Omega \vert_M$ is degenerate.\par
The paper is organized as follows. Section \ref{Section 2} is an exposition of the physical background behind the (mathematically conjectural) category of coisotropic A-branes. In Section \ref{Section 3}, we will examine the relationship between deformation quantizations of a general symplectic manifold $(M, \omega)$ and holomorphic deformation quantizations of its complexifications $(X, \Omega)$, and complete the proof of Theorem \ref{Theorem 1.2}. In Section \ref{Section 4}, we will examine the relationship between geometric quantization of $(M, \omega)$ and local models of the data for brane quantization near $M$, and we will prove Theorem \ref{Theorem 1.3}. Finally, in Section \ref{Section 5}, we will consider the case when $(M, \omega)$ is K\"ahler, construct sheaves of algebras and (bi)modules appeared in the table above, and prove our main result (Theorem \ref{Theorem 1.1}).

\section{Coisotropic A-branes and brane quantization}
\label{Section 2}
This section is devoted to the physical background motivating our work. Readers who are only interested in the mathematical content of this paper can skip over this section. In Subsection \ref{Subsection 2.1}, we will give a mathematical definition of coisotropic A-branes. In the rest of the subsections, we will review physical expectations on certain A-model morphism spaces in the literature, including $\operatorname{Hom}(\mathcal{B}_{\operatorname{cc}}, \mathcal{B}_{\operatorname{cc}})$, $\operatorname{Hom}(\mathcal{B}, \mathcal{B}_{\operatorname{cc}})$, $\operatorname{Hom}(\mathcal{B}_{\operatorname{cc}}, \mathcal{B})$ and $\operatorname{Hom}(\overline{\mathcal{B}}_{\operatorname{cc}}, \mathcal{B}_{\operatorname{cc}})$. These expectations mostly originate from Gukov-Witten brane quantization \cite{GukWit2009}. They are guiding principles for construction of sheaves of algebras and their (bi)modules in Section \ref{Section 5} so as to realize the above morphism spaces.

\subsection{Coisotropic A-branes}
\label{Subsection 2.1}
\quad\par
Recall that for a coisotropic submanifold $S$ of a symplectic manifold $(X, \omega_X)$, $\omega_X \vert_S$ descends to a section $\tilde{\omega}_X \in \Gamma(S, \bigwedge^2 \mathcal{F}^*S)$, where $\mathcal{F}S := TS/\mathcal{L}S$ and $\mathcal{L}S$ is the characteristic foliation of $S$.

\begin{definition}
	A \emph{coisotropic A-brane} (or \emph{A-brane} in short) on a symplectic manifold $(X, \omega_X)$ is a pair $\mathcal{B} = (S, L)$, where $S$ is a coisotropic submanifold of $X$ and $L$ is a Hermitian line bundle over $S$ with a unitary connection $\nabla$ of curvature $F_L$ such that
	\begin{enumerate}
		\item $F := \sqrt{-1} F_L \in \Omega^2(S)$ descends to a smooth section $\tilde{F} \in \Gamma(S, \bigwedge^2 \mathcal{F}^*S)$; and
		\item the composition $I := \tilde{\omega}_X^{-1} \circ \tilde{F}: \mathcal{F}S \to \mathcal{F}S$ defines an integrable almost complex structure on $\mathcal{F}S$.
	\end{enumerate}
	In this case, we call $S$ the \emph{support} of $\mathcal{B}$ and $L$ the \emph{Chan-Paton bundle} of $\mathcal{B}$. 
\end{definition}

There are two extreme cases for A-branes. An A-brane $\mathcal{B} = (S, L)$ on $(X, \omega_X)$ is said to be a \emph{Lagrangian A-brane} if $S$ is $\omega_X$-Lagrangian, in which case $L$ is flat; it is said to be \emph{space-filling} if $S = X$, in which case $X$ admits a holomorphic symplectic structure $F + \sqrt{-1}\omega_X$.\par
If $X$ is equipped with a holomorphic symplectic structure $\Omega$ such that $\omega_X = \operatorname{Im} \Omega$, then we call a space-filling A-brane on $(X, \omega_X)$ with Chan-Paton bundle of curvature $\tfrac{1}{\sqrt{-1}} \operatorname{Re} \Omega$ a \emph{canonical coisotropic A-brane} on $(X, \Omega)$ and denote such an A-brane by $\mathcal{B}_{\operatorname{cc}}$. It is unique up to a tensor product with a Hermitian line bundle equipped with a flat unitary connection.\par
Whenever $\mathcal{B}_{\operatorname{cc}} = (X, \widehat{L})$ exists, there is another space-filling A-brane $\overline{\mathcal{B}}_{\operatorname{cc}} = (X, \widehat{L}^\vee)$ on $(X, \omega_X)$. It is called the \emph{conjugate brane} of $\mathcal{B}_{\operatorname{cc}}$, and is indeed a canonical coisotropic A-brane on $(\overline{X}, -\overline{\Omega})$, where $\overline{X}$ is the conjugate complex manifold of $X$. 

\subsection{$\operatorname{Hom}(\mathcal{B}_{\operatorname{cc}}, \mathcal{B}_{\operatorname{cc}})$ and holomorphic deformation quantization of $(X, \Omega)$}
\label{Subsection 2.2}
\quad\par
The notion of a canonical coisotropic A-brane was first introduced by Kapustin-Witten \cite{KapWit2007} when they studied Hitchin moduli spaces in the Geometric Langlands Program. Here, we are concerned about the endomorphism algebra of a canonical coisotropic A-brane $\mathcal{B}_{\operatorname{cc}} = (X, \widehat{L})$ on a holomorphic symplectic manifold $(X, \Omega)$. Indeed, the endomorphism algebra for an arbitrary coisotropic A-brane has been determined in the classical limit in \cite{KapLi2005, KapOrl2003}, regarded as the BRST cohomology of the state space for topological open strings with both endpoints on the same A-brane. In particular, the classical limit of $\operatorname{Hom}(\mathcal{B}_{\operatorname{cc}}, \mathcal{B}_{\operatorname{cc}})$ is given by $H^*(X, \mathcal{O}_X)$, where $\mathcal{O}_X$ is the sheaf of holomorphic functions on $X$.\par
Later in \cite{GukWit2009}, Gukov-Witten considered it as the non-commutative algebra of quantum boundary observables for $(\mathcal{B}_{\operatorname{cc}}, \mathcal{B}_{\operatorname{cc}})$-strings, and said that it should be a holomorphic deformation quantization of $(X, \Omega)$. In recent years, Gaiotto-Witten \cite{GaiWit2022} have adopted a sheaf-theoretic approach to re-examine this algebra. Following this approach, we view $\operatorname{Hom}(\mathcal{B}_{\operatorname{cc}}, \mathcal{B}_{\operatorname{cc}})$ as a sheaf of algebras on $(X, \Omega)$.\par
Two problems arise from Gukov-Witten's suggestion. The first problem is that deformation quantization is formal, while a priori $\operatorname{Hom}(\mathcal{B}_{\operatorname{cc}}, \mathcal{B}_{\operatorname{cc}})$ is not. It is generally impossible to yield a non-formal non-commutative product on holomorphic functions on $X$ by the evaluation of the formal parameter $\hbar$ in holomorphic deformation quantization to a complex number. We will only aim at a dense subspace of holomorphic functions on $X$ equipped with such a non-formal product.\par
Indeed, a non-formal parameter $k$ can be introduced to A-models. For $k \in \mathbb{Z}^+$, $\mathcal{B}_{\operatorname{cc}}^{(k)} = (X, \widehat{L}^{\otimes k})$ is a canonical coisotropic A-brane on $(X, k\Omega)$. Our convention on $\hbar$ is that a corresponding non-formal product on a desired subspace of holomorphic functions is obtained by setting $\hbar = \tfrac{\sqrt{-1}}{k}$. For simplicity, we will focus on the case when $k = 1$ in this section.\par
The second problem is that there are different holomorphic deformation quantizations of $(X, \Omega)$, even up to equivalence. One way to identify an appropriate candidate is to study its representations. This will be discussed immediately in the next subsection.

\subsection{$\operatorname{Hom}(\mathcal{B}, \mathcal{B}_{\operatorname{cc}})$ and geometric quantization of $(M, \omega)$}
\label{Subsection 2.3}
\quad\par
Let $\omega_X = \operatorname{Im} \Omega$. Now we assume that $(X, \omega_X)$ admits a Lagrangian A-brane $\mathcal{B} = (M, L_0)$ such that $L_0$ is trivial, $M$ is spin and $\omega := \operatorname{Re} \Omega \vert_M$ is a symplectic form. The composition
\begin{equation}
	\label{Equation 2.1}
	\operatorname{Hom}(\mathcal{B}_{\operatorname{cc}}, \mathcal{B}_{\operatorname{cc}}) \otimes \operatorname{Hom}(\mathcal{B}, \mathcal{B}_{\operatorname{cc}}) \to \operatorname{Hom}(\mathcal{B}, \mathcal{B}_{\operatorname{cc}}).
\end{equation}
indicates a left $\operatorname{Hom}(\mathcal{B}_{\operatorname{cc}}, \mathcal{B}_{\operatorname{cc}})$-module structure on $\operatorname{Hom}(\mathcal{B}, \mathcal{B}_{\operatorname{cc}})$.\par
Not much is known about this representation, except when either $M$ is compact and $(X, \Omega)$ admits a hyperK\"ahler metric \cite{GukWit2009}, or $(M, \omega)$ is compact K\"ahler and $X$ is equipped with a holomorphic Lagrangian fibration onto $M$ so that all fibres are contractible \cite{GaiWit2022}. In both cases, it is suggested that $\operatorname{Hom}(\mathcal{B}, \mathcal{B}_{\operatorname{cc}})$ should be isomorphic to $H^0(M, L \otimes \sqrt{K})$ \footnote{As mentioned in \cite{GaiWit2022, GukWit2009}, it is not expected that A-model morphism spaces are in general Hilbert spaces. The A-model interpretation of the inner product on the quantum Hilbert space on $(M, \omega)$ requires extra structures on $(X, \Omega)$, but this is out of the scope of this paper.}, where $L := \widehat{L} \vert_M$ and $\sqrt{K}$ is a chosen square root of the canonical bundle of $M$. In the latter case, it is also suggested that $\operatorname{Hom}(\mathcal{B}_{\operatorname{cc}}, \mathcal{B}_{\operatorname{cc}})$ should be realized as the sheaf $\mathcal{D}_{L \otimes \sqrt{K}}$ of holomorphic differential operators of $L \otimes \sqrt{K}$. Observations on such a relation between A-branes and (twisted) $\mathcal{D}$-modules from examples date back to \cite{GukWit2009, KapWit2007}.\par
Therefore, if we take the sheaf-theoretic approach, then $\operatorname{Hom}(\mathcal{B}, \mathcal{B}_{\operatorname{cc}})$ should be realized as a sheaf on $X$ which restricts to the sheaf of holomorphic sections of $L \otimes \sqrt{K}$ such that the sheaf cohomologies for these two sheaves are isomorphic.

\begin{remark}
	Geometric quantization of $(M, \omega)$ and the A-model on $(X, \omega_X)$ are compatible under scaling by the non-formal parameter $k$, i.e. for any $k \in \mathbb{Z}^+$, $\mathcal{B}^{(k)} = (M, L_0^{\otimes k})$ and $\mathcal{B}_{\operatorname{cc}}^{(k)} = (X, \widehat{L}^{\otimes k})$ are A-branes on $(X, k\omega_X)$ and $L^{\otimes k} = \widehat{L}^{\otimes k} \vert_M$ is a prequantum line bundle of $(M, k\omega)$.
\end{remark}

The above suggestion is in agreement with the recent works of Chan-Leung-Li \cite{ChaLeuLi2023} and Bischoff-Gualtieri \cite{BisGua2022} from different aspects. First, the sheaf $\mathcal{O}_{\operatorname{qu}}^{(1)}$ of quantizable functions (of level $1$) on $M$ as mentioned in Section \ref{Section 1} is isomorphic to $\mathcal{D}_{L \otimes \sqrt{K}}$. Strictly speaking, we need an argument showing that $\mathcal{O}_{\operatorname{qu}}^{(1)}$ is the restriction of a subsheaf of a holomorphic deformation quantization of $(X, \Omega)$ after the evaluation at $\frac{\hbar}{\sqrt{-1}} = 1$. In Sections \ref{Section 3} and \ref{Section 5}, we will discuss this issue in detail.\par
Second, Bischoff-Gualtieri's idea is that one should take a holomorphic fibration of $X$ and define $\operatorname{Hom}(\mathcal{B}, \mathcal{B}_{\operatorname{cc}})$ as the direct sum, over the set of Bohr-Sommerfeld fibres intersecting $M$, of covariantly constant holomorphic sections of a canonical modification of $\widehat{L}$ restricted to the corresponding Bohr-Sommerfeld fibres. We will have a further discussion on it in Section \ref{Section 4}, where we will show that, assuming real analyticity of $M$ and $L$, the local model near $M$ for brane quantization is unique.

\subsection{$\operatorname{Hom}(\mathcal{B}_{\operatorname{cc}}, \mathcal{B})$ as the dual of $\operatorname{Hom}(\mathcal{B}, \mathcal{B}_{\operatorname{cc}})$}
\label{Subsection 2.4}
\quad\par
A physical consideration, recalled by Gaiotto-Witten in Section 2.7 in \cite{GaiWit2022} for instance, suggests that for any A-branes $\mathcal{B}_1, \mathcal{B}_2$ on $(X, \omega_X)$, no matter whether they are Lagrangian or not, 
there should be a non-degenerate pairing between $\operatorname{Hom}(\mathcal{B}_1, \mathcal{B}_2)$ and $\operatorname{Hom}(\mathcal{B}_2, \mathcal{B}_1)$ under appropriate compactness assumptions on $\mathcal{B}_1, \mathcal{B}_2$ or $X$. This feature is known as a Calabi-Yau ($A_\infty$-)structure in mathematics.\par
In particular, $\operatorname{Hom}(\mathcal{B}_{\operatorname{cc}}, \mathcal{B})$ should be the dual of $\operatorname{Hom}(\mathcal{B}, \mathcal{B}_{\operatorname{cc}})$ when $(M, \omega)$ is compact K\"ahler. Due to Serre duality, the space $\operatorname{Hom}(\mathcal{B}_{\operatorname{cc}}, \mathcal{B})$ should be $H^*(M, L^\vee \otimes \sqrt{K})$. The following composition
\begin{equation}
	\label{Equation 2.2}
	\operatorname{Hom}(\mathcal{B}_{\operatorname{cc}}, \mathcal{B}) \otimes \operatorname{Hom}(\mathcal{B}_{\operatorname{cc}}, \mathcal{B}_{\operatorname{cc}}) \to \operatorname{Hom}(\mathcal{B}_{\operatorname{cc}}, \mathcal{B})
\end{equation}
indicates that $\operatorname{Hom}(\mathcal{B}_{\operatorname{cc}}, \mathcal{B})$ is a right $\operatorname{Hom}(\mathcal{B}_{\operatorname{cc}}, \mathcal{B}_{\operatorname{cc}})$-module. This is consistent with the basic theory of twisted $\mathcal{D}$-modules, as the opposite algebra of $\mathcal{D}_{L \otimes \sqrt{K}}$ is isomorphic to the sheaf of holomorphic differential operators of $L^\vee \otimes \sqrt{K}$.\par
Taking the sheaf-theoretic approach again, we should view $\operatorname{Hom}(\mathcal{B}_{\operatorname{cc}}, \mathcal{B})$ as a sheaf on $X$ which restricts to the sheaf of holomorphic sections of $L^\vee \otimes \sqrt{K}$ such that the sheaf cohomologies for these two sheaves are isomorphic.

\subsection{The brane-conjugate brane morphism space $\operatorname{Hom}(\overline{\mathcal{B}}_{\operatorname{cc}}, \mathcal{B}_{\operatorname{cc}})$}
\label{Subsection 2.5}
\quad\par
For the conjugate brane $\overline{\mathcal{B}}_{\operatorname{cc}} = (X, \widehat{L}^\vee)$, we should similarly view $\operatorname{Hom}(\overline{\mathcal{B}}_{\operatorname{cc}}, \overline{\mathcal{B}}_{\operatorname{cc}})$ as a holomorphic deformation quantization of $(\overline{X}, -\overline{\Omega})$ restricting to a deformation quantization of $(M, -\omega)$.\par
In Section 5 in \cite{GaiWit2022}, Gaiotto-Witten gave a physical argument that $\operatorname{Hom}(\overline{\mathcal{B}}_{\operatorname{cc}}, \mathcal{B}_{\operatorname{cc}})$ should be given by the underlying vector space of the quantum Hilbert space of $(X, 2\operatorname{Re}\Omega)$ with prequantum line bundle $\widehat{L}^{\otimes 2}$. A holomorphic Lagrangian fibration $\pi: X \to M$ as before gives a real polarization on $(X, 2\operatorname{Re}\Omega)$. In this case, we can view $\operatorname{Hom}(\overline{\mathcal{B}}_{\operatorname{cc}}, \mathcal{B}_{\operatorname{cc}})$ as the sheaf of smooth sections of $\widehat{L}^{\otimes 2}$ covariantly constant along the fibres of $\pi$, restricting to the sheaf of smooth sections of $L^{\otimes 2}$.\par
The compositions
\begin{align}
	\label{Equation 2.3}
	\operatorname{Hom}(\mathcal{B}_{\operatorname{cc}}, \mathcal{B}_{\operatorname{cc}}) \otimes \operatorname{Hom}(\overline{\mathcal{B}}_{\operatorname{cc}}, \mathcal{B}_{\operatorname{cc}}) & \to \operatorname{Hom}(\overline{\mathcal{B}}_{\operatorname{cc}}, \mathcal{B}_{\operatorname{cc}}),\\
	\label{Equation 2.4}
	\operatorname{Hom}(\overline{\mathcal{B}}_{\operatorname{cc}}, \mathcal{B}_{\operatorname{cc}}) \otimes \operatorname{Hom}(\overline{\mathcal{B}}_{\operatorname{cc}}, \overline{\mathcal{B}}_{\operatorname{cc}}) & \to \operatorname{Hom}(\overline{\mathcal{B}}_{\operatorname{cc}}, \mathcal{B}_{\operatorname{cc}}).
\end{align}
indicate that $\operatorname{Hom}(\overline{\mathcal{B}}_{\operatorname{cc}}, \mathcal{B}_{\operatorname{cc}})$ is a $\operatorname{Hom}(\mathcal{B}_{\operatorname{cc}}, \mathcal{B}_{\operatorname{cc}})$-$\operatorname{Hom}(\overline{\mathcal{B}}_{\operatorname{cc}}, \overline{\mathcal{B}}_{\operatorname{cc}})$ bimodule. Moreover, when we have the Lagrangian A-brane $\mathcal{B}$ given as in Subsection \ref{Subsection 2.3}, this bimodule structure should be compatible with the module structures on $\operatorname{Hom}( \mathcal{B}, \mathcal{B}_{\operatorname{cc}} )$ and $\operatorname{Hom}( \overline{\mathcal{B}}_{\operatorname{cc}}, \mathcal{B} )$. More precisely, the following two diagrams should be commutative due to the associative property of compositions:
\begin{center}
	\begin{tikzcd}
		\operatorname{Hom}( \mathcal{B}_{\operatorname{cc}}, \mathcal{B}_{\operatorname{cc}} ) \otimes \operatorname{Hom}( \mathcal{B}, \mathcal{B}_{\operatorname{cc}} ) \otimes \operatorname{Hom}( \overline{\mathcal{B}}_{\operatorname{cc}}, \mathcal{B} ) \ar[r] \ar[d] & \operatorname{Hom}( \mathcal{B}_{\operatorname{cc}}, \mathcal{B}_{\operatorname{cc}} ) \otimes \operatorname{Hom}( \overline{\mathcal{B}}_{\operatorname{cc}}, \mathcal{B}_{\operatorname{cc}} ) \ar[d]\\
		\operatorname{Hom}( \mathcal{B}, \mathcal{B}_{\operatorname{cc}} ) \otimes \operatorname{Hom}( \overline{\mathcal{B}}_{\operatorname{cc}}, \mathcal{B} ) \ar[r] & \operatorname{Hom}( \overline{\mathcal{B}}_{\operatorname{cc}}, \mathcal{B}_{\operatorname{cc}} )
	\end{tikzcd}
\end{center}
and 
\begin{center}
	\begin{tikzcd}
		\operatorname{Hom}( \mathcal{B}, \mathcal{B}_{\operatorname{cc}} ) \otimes \operatorname{Hom}( \overline{\mathcal{B}}_{\operatorname{cc}}, \mathcal{B} ) \otimes \operatorname{Hom}( \overline{\mathcal{B}}_{\operatorname{cc}}, \overline{\mathcal{B}}_{\operatorname{cc}} ) \ar[r] \ar[d] & \operatorname{Hom}( \mathcal{B}, \mathcal{B}_{\operatorname{cc}} ) \otimes \operatorname{Hom}( \overline{\mathcal{B}}_{\operatorname{cc}}, \mathcal{B} ) \ar[d]\\
		\operatorname{Hom}( \overline{\mathcal{B}}_{\operatorname{cc}}, \mathcal{B}_{\operatorname{cc}} ) \otimes \operatorname{Hom}( \overline{\mathcal{B}}_{\operatorname{cc}}, \overline{\mathcal{B}}_{\operatorname{cc}} ) \ar[r] & \operatorname{Hom}( \overline{\mathcal{B}}_{\operatorname{cc}}, \mathcal{B}_{\operatorname{cc}} )
	\end{tikzcd}
\end{center}
When $(M, \omega)$ is compact real analytic K\"ahler and $X$ is shrunk to a sufficiently small neighbourhood of $M$, Theorem \ref{Theorem 1.1} realizes the above physical expectations on $\operatorname{Hom}( \overline{\mathcal{B}}_{\operatorname{cc}}, \mathcal{B}_{\operatorname{cc}} )$.

\section{(Holomorphic) deformation quantizations of $(M, \omega)$ and its complexifications}
\label{Section 3}
We will start with a precise definition of a complexification $(X, \Omega)$ of a symplectic manifold $(M, \omega)$ in Subsection \ref{Subsection 3.1}. Then in Subsection \ref{Subsection 3.2}, we will show that every equivalence class of deformation quantizations of $(M, \omega)$ has a real analytic representative. In Subsection \ref{Subsection 3.3}, we will discuss holomorphic deformation quantizations of a Stein holomorphic symplectic manifold. After these preparations, in Subsection \ref{Subsection 3.4}, we will examine the possibility of recovering a holomorphic deformation quantization of $(X, \Omega)$ from a deformation quantization of $(M, \omega)$ as an extension.

\subsection{Complexifications of a symplectic manifold}
\label{Subsection 3.1}
\quad\par
By Kutzschebauch-Loose's Theorem \cite{KutLoo2000}, any symplectic manifold $(M, \omega)$ possesses a unique (up to isomorphism) real analytic structure such that $\omega$ is real analytic. Throughout the rest of this paper, we always equip $(M, \omega)$ with this canonical real analytic structure.

\begin{definition}
	\label{Definition 3.1}
	A \emph{complexification} \footnote{The notion of complexification in \cite{GukWit2009} is more restrictive than our definition - its requires an antiholomorphic involution $\tau$ of $X$ such that $M$ is the fixed point set of $\tau$ and $\tau^*\Omega = \overline{\Omega}$. However, shrinking a complexification $(X, \Omega)$ of $(M, \omega)$ in our definition if necessary, we can assume $X$ admits such a map $\tau$.} of a symplectic manifold $(M, \omega)$ is a complexification $X$ of the underlying real analytic manifold $M$ together with a holomorphic symplectic form $\Omega$ on $X$ which is a holomorphic extension of $\omega$, i.e. $\Omega \vert_M = \omega$.
\end{definition}

Every symplectic manifold $(M, \omega)$ has a Stein complexification and the germ of complexifications of $(M, \omega)$ is unique (see Appendix \ref{Appendix A}). Shrinking $X$ if necessary, we can always find a symplectic embedding of $(X, \operatorname{Im}\Omega)$ into a neighbourhood of the zero section of the cotangent bundle $T^\vee M$ such that $M$ is identified with the zero section. It implies that we can assume $\operatorname{Im}\Omega = d\alpha$ for some $\alpha \in \Omega^1(X)$ with $\iota^*\alpha = 0$ (hence $[\Omega] = [\operatorname{Re}\Omega]$) and $X$ has a smooth fibration $\pi: X \to M$ such that $\pi \circ \iota = \operatorname{Id}_M$ and the fibres of $\pi$ are $\operatorname{Im} \Omega$-Lagrangian.

\begin{remark}
	Let $(X, \Omega)$ be a holomorphic symplectic manifold. Suppose $X$ has an $\operatorname{Im}\Omega$-Lagrangian submanifold $M$ such that $\omega := \operatorname{Re} \Omega \vert_M$ is non-degenerate. Then $(X, \Omega)$ is a complexification of $(M, \omega)$ if and only if $M$ is a real analytic submanifold of $(X, \Omega)$.
\end{remark}

\begin{example}
	\label{Example 3.3}
	Suppose $(M, \omega)$ is a real analytic K\"ahler manifold, i.e. the K\"ahler form $\omega$ is real analytic with respect to the real analytic structure determined by the compatible complex structure on $M$. There are several constructions of a complexification $(X, \Omega)$ of $(M, \omega)$ as follows.
	\begin{enumerate}
		\item Constructing $\Omega$ by a holomorphic extension of $\omega$ on a neighbourhood $X$ of the diagonal in the complex manifold $M \times \overline{M}$, where $\overline{M}$ is the conjugate complex manifold of $M$ \cite{BhaBurLupUri2022, Rua1998}.
		\item Constructing $\Omega$ by deforming the natural holomorphic symplectic form on the holomorphic cotangent bundle $X = T^\vee M$ by the pullback of $\omega$ along the projection $X \to M$ \cite{BisGua2022, BisGuaZab2022, Don2002}.
		\item Defining $\Omega = \omega_J + \sqrt{-1}\omega_K$ on a neighbourhood $X$ of the zero section in $T^\vee M$, which admits a hyperK\"ahler structure $(g, I, J, K)$ such that the restriction of the K\"ahler form $\omega_J$ associated to $J$ on the zero section is $\omega$ and the K\"ahler form $\omega_K$ associated to $K$ is the restriction of the natural symplectic form on $T^\vee M$  \cite{Fei2001, Kal1999, VerKal1999}.
	\end{enumerate}
\end{example}

\subsection{Deformation quantizations of a symplectic manifold}
\label{Subsection 3.2}
\quad\par
Consider a symplectic manifold $(M, \omega)$. Let $\{ \quad, \quad \}$ be the Poisson bracket on $\mathcal{C}^\infty(M, \mathbb{C})$ given by $\{f, g\} = \omega(X_g, X_f)$, where $X_f$ is the Hamiltonian vector field of $f$, i.e. $\iota_{X_f}\omega = df$, and let $\hbar$ be an indeterminate over $\mathbb{C}$. A \emph{star product} on $M$ is a $\mathbb{C}[[\hbar]]$-bilinear map
\begin{equation*}
	\star: \mathcal{C}^\infty(M, \mathbb{C})[[\hbar]] \times \mathcal{C}^\infty(M, \mathbb{C})[[\hbar]] \to \mathcal{C}^\infty(M, \mathbb{C})[[\hbar]],
\end{equation*}
such that $(\mathcal{C}^\infty(M, \mathbb{C})[[\hbar]], \star)$ forms a $\mathbb{C}[[\hbar]]$-algebra with unit $1$ and there is a sequence $\{C_r\}_{r=1}^\infty$ of bi-differential operators on $M$ such that for all $f, g \in \mathcal{C}^\infty(M, \mathbb{C})$, $f \star g = fg + \sum_{r=1}^\infty \hbar^r C_r(f, g)$. Since $C_r$'s are all local operators, $\star$ can also be regarded as a $\mathbb{C}[[\hbar]]$-bilinear sheaf morphism $\mathcal{C}_M^\infty[[\hbar]] \times \mathcal{C}_M^\infty[[\hbar]] \to \mathcal{C}_M^\infty[[\hbar]]$, where $\mathcal{C}_M^\infty$ is the sheaf of smooth complex valued functions on $M$. A \emph{deformation quantization} of $(M, \omega)$ is a $\mathbb{C}[[\hbar]]$-algebra of the form $(\mathcal{C}^\infty(M, \mathbb{C})[[\hbar]], \star)$, where $\star$ is a star product on $M$ such that $C_1(f, g) - C_1(g, f) = \{f, g\}$ for all $f, g \in \mathcal{C}^\infty(M, \mathbb{C})$.\par
As mentioned in Subsection \ref{Subsection 3.1}, $M$ has a real analytic structure for which $\omega$ is real analytic. One can thus consider a deformation quantization $(\mathcal{C}^\infty(M, \mathbb{C})[[\hbar]], \star)$ of $(M, \omega)$ which is \emph{real analytic} in the sense that $\star$ restricts to a $\mathbb{C}[[\hbar]]$-bilinear sheaf morphism $\mathcal{C}_M^\omega[[\hbar]] \times \mathcal{C}_M^\omega[[\hbar]] \to \mathcal{C}_M^\omega[[\hbar]]$, where $\mathcal{C}_M^\omega$ is the sheaf of real analytic complex valued functions on $M$. The goal of this subsection is to prove the following theorem.

\begin{theorem}
	\label{Theorem 3.4}
	Every equivalence class of deformation quantizations of a symplectic manifold $(M, \omega)$ has a representative $(\mathcal{C}^\infty(M, \mathbb{C})[[\hbar]], \star)$ which is real analytic.
\end{theorem}

To prove Theorem \ref{Theorem 3.4}, we need to introduce Fedosov quantization \cite{Fed1994} of $(M, \omega)$. It is a geometric construction of a star product on $M$ via its Weyl bundle.\par
The \emph{Weyl bundle} of $(M, \omega)$ is the infinite rank vector bundle $\mathcal{W} = \widehat{\operatorname{Sym}} T^\vee M_\mathbb{C}[[\hbar]]$, where $\widehat{\operatorname{Sym}} T^\vee M_\mathbb{C}$ is the completed symmetric algebra bundle on the complexified cotangent bundle $T^\vee M_\mathbb{C}$. In local real analytic coordinates $(x^1, ..., x^{2n})$, a smooth (resp. real analytic) section of $\mathcal{W}$ is given by a formal power series
\begin{equation*}
	\sum_{r, l \geq 0} \sum_{i_1, ..., i_l \geq 0} \hbar^r a_{r, i_1, ..., i_l} y^{i_1} \cdots y^{i_l},
\end{equation*}
where $a_{r, i_1, ..., i_l}$ are local smooth (resp. real analytic) complex valued functions, $y^i$ denotes the covector $dx^i$ regarded as a section of $\mathcal{W}$, and we suppress the notations of symmetric products in the expression $y^{i_1} \cdots y^{i_l}$. We assign weights on $\mathcal{W}$ by setting the weight of $y^i$ to be $1$ and the weight of $\hbar$ to be $2$. For $r \in \mathbb{N}$, we denote by $\mathcal{W}_{(r)}$ the subbundle of $\mathcal{W}$ of weight at least $r$. We then have a decreasing filtration
\begin{align*}
	\mathcal{W} = \mathcal{W}_{(0)} \supset \mathcal{W}_{(1)} \supset \cdots \supset  \mathcal{W}_{(r)} \supset \cdots.
\end{align*}
There are three $\mathcal{C}^\infty(M, \mathbb{C})[[\hbar]]$-linear operators $\delta, \delta^{-1}, \pi_0$ on $\Omega^*(M, \mathcal{W})$ defined as follows: for a local section $a = y^{i_1} \cdots y^{i_l} dx^{j_1} \wedge \cdots \wedge dx^{j_m}$,
\begin{align*}
	\delta a = dx^k \wedge \frac{\partial a}{\partial y^k}, \quad \delta^{-1} a =
	\begin{cases}
		\frac{1}{l+m} y^k \iota_{\partial_{x^k}} a & \text{ if } l + m > 0;\\
		0 & \text{ if } l + m = 0,
	\end{cases}
	\quad \pi_0 (a) =
	\begin{cases}
		0  & \text{ if } l + m > 0;\\
		a & \text{ if } l + m = 0,
	\end{cases}
\end{align*}
and the equality $\operatorname{Id} - \pi_0 = \delta \circ \delta^{-1} + \delta^{-1} \circ \delta$ holds on $\Omega^*(M, \mathcal{W})$.\par
Write $\omega = \tfrac{1}{2} \omega_{ij} dx^i \wedge dx^j$. There is a $\mathcal{C}^\infty(M, \mathbb{C})[[\hbar]]$-bilinear map
\begin{equation*}
	\star_{\operatorname{M}}: \Omega^*(M, \mathcal{W}) \times \Omega^*(M, \mathcal{W}) \to \Omega^*(M, \mathcal{W})
\end{equation*}
called the \emph{fibrewise Moyal product} and defined by
\begin{equation}
	a \star_{\operatorname{M}} b = \sum_{r=0}^\infty \frac{1}{r!} \left( \frac{\hbar}{2} \right)^r \omega^{i_1j_1} \cdots \omega^{i_rj_r} \frac{\partial^r a}{\partial y^{i_1} \cdots \partial y^{i_r}} \wedge \frac{\partial^r b}{\partial y^{j_1} \cdots \partial y^{j_r}},
\end{equation}
for all $a, b \in \Omega^*(M, \mathcal{W})$, where $(\omega^{ij})$ is the inverse of $(\omega_{ij})$. It preserves the weight filtration of $\mathcal{W}$, i.e. $\mathcal{W}_{(r_1)} \star_{\operatorname{M}} \mathcal{W}_{(r_2)} \subset \mathcal{W}_{(r_1 + r_2)}$. Choose a symplectic connection $\nabla$ on $(M, \omega)$. Since $\nabla \omega = 0$,
\begin{equation*}
	(\Omega^*(M, \mathcal{W}), R, \nabla, \tfrac{1}{\hbar}[\quad, \quad ]_{\star_{\operatorname{M}}})
\end{equation*}
forms a curved dgla, where $[\quad, \quad]_{\star_{\operatorname{M}}}$ is the graded commutator of $\star_{\operatorname{M}}$ (with respect to the natural $\mathbb{N}$-grading on forms on $M$) and $R \in \Omega^2(M, \operatorname{Sym}^2T^\vee M)$ is determined by $\nabla^2 = \tfrac{1}{\hbar}[R, \quad]_{\star_{\operatorname{M}}}$.\par
Let $\tilde{\omega} \in \Omega^1(M, T^\vee M)$ be the unique element such that $\delta = \tfrac{1}{\hbar}[\tilde{\omega}, \quad]_{\star_{\operatorname{M}}}$. The following theorem states that we can deform $\nabla$ to a flat connection by a $1$-form $\gamma \in \Omega^1(M, \mathcal{W})$.

\begin{theorem}[Theorem 2.1 in \cite{Neu2002}; a slight generalization of Theorem 3.2 in \cite{Fed1994}]
	\label{Theorem 3.5}
	Suppose $\omega_\hbar \in \hbar\Omega^2(M, \mathbb{C})[[\hbar]]$ is $d$-closed and $s \in \Omega^0(M, \mathcal{W}_{(3)})$ is such that $\pi_0(s) = 0$. Then there exists a unique element $A \in \Omega^1(M, \mathcal{W}_{(2)})$ such that $\delta^{-1}A = s$ and $\gamma = -\tilde{\omega} + A$ is a solution to
	\begin{equation}
		\label{Equation 3.2}
		R + \nabla \gamma + \tfrac{1}{2\hbar} [\gamma, \gamma]_{\star_{\operatorname{M}}} = -\omega + \omega_\hbar.
	\end{equation}
	In particular, the connection $D_\gamma = \nabla + \tfrac{1}{\hbar}[\gamma, \quad]_{\star_{\operatorname{M}}} = \nabla - \delta + \tfrac{1}{\hbar}[A, \quad]_{\star_{\operatorname{M}}}$ on $\mathcal{W}$ is flat.
\end{theorem}

\begin{remark}
	In the above theorem, if $s = 0$, then $A \in \Omega^1(M, \mathcal{W}_{(3)})$.
\end{remark}

Let $\gamma$ be given as in Theorem \ref{Theorem 3.5}. The projection $\pi_0: (\Omega^*(M, \mathcal{W}), D_\gamma) \to (\mathcal{C}^\infty(M, \mathbb{C}), 0)$ is then a quasi-isomorphism. Identifying functions on $M$ with $D_\gamma$-flat sections of $\mathcal{W}$, $\star_{\operatorname{M}}$ induces a star product on $M$. Indeed, the de Rham cohomology class $\tfrac{1}{\hbar}[-\omega + \omega_\hbar]$ is independent of the choices of symplectic connections $\nabla$ and elements $s \in \Omega^0(M, \mathcal{W}_{(3)})$ with $\pi_0(s) = 0$ \cite{Neu2002}.\par
Every deformation quantization $(\mathcal{C}^\infty(M, \mathbb{C})[[\hbar]], \star)$ of $(M, \omega)$ is equivalent to the one obtained by a solution $\gamma$ to (\ref{Equation 3.2}) as in Theorem \ref{Theorem 3.5} for some $d$-closed form $\omega_\hbar \in \hbar \Omega^2(M, \mathbb{C})[[\hbar]]$. We call $\tfrac{1}{\hbar}[-\omega + \omega_\hbar]$ \emph{Fedosov's characteristic class} of $(\mathcal{C}^\infty(M, \mathbb{C})[[\hbar]], \star)$. This establishes a bijection between equivalence classes of deformation quantizations of $(M, \omega)$ and elements in $-\tfrac{1}{\hbar} [\omega] + H^2(M, \mathbb{C})[[\hbar]]$.

\begin{proof}[\myproof{Theorem}{\ref{Theorem 3.4}}]
	By Lemma 2 in \cite{KutLoo2000}, every class in $\hbar H^2(M, \mathbb{C})[[\hbar]]$ can be represented by a real analytic $d$-closed element $\omega_\hbar \in \hbar \Omega^*(M, \mathbb{C})[[\hbar]]$. We pick such an element $\omega_\hbar$, and also a real analytic symplectic connection $\nabla$ on $(M, \omega)$, whose existence is shown by Proposition \ref{Proposition C.1}. In this case, $(M, \omega, \nabla)$ is called a \emph{real analytic Fedosov manifold} in \cite{GelRetShu1998}.\par
	Now let $A \in \Omega^1(M, \mathcal{W}_{(3)})$ be the unique element such that $\delta^{-1}A = 0$ and $\gamma = -\tilde{\omega} + A$ is a solution to (\ref{Equation 3.2}). In the proof of Theorem \ref{Theorem 3.5}, $A$ is solved iteratively and the terms involved in the iteration are all real analytic. Hence, we can deduce that $D_\gamma$ is a real analytic connection. The projection $\pi_0: \Omega^*(M, \mathcal{W}) \to \mathcal{C}^\infty(M, \mathbb{C})$ is also real analytic. We can then clearly see that if $f, g$ are local real analytic complex valued functions on $M$, then so is $f \star g$.
\end{proof}

\subsection{Holomorphic deformation quantizations of a Stein complex symplectic manifold}
\label{Subsection 3.3}
\quad\par
Let $(X, \Omega)$ be a holomorphic symplectic manifold and $\mathcal{O}_X$ be the sheaf of holomorphic functions on $X$. A \emph{holomorphic deformation quantization} of $(X, \Omega)$ is a sheaf of $\hbar$-adically complete flat $\mathbb{C}[[\hbar]]$-algebras $\mathcal{A}$ on $X$ together with an isomorphism of sheaves of $\mathbb{C}$-algebras $\psi: \mathcal{A}/\hbar \mathcal{A} \to \mathcal{O}_X$ such that the first order term of the commutator on $\mathcal{A}$ is the holomorphic Poisson structure induced by $\Omega$. It was first studied by Nest-Tsygan \cite{NesTsy2001} (the general theory they developed will be reviewed in Appendix \ref{Appendix B}). Unlike the smooth case, although $\mathcal{A}$ can be locally trivialized by a holomorphic star product of local formal holomorphic functions, it might not be globally trivialized. In this subsection, we focus on the case when $X$ is Stein. We will prove the following theorem.

\begin{theorem}
	\label{Theorem 3.7}
	Let $(X, \Omega)$ be a Stein holomorphic symplectic manifold. Then every holomorphic deformation quantization $(\mathcal{A}, \psi)$ of $(X, \Omega)$ has a global differential trivialization, i.e. there is an isomorphism of $\mathbb{C}[[\hbar]]$-algebras $\mathcal{A} \cong (\mathcal{O}_X[[\hbar]], \star)$ for some holomorphic star product $\star$ on $\mathcal{O}_X[[\hbar]]$ such that $\psi$ is identified as the canonical isomorphism $\mathcal{O}_X[[\hbar]]/ ( \hbar \mathcal{O}_X[[\hbar]] ) \cong \mathcal{O}_X$.
\end{theorem}

An implication of Steinness of $X$ is that $X$ is admissible in the sense of Definition \ref{Definition B.2} because $H^1(X, \mathcal{O}_X) = H^2(X, \mathcal{O}_X) = 0$. By Theorem \ref{Theorem B.4}, equivalence classes of holomorphic deformation quantizations of $(X, \Omega)$ are thus in bijection with elements in $-\tfrac{1}{\hbar}[\Omega] + \mathbb{H}^2(X, F^1\Omega_X^*)[[\hbar]]$, where $\mathbb{H}^2(X, F^1\Omega_X^*)$ is the second hypercohomology of the Hodge filtration $F^1\Omega_X^*$ of the sheaf $(\Omega_X^*, \partial)$ of holomorphic de-Rham complex on $X$. As $H^1(X, \mathcal{O}_X) = H^2(X, \mathcal{O}_X) = H^1(X, \Omega_X^1) = 0$,
\begin{equation*}
	H^2(X, \mathbb{C}) \cong \mathbb{H}^2(X, F^1\Omega_X^*) \cong \frac{ \{ \alpha \in H^0(X, \Omega_X^2): \partial \alpha = 0 \} }{ \partial (H^0(X, \Omega_X^1)) }.
\end{equation*}
Thus, it suffices to consider a $d$-closed element $\Omega_\hbar \in \hbar \Omega^{2, 0}(X)[[\hbar]]$, whose cohomology class is identified with an element in $\hbar \mathbb{H}^2(X, F^1\Omega_X^*)[[\hbar]]$, and we need to construct a holomorphic star product on $\mathcal{O}_X[[\hbar]]$ whose equivalence class corresponds to $\tfrac{1}{\hbar} [-\Omega + \Omega_\hbar]$.\par
Indeed, every holomorphic deformation quantization of $(X, \Omega)$ is isomorphic to the sheaf of flat sections of the holomorphic Weyl bundle $\mathcal{W} := \widehat{\operatorname{Sym}} T^{\vee (1, 0)}X[[\hbar]]$ of $(X, \Omega)$ equipped with a certain flat connection \cite{NesTsy2001}. This provides a geometric construction of holomorphic deformation quantizations, the first step of which is to choose an $T^{1, 0}X$-connection $\nabla$ on $T^{1, 0}X$ preserving $\Omega$. Similar to the smooth case, one can define the operators $\delta, \delta^{-1}, \pi_0$, the fibrewise holomorphic Moyal product $\star_{\operatorname{M}}$ with its graded commutator $[\quad, \quad]_{\star_{\operatorname{M}}}$, and the weight filtration on $\mathcal{W}$ (see Appendix \ref{Appendix B}). Let $R \in \Omega^{2, 0}(X, \operatorname{Sym}^2 T^{\vee (1, 0)}X)$ be given by $(\nabla + \overline{\partial})^2 = \tfrac{1}{\hbar} [R, \quad]_{\star_{\operatorname{M}}}$ and $\tilde{\Omega} \in \Omega^{1, 0}(X, T^{\vee (1, 0)}X)$ be given by $\delta = \tfrac{1}{\hbar} [\tilde{\Omega}, \quad]_{\star_{\operatorname{M}}}$.\par
By Theorem \ref{Theorem B.3}, one can deform $\nabla + \overline{\partial}$ to a flat connection $D_\gamma := \nabla + \overline{\partial} - \delta + \tfrac{1}{\hbar} [A + B, \quad]_{\star_{\operatorname{M}}}$ for some $A \in \Omega^{1, 0}(X, \mathcal{W}_{(2)})$ and $B \in \Omega^{0, 1}(X, \mathcal{W}_{(3)})$. Then the sheaf of $D_\gamma$-flat sections of $\mathcal{W}$ forms a holomorphic deformation quantization $\mathcal{A}$ of $(X, \Omega)$. If $B$ is non-zero, then the $(0, 1)$-part $\overline{\partial} + \tfrac{1}{\hbar}[B, \quad]_{\star_{\operatorname{M}}}$ of $D_\gamma$ is not the holomorphic structure of $\mathcal{W}$, thus $\mathcal{A}$ is not of the form $(\mathcal{O}_X[[\hbar]], \star)$. In general, $B$ cannot be chosen to vanish, unless one can choose a torsion-free holomorphic connection $\nabla$ on $T^{1, 0}X$ such that $\nabla \Omega = 0$, which we call a \emph{holomorphic symplectic connection} on $(X, \Omega)$. The following theorem does not require that $X$ is Stein.

\begin{theorem}[a direct application of Theorem 2.1 in \cite{Vai2002}]
	\label{Theorem 3.8}
	Let $\nabla$ be a holomorphic symplectic connection on $(X, \Omega)$ and $\Omega_\hbar \in \hbar\Omega^{2, 0}(X)[[\hbar]]$ be $d$-closed. Then there exists a unique $\overline{\partial}$-closed element $A \in \Omega^{1, 0}(X, \mathcal{W}_{(3)})$ such that $\delta^{-1} A = 0$ and $\gamma := -\tilde{\Omega} + A$ is a solution to
	\begin{equation}
		\label{Equation 3.3}
		R + \nabla \gamma + \tfrac{1}{2\hbar} [\gamma, \gamma]_{\star_{\operatorname{M}}} = -\Omega + \Omega_\hbar,
	\end{equation}
	In particular, the connection $D_\gamma := \nabla + \overline{\partial} + \tfrac{1}{\hbar}[\gamma, \quad]_{\star_{\operatorname{M}}}$ on $\mathcal{W}$ is flat.
\end{theorem}

\begin{proof}[\myproof{Theorem}{\ref{Theorem 3.7}}]
	Fix a $d$-closed element $\Omega_\hbar \in \hbar \Omega^{2, 0}(X)[[\hbar]]$. By hypothesis, $X$ is Stein. Then we pick a holomorphic symplectic connection $\nabla$ on $(X, \Omega)$, whose existence is shown by Proposition \ref{Proposition C.2}. By Theorem \ref{Theorem 3.8}, there exists a unique $\overline{\partial}$-closed element $A \in \Omega^{1, 0}(X, \mathcal{W}_{(3)})$ such that $\delta^{-1} A = 0$ and $\gamma := -\tilde{\Omega} + A$ is a solution to (\ref{Equation 3.3}). As a consequence, the embedding
	\begin{equation*}
		(\Omega_{\operatorname{hol}}^{*, 0}(X, \mathcal{W}), -\Omega + \Omega_\hbar, D_\gamma^{1, 0}, \star_{\operatorname{M}}) \hookrightarrow (\Omega^*(X, \mathcal{W}), -\Omega + \Omega_\hbar, D_\gamma, \star_{\operatorname{M}})
	\end{equation*}
	is a quasi-isomorphism of curved dgas, where $\Omega_{\operatorname{hol}}^{*, 0}(X, \mathcal{W})$ is the space of holomorphic $\mathcal{W}$-valued $(*, 0)$-forms on $X$ and $D_\gamma^{1, 0} = \nabla + \tfrac{1}{\hbar}[\gamma, \quad]_{\star_{\operatorname{M}}}$ is the $(1, 0)$-part of $D_\gamma$. The canonical projection $\pi_0: (\Omega_{\operatorname{hol}}^{*, 0}(X, \mathcal{W}), D_\gamma^{1, 0}) \to (\mathcal{O}(X)[[\hbar]], 0)$ is also a quasi-isomorphism. Therefore, the holomorphic deformation quantization of $(X, \Omega)$ determined by $D_\gamma$-flat sections of $\mathcal{W}$ is of the form $(\mathcal{O}_X[[\hbar]], \star)$.
\end{proof}

\subsection{Recovering holomorphic star products on $\mathcal{O}_X[[\hbar]]$ from star products on $\mathcal{C}_M^\infty[[\hbar]]$}
\label{Subsection 3.4}
\quad\par
Suppose $\iota: (M, \omega) \hookrightarrow (X, \Omega)$ is a complexification of a symplectic manifold $(M, \omega)$. We discuss how a holomorphic deformation quantization $\mathcal{A}$ of $(X, \Omega)$ restricts to a deformation quantization of $(M, \omega)$, up to equivalence.\par
By Theorem \ref{Theorem 3.7}, the restriction of $\mathcal{A}$ on a Stein neighbourhood of $M$ in $X$ has a differential trivialization. Without loss of generality, we assume $X$ is Stein and $\mathcal{A} = (\mathcal{O}_X[[\hbar]], \widetilde{\star})$ for some holomorphic star product $\widetilde{\star}$ on $\mathcal{O}_X[[\hbar]]$. Let $\{\widetilde{C}_r\}_{r=1}^\infty$ be the defining sequence of holomorphic bi-differential operators for $\widetilde{\star}$. Each term $\widetilde{C}_r$ induces a bi-differential operator $C_r$ on $M$ and the sequence $\{C_r\}_{r=1}^\infty$ determines a real analytic deformation quantization $(\mathcal{C}^\infty(M, \mathbb{C}), \star)$ of $(M, \omega)$, whence we have a morphism of sheaves of $\mathbb{C}[[\hbar]]$-algebras
\begin{equation*}
	\iota^*: \mathcal{A} \to \iota_*( \mathcal{C}_M^\omega[[\hbar]], \star).
\end{equation*}

Conversely, suppose $(\mathcal{C}^\infty(M, \mathbb{C})[[\hbar]], \star)$ is a real analytic deformation quantization on $(M, \omega)$ with the defining sequence $\{ C_r \}_{r=1}^\infty$ of bi-differential operators. A natural question is whether one can reconstruct a complexification $(X, \Omega)$ of $(M, \omega)$ and a holomorphic deformation quantization $(\mathcal{A}, \psi)$ of $(X, \Omega)$ so that $\mathcal{A}$ restricts to $(\mathcal{C}^\infty(M, \mathbb{C})[[\hbar]], \star)$. A naive consideration suggests that we should take analytic continuation of $C_r$'s. The sequence $\{C_r\}_{r=1}^\infty$ is, however, infinite, thus there might not be a complexification $(X, \Omega)$ of $(M, \omega)$ admitting analytic continuation of all the $C_r$'s.\par
To overcome this problem, we observe that equivalence classes of deformation quantizations of $(M, \omega)$ are determined by Fedosov's characteristic classes. By taking analytic continuation of a real analytic representative of Fedosov's characteristic classes of $(\mathcal{C}^\infty(M, \mathbb{C})[[\hbar]], \star)$ instead, we can show that the answer to the above question is affirmative.

\begin{theorem}[$=$ Theorem \ref{Theorem 1.2}]
	Every symplectic manifold $(M, \omega)$ admits a complexification $\iota: (M, \omega) \hookrightarrow (X, \Omega)$ which satisfies the following condition. For any deformation quantization $(\mathcal{C}^\infty(M, \mathbb{C})[[\hbar]], \star)$ of $(M, \omega)$, there exists a holomorphic deformation quantization $(\mathcal{A}, \psi)$ of $(X, \Omega)$ and a morphism of sheaves of $\mathbb{C}[[\hbar]]$-algebras $\tau: \mathcal{A} \to \iota_*( \mathcal{C}_M^\infty[[\hbar]], \star)$, such that the following diagram is commutative:
	\begin{center}
		\begin{tikzcd}
			\mathcal{A} \ar[r, "\tau"] \ar[d] & \iota_*\mathcal{C}_M^\infty[[\hbar]] \ar[d]\\
			\mathcal{O}_X \ar[r, "\iota^*"'] & \iota_*\mathcal{C}_M^\infty
		\end{tikzcd}
	\end{center}
	where $\mathcal{A} \to \mathcal{O}_X$ is the composition of $\psi: \mathcal{A}/\hbar \mathcal{A} \to \mathcal{O}_X$ with the canonical projection $\mathcal{A} \to \mathcal{A}/\hbar \mathcal{A}$, and $\iota_*\mathcal{C}_M^\infty[[\hbar]] \to \iota_*\mathcal{C}_M^\infty$ is the pushforward of the canonical projection $\mathcal{C}_M^\infty[[\hbar]] \to \mathcal{C}_M^\infty$ by $\iota$.
\end{theorem}

The main idea is to make use of Fedosov quantization as discussed in Subsections \ref{Subsection 3.2} and \ref{Subsection 3.3}. Note that there are similar structures on the Weyl bundle of a real symplectic manifold and the holomorphic Weyl bundle of a holomorphic symplectic manifold. To avoid confusion, we reserve the notations $\mathcal{W}, \mathcal{W}_{(k)}, \delta, \delta^{-1}, \pi_0, \star_{\operatorname{M}}$ for the Weyl bundle of $(M, \omega)$ and its structures, and use notations $\widetilde{\mathcal{W}}, \widetilde{\mathcal{W}}_{(k)}, \widetilde{\delta}, \widetilde{\delta}^{-1}, \widetilde{\pi}_0, \widetilde{\star}_{\operatorname{M}}$ for their counterparts on a complexification $(X, \Omega)$ of $(M, \omega)$.

\begin{proof}[\myproof{Theorem}{\ref{Theorem 1.2}}]
	Pick a Stein complexification $\iota: (M, \omega) \hookrightarrow (X, \Omega)$ admitting a holomorphic symplectic connection $\widetilde{\nabla}$ on $(X, \Omega)$, i.e. a torsion-free holomorphic connection $\widetilde{\nabla}$ such that $\widetilde{\nabla} \Omega = 0$. Then $\nabla := \iota^*\widetilde{\nabla}$ is a real analytic symplectic connection on $(M, \omega)$. Shrinking $X$ if necessary, we can also assume that $\iota: M \hookrightarrow X$ admits a deformation retraction so that we have isomorphisms
	\begin{equation*}
		\frac{\{ \alpha \in H^0(X, \Omega_X^2): \partial \alpha = 0 \}}{\partial H^0(X, \Omega_X^1)} \cong H^2(X, \mathbb{C}) \cong H^2(M, \mathbb{C}).
	\end{equation*}
	Now consider any deformation quantization $(\mathcal{C}^\infty(M, \mathbb{C})[[\hbar]], \star)$. Via the above isomorphisms of cohomologies, we can choose a $d$-closed element $\Omega_\hbar \in \hbar \Omega^{2, 0}(X)[[\hbar]]$ so that $\tfrac{1}{\hbar} [-\omega + \omega_\hbar]$ is Fedosov's characteristic class of $(\mathcal{C}^\infty(M, \mathbb{C})[[\hbar]], \star)$, where $\omega_\hbar := \iota^*\Omega_\hbar$ is by construction real analytic. Let $A \in \Omega^1(M, \mathcal{W}_{(3)})$ be the unique element such that $\delta^{-1}A = 0$ and $\gamma := -\tilde{\omega} + A$ is a solution to (\ref{Equation 3.2}). It defines a real analytic deformation quantization of $(M, \omega)$ whose Fedosov's characteristic class is $\tfrac{1}{\hbar} [-\omega + \omega_\hbar]$. Without loss of generality, we assume $(\mathcal{C}^\infty(M, \mathbb{C})[[\hbar]], \star)$ is given in this way.\par
	Similarly, let $\widetilde{A} \in \Omega^{1, 0}(X, \widetilde{\mathcal{W}}_{(3)})$ be such that $\widetilde{\delta}^{-1}\widetilde{A} = \overline{\partial} \widetilde{A} = 0$ and $\widetilde{\gamma} := -\tilde{\Omega} + \widetilde{A}$ is a solution to (\ref{Equation 3.3}). It defines a holomorphic deformation quantization $(\mathcal{A}, \psi)$ of $(X, \Omega)$, where $\mathcal{A}$ is of the form $(\mathcal{O}_X[[\hbar]], \widetilde{\star})$ and $\psi: \mathcal{A}/\hbar \mathcal{A} \to \mathcal{O}_X$ is the canonical identification $\mathcal{O}_X[[\hbar]] / ( \hbar \mathcal{O}_X[[\hbar]] ) \cong \mathcal{O}_X$.\par
	Tracing the iterative construction of $A$ in Theorem \ref{Theorem 3.5} and comparing it with that of $\widetilde{A}$ in Theorem \ref{Theorem 3.8}, it is evident that $A = \iota^*\widetilde{A}$. Then the following is a morphism of curved dgas
	\begin{equation*}
		\iota^*: (\Omega_{\operatorname{hol}}^{*, 0}(X, \widetilde{\mathcal{W}}), -\Omega + \Omega_\hbar, \widetilde{D}_{\widetilde{\gamma}}^{1, 0}, \widetilde{\star}_{\operatorname{M}}) \to (\Omega_\omega^*(M, \mathcal{W}), -\omega + \omega_\hbar, D_\gamma, \star_{\operatorname{M}}),
	\end{equation*}
	where $\Omega_{\operatorname{hol}}^{*, 0}(X, \widetilde{\mathcal{W}})$ is the space of holomorphic sections of $\left( \textstyle\bigwedge T^{\vee (1, 0)} X \right) \otimes \widetilde{\mathcal{W}}$, $\Omega_\omega^*(M, \mathcal{W})$ is the space of real analytic sections of $\left( \textstyle\bigwedge T^\vee M \right) \otimes \mathcal{W}$, $\widetilde{D}_{\widetilde{\gamma}}^{1, 0} = \widetilde{\nabla} + \tfrac{1}{\hbar} [\widetilde{\gamma}, \quad]_{\widetilde{\star}_{\operatorname{M}}}$ and $D_\gamma = \nabla + \tfrac{1}{\hbar} [\gamma, \quad]_{\star_{\operatorname{M}}}$. The restriction of $\iota^*$ on $\widetilde{D}_{\tilde{\gamma}}^{1, 0}$-flat holomorphic sections induces a morphism of sheaves of $\mathbb{C}[[\hbar]]$-algebras
	\begin{equation*}
		\tau := \iota^*: \mathcal{A} \to \iota_*(\mathcal{C}_M^\omega[[\hbar]], \star) \subset \iota_*(\mathcal{C}_M^\infty[[\hbar]], \star)
	\end{equation*}
	such that the required diagram is commutative.
\end{proof}

\section{The local model for brane quantization}
\label{Section 4}
Geometric quantization and brane quantization \cite{GukWit2009} are different approaches to obtain quantum Hilbert spaces of a symplectic manifold $(M, \omega)$. The former requires a prequantum line bundle and a polarization of $(M, \omega)$, while the latter requires a complexification $(X, \Omega)$ of $(M, \omega)$, a Lagrangian A-brane $\mathcal{B}$ supported on $M$ and a canonical coisotropic A-brane $\mathcal{B}_{\operatorname{cc}}$ on $(X, \Omega)$. The goal of this section is to show that, if the data in geometric quantization of $(M, \omega)$ are real analytic, then they uniquely determine a local model near $M$ for brane quantization.\par
In Subsection \ref{Subsection 4.1}, we will discuss a correspondence between a prequantum line bundle of $(M, \omega)$ and a canonical coisotropic A-brane $\mathcal{B}_{\operatorname{cc}}$ on $(X, \Omega)$, and we will prove Theorem \ref{Theorem 1.3}. In Subsection \ref{Subsection 4.2}, we will describe the relationship between a polarization of $(M, \omega)$ and a \emph{holomorphic polarization} of $(X, \Omega)$, which is a required datum for the definition of $\operatorname{Hom}(\mathcal{B}, \mathcal{B}_{\operatorname{cc}})$ proposed in \cite{BisGua2022, GaiWit2022}. In Subsection \ref{Subsection 4.3}, we will discuss the relationship between quantum Hilbert spaces of $(M, \omega)$ in geometric quantization and $\operatorname{Hom}(\mathcal{B}, \mathcal{B}_{\operatorname{cc}})$.

\subsection{Canonical coisotropic A-branes as extensions of prequantum line bundles}
\label{Subsection 4.1}
\quad\par
Throughout this section, we only consider \emph{real analytic prequantum line bundles} of $(M, \omega)$, i.e. real analytic Hermitian line bundles over $M$ with a real analytic unitary connection of curvature $\tfrac{1}{\sqrt{-1}}\omega$, and \emph{real analytic polarizations} of $(M, \omega)$, i.e. real analytic involutive complex Lagrangian vector subbundles of $TM_\mathbb{C}$. We first show how a canonical coisotropic A-brane on a complexification of $(M, \omega)$ induces a prequantum line bundle of $(M, \omega)$.

\begin{proposition}
	\label{Proposition 4.1}
	Let $\iota: (M, \omega) \hookrightarrow (X, \Omega)$ be a complexification of $(M, \omega)$. Suppose that $\mathcal{B}_{\operatorname{cc}}$ is a canonical coisotropic A-brane on $(X, \Omega)$ with Chan-Paton bundle $(\widehat{L}, \widehat{h}, \widehat{\nabla})$. Then
	\begin{equation*}
		(L, h, \nabla) := (\iota^*\widehat{L}, \iota^*\widehat{h}, \iota^*\widehat{\nabla})
	\end{equation*}
	is a real analytic prequantum line bundle of $(M, \omega)$.
\end{proposition}
\begin{proof}
	Shrinking $X$ if necessary, pick $\alpha \in \Omega^1(X)$ such that $\operatorname{Im} \Omega = d\alpha$ and $\iota^*\alpha = 0$. Then $\widehat{\nabla} - \alpha$ gives a holomorphic structure on $\widehat{L}$ restricting to a real analytic structure on $L$ such that $\nabla = \iota^*(\widehat{\nabla} - \alpha)$ is real analytic, and hence so is $h$. The curvature of $\nabla$ is $\tfrac{1}{\sqrt{-1}} \iota^* \operatorname{Re} \Omega = \tfrac{1}{\sqrt{-1}} \omega$.
\end{proof}

Motivated by Proposition \ref{Proposition 4.1}, we introduce the following definition.

\begin{definition}
	Let $L$ be a real analytic prequantum line bundle of $(M, \omega)$. A \emph{brane extension} of $(M, \omega, L)$ is a quadruple $(X, \Omega, \iota, \widehat{L})$, where $\iota: (M, \omega) \hookrightarrow (X, \Omega)$ is a complexification and $\mathcal{B}_{\operatorname{cc}} = (X, \widehat{L})$ is a canonical coisotropic A-brane on $(X, \Omega)$ such that $L = \iota^*\widehat{L}$ as Hermitian line bundles with unitary connections over $M$.
\end{definition}

The existence of a brane extension is known when $(M, \omega)$ is real analytic K\"ahler \cite{Don2002}.

\begin{theorem}[$=$ Theorem \ref{Theorem 1.3}]
	Let $L$ be a real analytic prequantum line bundle of $(M, \omega)$. Then $(M, \omega, L)$ admits a brane extension $(X, \Omega, \iota, \widehat{L})$. Moreover, if $(X', \Omega', \iota', \widehat{L}')$ is another brane extension of $(M, \omega, L)$, then there exists a diffeomorphism $\phi: U \to U'$ from a neighbourhood $U$ of $\iota(M)$ in $X$ onto a neighbourhood $U'$ of $\iota'(M)$ in $X'$ and an isomorphism
	\begin{equation*}
		\Phi: \widehat{L} \vert_U \to \phi^*(\widehat{L}' \vert_{U'})
	\end{equation*}
	of Hermitian line bundles with unitary connections over $U$ such that $\phi \circ \iota = \iota'$, $\phi^*(\Omega' \vert_{U'}) = \Omega \vert_U$ and $\Phi$ restricts to the identity map $\operatorname{Id}_L$.
\end{theorem}
\begin{proof}
	We first show the existence of a brane extension of $(M, \omega, L)$. Pick a complexification $\iota: (M, \omega) \hookrightarrow (X, \Omega)$ admitting a holomorphic extension of $L$ such that $\operatorname{Im} \Omega$ is $d$-exact, implying that the cohomology class $[\operatorname{Re}\Omega] = [\Omega]$ is integral. Thus, there exists a canonical coisotropic A-brane $\mathcal{B}'_{\operatorname{cc}} = (X, \widehat{L}')$ on $(X, \Omega)$. Then $L' := (\iota^*\widehat{L}') \otimes L^\vee$ is flat. Shrinking $X$ if necessary, we can assume that $\iota: M \hookrightarrow X$ admits a deformation retraction $\pi: X \to M$. Then $\widehat{L} := \widehat{L}' \otimes \pi^*(L')^\vee$ is the Chan-Paton bundle of a canonical coisotropic A-brane on $(X, \Omega)$, which restricts to $L$.\par
	Next, suppose that $(X', \Omega', \iota', \widehat{L}')$ is also a brane extension of $(M, \omega, L)$. Because the germ of complexifications of $(M, \omega)$ is unique, without loss of generality, we can assume that $X = X'$, $\Omega = \Omega'$ and $\iota = \iota'$ (and hence we can take $\phi = \operatorname{Id}_X$).\par
	Now, both the curvatures of $\nabla^{\widehat{L}}$ and $\nabla^{\widehat{L}'}$ are $\tfrac{1}{\sqrt{-1}}\operatorname{Re} \Omega$. It implies that there is an isomorphism $\Psi: \widehat{L} \to \widehat{L}'$ of Hermitian line bundles over $X$. We can modify $\Psi$ so as to obtain an isomorphism $\Psi': \widehat{L} \to \widehat{L}'$ of Hermitian line bundles over $X$ which restricts to the identity map $\operatorname{Id}_L$ as follows. The map $L \to L$ obtained from the restriction of $\Psi$ must be of the form $e^{\sqrt{-1}\eta}$ for some $\eta \in \mathcal{C}^\infty(M)$. Then $\Psi' := e^{-\sqrt{-1}\pi^*\eta} \Psi$ is a desired isomorphism.\par
	By the above argument, without loss of generality again, we can assume that $\widehat{L} = \widehat{L}'$ as Hermitian line bundles over $X$. As both the connections $\nabla^{\widehat{L}}$ and $\nabla^{\widehat{L}'}$ are unitary, $\nabla^{\widehat{L}'} = \nabla^{\widehat{L}} + \tfrac{1}{\sqrt{-1}} \alpha$ for some $d$-closed $\alpha \in \Omega^1(X)$. We can see that $\iota^*\alpha = 0$, whence $[\alpha] = 0 \in H^1(X, \mathbb{R})$. In other words, there exists $\theta \in \mathcal{C}^\infty(X)$ such that $\alpha = d\theta$. Then $\Phi = e^{\sqrt{-1}\theta}: \widehat{L} \to \widehat{L}$ is an isomorphism of Hermitian line bundles over $X$ intertwining $\nabla^{\widehat{L}}$ and $\nabla^{\widehat{L}'}$. Observe that the restriction of $\Phi$ onto $L$ preserves $\nabla^L$, whence $d(\iota^*\theta) = 0$. Replacing $\theta$ by $\theta - \pi^*\iota^*\theta$, we can conclude that $\Phi$ restricts to the identity map $\operatorname{Id}_L$.
\end{proof}

\subsection{Holomorphic extension of polarizations}
\label{Subsection 4.2}
\quad\par
We first introduce the definition of holomorphic polarizations.

\begin{definition}
	\label{Definition 4.4}
	A \emph{holomorphic polarization} \footnote{It is weaker than that in \cite{GaiWit2022} requiring a holomorphic Lagrangian fibration of $X$ with contractible fibres.} of a holomorphic symplectic manifold $(X, \Omega)$ is an involutive Lagrangian holomorphic vector subbundle $\mathcal{P}^{\operatorname{c}}$ of $T^{1, 0}X$.
\end{definition}
Note that $\mathcal{P}^{\operatorname{c}}$ determines an involutive complex Lagrangian vector subbundle $\widehat{\mathcal{P}}$ of $TX$ such that $\mathcal{P}^{\operatorname{c}} = \widehat{\mathcal{P}}^{1, 0}$. Indeed, germs of holomorphic polarizations of $(X, \Omega)$ and real analytic polarizations of $(M, \omega)$ are in one-to-one correspondence.

\begin{proposition}
	Let $(X, \Omega)$ be a complexification of $(M, \omega)$. Suppose that $\mathcal{P}^{\operatorname{c}}$ is a holomorphic polarization of $(X, \Omega)$. Then there exists a unique real analytic polarization $\mathcal{P}$ of $(M, \omega)$ such that $\mathcal{P}^{\operatorname{c}}$ is a holomorphic extension of $\mathcal{P}$ and $\mathcal{P}^{\operatorname{c}} \cap \overline{\mathcal{P}^{\operatorname{c}}} = (\widehat{\mathcal{P}} \cap TM) \otimes \mathbb{C}$. Conversely, suppose $\mathcal{P}$ is a real analytic polarization of $(M, \omega)$ and $\mathcal{P}^{\operatorname{c}} \subset T^{1, 0}X$ is its holomorphic extension. Then $\mathcal{P}^{\operatorname{c}}$ is a holomorphic polarization of $(X, \Omega)$.
\end{proposition}

\begin{example}
	\label{Example 4.5}
	If $(M, \omega)$ is real analytic K\"ahler, then the polarizations $T^{0, 1}M$ and $T^{1, 0}M$ are real analytic. Indeed, we can see from the first construction in Example \ref{Example 3.3} that, for a sufficiently small complexification $(X, \Omega)$ of $(M, \omega)$, we can always equip $(X, \Omega)$ with a pair of holomorphic Lagrangian fibrations $\pi: X \to M$ and $\check{\pi}: X \to \overline{M}$ such that the kernel $\mathcal{P}^{\operatorname{c}}$ of $d\check{\pi}: T^{1, 0}X \to T^{0, 1}M$ (resp. the kernel $\check{\mathcal{P}}^{\operatorname{c}}$ of $d\pi: T^{1, 0}X \to T^{1, 0}M$) is a holomorphic extension of $T^{1, 0}M$ (resp. $T^{0, 1}M$).
\end{example}

\subsection{Holomorphic geometric quantization}
\label{Subsection 4.3}
\quad\par
Suppose $L$ is a real analytic prequantum line bundle of $(M, \omega)$ and $\mathcal{P}$ is a real analytic polarization of $(M, \omega)$. Pick a brane extension $(X, \Omega, \iota, \widehat{L})$ of $(M, \omega, L)$ so that we can obtain a canonical coisotropic A-brane $\mathcal{B}_{\operatorname{cc}} = (X, \widehat{L})$ on $(X, \Omega)$. Let $\mathcal{B}$ be the Lagrangian A-brane on $(X, \operatorname{Im} \Omega)$ with support $M$ and Chan-Paton bundle $(M \times \mathbb{C}, d)$. We will discuss how the morphism space $\operatorname{Hom}(\mathcal{B}, \mathcal{B}_{\operatorname{cc}})$ is related to a quantum Hilbert space for the geometric quantization of $(M, \omega)$, based on the ideas given by Gaiotto-Witten \cite{GaiWit2022} and Bischoff-Gualtieri \cite{BisGua2022}. Roughly speaking, their ideas point to a holomorphic version of geometric quantization.\par
More precisely, we need to include metaplectic correction. Note that since $\mathcal{P}$ is real analytic, the quotient bundle $\mathcal{Q} := TM_\mathbb{C}/\mathcal{P}$ and its Bott connection are real analytic. The \emph{canonical line bundle} $K_\mathcal{P}$ of $\mathcal{P}$, defined as the determinant line bundle of $\mathcal{Q}^\vee$ equipped with the flat $\mathcal{P}$-connection induced by the Bott connection, is also real analytic. In order to perform metaplectic correction, we assume the existence of a real analytic square root of $K_\mathcal{P}$ and we choose one, denoted by $\sqrt{K_\mathcal{P}}$.\par
The unitary connection $\nabla$ on $L$ restricts to a flat $\mathcal{P}$-connection on $L$ so that we can equip $L^{(1)} := L \otimes \sqrt{K_\mathcal{P}}$ with the induced flat $\mathcal{P}$-connection $\nabla^{\mathcal{P}}$, which is real analytic. Similar to an ordinary real analytic connection, a real analytic $\mathcal{P}$-connection can be holomorphically extended on a sufficiently small complexification of $(M, \omega)$. By shrinking $(X, \Omega)$, we assume that $\mathcal{P}$ and $(L^{(1)}, \nabla^{\mathcal{P}})$ have holomorphic extensions $\mathcal{P}^{\operatorname{c}}$ and $((L^{(1)})^{\operatorname{c}}, \nabla^{\mathcal{P}^{\operatorname{c}}})$ on $(X, \Omega)$ respectively.\par
In this case, it makes sense to define a $\mathcal{P}^{\operatorname{c}}$-\emph{polarized holomorphic section} of $(L^{(1)})^{\operatorname{c}}$ as a $\nabla^{\mathcal{P}^{\operatorname{c}}}$-parallel holomorphic section $\sigma$ of $(L^{(1)})^{\operatorname{c}}$ 
i.e. for all local holomorphic sections $\xi$ of $\mathcal{P}^{\operatorname{c}}$, $\nabla_\xi^{\mathcal{P}^{\operatorname{c}}} \sigma = 0$. Clearly, a local $\mathcal{P}^{\operatorname{c}}$-polarized holomorphic section of $(L^{(1)})^{\operatorname{c}}$ restricts to a local $\mathcal{P}$-polarized real analytic sections of $L^{(1)}$, while a local $\mathcal{P}$-polarized real analytic sections of $L^{(1)}$ extends to a local $\mathcal{P}^{\operatorname{c}}$-polarized holomorphic sections of $(L^{(1)})^{\operatorname{c}}$. A suspicion motivated by Gaiotto-Witten \cite{GaiWit2022} and Bischoff-Gualtieri \cite{BisGua2022} is that the morphism space for the pair $(\mathcal{B}, \mathcal{B}_{\operatorname{cc}})$ should be regarded as the sheaf of $\mathcal{P}^{\operatorname{c}}$-polarized holomorphic sections of $(L^{(1)})^{\operatorname{c}}$. In this case, $\operatorname{Hom}(\mathcal{B}, \mathcal{B}_{\operatorname{cc}})$ restricts to the sheaf of $\mathcal{P}$-polarized real analytic sections of $L^{(1)}$.\par
The case when $(M, \omega)$ is spin and prequantizable real analytic K\"ahler serves as a nice example for understanding the above correspondence between geometric quantization of $(M, \omega)$ and the A-model on $(X, \operatorname{Im} \Omega)$. Pick a holomorphic Hermitian line bundle $L$ with Chern connection of curvature $\tfrac{1}{\sqrt{-1}} \omega$ and a square root $\sqrt{K}$ of the canonical line bundle of $M$. With respect to the K\"ahler polarization $\mathcal{P} = T^{0, 1}M$, a $\mathcal{P}$-polarized section of $L^{(1)} = L \otimes \sqrt{K}$ is equivalent to a holomorphic section of $L^{(1)}$, which is automatically real analytic.

\section{Compositions of morphisms of A-branes in K\"ahler case}
\label{Section 5}
The goal of this section is to construct various sheaves of algebras and their (bi)modules on any K\"ahler manifold $(M, \omega)$, which realize the compositions of morphisms in  $\operatorname{Hom}(\mathcal{B}_{\operatorname{cc}}, \mathcal{B}_{\operatorname{cc}})$, $\operatorname{Hom}(\mathcal{B}, \mathcal{B}_{\operatorname{cc}})$, $\operatorname{Hom}(\mathcal{B}_{\operatorname{cc}}, \mathcal{B})$, $\operatorname{Hom}(\overline{\mathcal{B}}_{\operatorname{cc}}, \mathcal{B}_{\operatorname{cc}})$ and $\operatorname{Hom}(\overline{\mathcal{B}}_{\operatorname{cc}}, \overline{\mathcal{B}}_{\operatorname{cc}})$ on a complexification $(X, \Omega)$ of $M$ when $M$ is real analytic K\"ahler and the Chan-Paton bundle of $\mathcal{B}$ is trivial. We follow the idea of a series of papers \cite{ChaLeuLi2021, ChaLeuLi2022a, ChaLeuLi2022b, ChaLeuLi2023} that Fedosov's approach to deformation quantization is used as a tool.\par
We first introduce some basic notations in K\"ahler geometry. In local complex coordinates $(z^1, ..., z^n)$ on $M$, we write $\omega = \omega_{\alpha\overline{\beta}} dz^\alpha \wedge d\overline{z}^\beta$ and let $(\omega^{\overline{\alpha}\beta})$ be the inverse of $(\omega_{\alpha\overline{\beta}})$. Let $\nabla$ be the Levi-Civita connection of the K\"ahler manifold $(M, \omega)$. The curvature of $\nabla$ is locally written as
\begin{equation*}
	\nabla^2 \left( \frac{\partial}{\partial z^\mu} \right) = R_{\alpha\overline{\beta}\mu}^\nu dz^\alpha \wedge d\overline{z}^\beta \otimes \frac{\partial}{\partial z^\nu} \quad \text{and} \quad \nabla^2 \left( \frac{\partial}{\partial \overline{z}^\mu} \right) = R_{\alpha\overline{\beta}\overline{\mu}}^{\overline{\nu}} dz^\alpha \wedge d\overline{z}^\beta \otimes \frac{\partial}{\partial \overline{z}^\nu}.
\end{equation*}
The Ricci form of $(M, \omega)$ is then locally given by $\sqrt{-1} R_{\alpha\overline{\beta}\lambda}^\lambda dz^\alpha \wedge d\overline{z}^\beta$.\par
Recall that $(M, \omega)$ has the Weyl bundle $\mathcal{W} = \widehat{\operatorname{Sym}} T^\vee M_\mathbb{C}[[\hbar]]$. By abuse of notations, we denote by $\nabla$ the connection on $\mathcal{W}$ induced by the Levi-Civita connection. Note that on $\Omega^*(M, \mathcal{W})$,
\begin{equation*}
	\nabla^2 = - dz^\alpha \wedge d\overline{z}^\beta \otimes \left( R_{\alpha\overline{\beta}\mu}^\nu w^\mu \frac{\partial}{\partial w^\nu} + R_{\alpha\overline{\beta}\overline{\mu}}^{\overline{\nu}} \overline{w}^\mu \frac{\partial}{\partial \overline{w}^\nu} \right),
\end{equation*}
where we denote by $w^\mu$ (resp. $\overline{w}^\mu$) the covector $dz^\mu$ (resp. $d\overline{z}^\mu$) regarded as a section of $\mathcal{W}$.\par
In Subsection \ref{Subsection 5.1}, we will construct a (non-formal) deformation quantization of $(M, \omega)$ to realize the multiplication on $\operatorname{Hom}(\mathcal{B}_{\operatorname{cc}}, \mathcal{B}_{\operatorname{cc}})$. By Theorem \ref{Theorem 1.2}, it extends to a holomorphic deformation quantization of a complexification of $(M, \omega)$. In Subsection \ref{Subsection 5.2}, we will refine this extension when $(M, \omega)$ is real analytic K\"ahler. In Subsection \ref{Subsection 5.3}, we will construct a sheaf of left (resp. right) modules over the above deformation quantization to realize the composition in (\ref{Equation 2.1}) involving $\operatorname{Hom}(\mathcal{B}, \mathcal{B}_{\operatorname{cc}})$ (resp. in (\ref{Equation 2.2}) involving $\operatorname{Hom}(\mathcal{B}_{\operatorname{cc}}, \mathcal{B})$). In Subsection \ref{Subsection 5.4}, we will discuss the relationship among the above constructions and their counterparts for another K\"ahler manifold $(\overline{M}, -\omega)$. In Subsection \ref{Subsection 5.5}, we will construct a sheaf of bimodules on $M$ to realize the compositions in (\ref{Equation 2.3}) and (\ref{Equation 2.4}) involving $\operatorname{Hom}(\overline{\mathcal{B}}_{\operatorname{cc}}, \mathcal{B}_{\operatorname{cc}})$.

\subsection{A sheaf of algebras $\mathcal{O}_{\operatorname{qu}}^{(k)}$ realizing $\operatorname{Hom}(\mathcal{B}_{\operatorname{cc}}^{(k)}, \mathcal{B}_{\operatorname{cc}}^{(k)})$}
\label{Subsection 5.1}
\quad\par
We will first construct a particular deformation quantization using `polarized' structures on the Weyl bundle $\mathcal{W}$ of the K\"ahler manifold $(M, \omega)$. Then we will show how this gives rise to sheaves of twisted differential operators on $M$ (see Definition \ref{Definition 5.6}).\par
With respect to the underlying complex structure, we can define $\mathcal{W}^{1, 0} = \widehat{\operatorname{Sym}} T^{\vee (1, 0)}M [[\hbar]]$. There are $\mathcal{C}^\infty(M, \mathbb{C})[[\hbar]]$-linear operators $\delta^{1, 0}, (\delta^{1, 0})^{-1}, \pi_{0, *}$ on $\Omega^*(M, \mathcal{W})$ defined as follows: for a local section $a = w^{\mu_1} \cdots w^{\mu_l} \overline{w}^{\nu_1} \cdots \overline{w}^{\nu_m} dz^{\alpha_1} \wedge \cdots \wedge dz^{\alpha_p} \wedge d\overline{z}^{\beta_1} \wedge \cdots \wedge d\overline{z}^{\beta_q}$,
\begin{align*}
	\delta^{1, 0} a = dz^\mu \wedge \frac{\partial a}{\partial w^\mu},\quad
	(\delta^{1, 0})^{-1} a =
	\begin{cases}
		\frac{1}{l+p} w^\mu \iota_{\partial_{z^\mu}} a & \text{ if } l + p > 0;\\
		0 & \text{ if } l + p = 0,
	\end{cases}
	\quad \pi_{0, *}(a) =
	\begin{cases}
		0 & \text{ if } l + p > 0;\\
		a & \text{ if } l + p = 0,
	\end{cases}
\end{align*}
and the equality $\operatorname{Id} - \pi_{0, *} = \delta^{1, 0} \circ (\delta^{1, 0})^{-1} + (\delta^{1, 0})^{-1} \circ \delta^{1, 0}$ holds on $\Omega^*(M, \mathcal{W})$. Clearly, we can define the antiholomorphic counterparts $\mathcal{W}^{0, 1}$, $\delta^{0, 1}$, $(\delta^{0, 1})^{-1}$, $\pi_{*, 0}$ of $\mathcal{W}^{1, 0}$, $\delta^{1, 0}$, $(\delta^{1, 0})^{-1}$, $\pi_{0, *}$ respectively.\par
Instead of the fibrewise Moyal product $\star_{\operatorname{M}}$, we equip $\Omega^*(M, \mathcal{W})$ with the \emph{fibrewise anti-Wick product} $\star$ defined as follows \footnote{A better notation might be $\star_{\overline{\operatorname{W}}}$, but we prefer $\star$ for the ease of notations.} (c.f., for instance, \cite{BorWal1997, Neu2003}). For $a, b \in \Omega^*(M, \mathcal{W})$,
\begin{equation}
	a \star b := \sum_{r=0}^\infty \frac{\hbar^r}{r!} \omega^{\overline{\nu}_1\mu_1} \cdots \omega^{\overline{\nu}_r\mu_r} \frac{\partial^r a}{\partial \overline{w}^{\nu_1} \cdots \partial \overline{w}^{\nu_r}} \wedge \frac{\partial^r b}{\partial w^{\mu_1} \cdots \partial w^{\mu_r}}.
\end{equation}
It gives a natural ordering of holomorphic and anti-holomorphic variables for our construction of a deformation quantization which gives rise to sheaves of twisted differential operators. 

\begin{remark}
	In \cite{ChaLeuLi2023}, however, the fibrewise Wick product was used to construct a `sheafification' of the asymptotic action of the Berezin-Toeplitz deformation quantization of $(M, \omega)$ on the quantum Hilbert space. The deformation quantization to be constructed in this subsection is equivalent to the metaplectic corrected Berezin-Toeplitz deformation quantization. An equivalence between them is supposed to be given by the formal Berezin transform \cite{KarSch2001} with metaplectic correction.
\end{remark}

As mentioned in Subsection \ref{Subsection 3.2}, there exists a unique element $\tilde{\omega} \in \Omega^1(M, T^\vee M)$ such that $\delta = \tfrac{1}{\hbar}[\tilde{\omega}, \quad]_{\star_{\operatorname{M}}}$. It can now be expressed as $\tilde{\omega} = -(\delta^{0, 1})^{-1} \omega - (\delta^{1, 0})^{-1} \omega$, and it satisfies the condition $\delta = \tfrac{1}{\hbar}[\tilde{\omega}, \quad]_\star$, where $[\quad, \quad]_\star$ is the graded commutator of the fibrewise anti-Wick product. Also, there is a unique element $R \in \Omega^2(M, \operatorname{Sym}^2 T^\vee M)$ such that $\nabla^2 = \tfrac{1}{\hbar}[R, \quad]_{\star_{\operatorname{M}}}$ on $\Omega^*(M, \mathcal{W})$, which can be expressed locally in terms of K\"ahler coordinates: $R = -\omega_{\eta\overline{\nu}} R_{\alpha\overline{\beta}\mu}^\eta dz^\alpha \wedge d\overline{z}^\beta \otimes w^\mu \overline{w}^\nu$. We can show that $\nabla^2 = \tfrac{1}{\hbar}[R, \quad]_\star$ on $\Omega^*(M, \mathcal{W})$.\par
In order to obtain the desired deformation quantization of $(M, \omega)$, we define a closed $(1, 1)$-form
\begin{equation}
	\omega_1 = -\tfrac{\sqrt{-1}}{2} R_{\alpha\overline{\beta}\eta}^\eta dz^\alpha \wedge d\overline{z}^\beta.
\end{equation}

\begin{remark}
	\label{Remark 5.2}
	In local coordinates $(z^1, ..., z^n)$ defined on an open subset $U$ of $M$, we can define a local function $\rho_1 = -\tfrac{1}{2} \log \det (-\sqrt{-1} \omega_{\alpha\overline{\beta}})$, such that $\omega_1 \vert_U = -\sqrt{-1} \partial \overline{\partial} \rho_1$. By Jacobi's formula and basic identities in K\"ahler geometry,
	\begin{equation}
		\label{Equation 5.3}
		-2 \frac{\partial\rho_1}{\partial z^\alpha} = \omega^{\overline{\nu}\mu} \frac{\partial \omega_{\mu{\overline{\nu}}}}{\partial z^\alpha} = \omega^{\overline{\nu}\mu} \frac{\partial \omega_{\alpha{\overline{\nu}}}}{\partial z^\mu}.
	\end{equation}
\end{remark}

By (a variant of) Theorem 2.17 in \cite{ChaLeuLi2022b},
\begin{equation}
	\label{Equation 5.4}
	\gamma := (\delta^{0, 1})^{-1} \omega + (\delta^{1, 0})^{-1} \omega + A + \hbar B
\end{equation}
is a solution in $\Omega^1(M, \mathcal{W})$ to the equation
\begin{equation}
	R + \nabla \gamma + \tfrac{1}{2\hbar} [\gamma, \gamma]_\star = -\left( \omega + \tfrac{\hbar}{\sqrt{-1}} \omega_1 \right),
\end{equation}
where $A = \sum_{r=2}^\infty A_{(r)}$ and $B = \sum_{r=1}^\infty B_{(r)}$ with
\begin{itemize}
	\item $A_{(r)} = (\tilde{\nabla}^{1, 0})^{r-2} (\delta^{1, 0})^{-1}(R) \in \Omega^{0, 1}(M, \operatorname{Sym}^r T^{\vee (1, 0)}M \otimes T^{\vee (0, 1)}M)$ for $r \geq 2$; and
	\item $B_{(r)} = (\tilde{\nabla}^{1, 0})^{r-1} (\delta^{1, 0})^{-1}\left( \frac{1}{\sqrt{-1}} \omega_1 \right) \in \Omega^{0, 1}(M, \operatorname{Sym}^r T^{\vee (1, 0)}M)$ for $r \geq 1$.
\end{itemize}
In the above definitions, $\tilde{\nabla}^{1, 0} = (\delta^{1, 0})^{-1} \circ \nabla^{1, 0}$.

\begin{remark}
	For $r \geq 2$, we write $A_{(r)}$ locally as $A_{(r)} = \omega_{\eta\overline{\nu}} F_{\mu_1, ..., \mu_r, \overline{\beta}}^\eta d\overline{z}^\beta \otimes w^{\mu_1} \cdots w^{\mu_r} \overline{w}^\nu$. Then by Lemma B.1 in \cite{ChaLeuLi2021} (or Lemma A.1 in \cite{ChaLeuLi2023}), $2 B_{(r-1)} = r F_{\eta, \mu_2, ..., \mu_r, \overline{\beta}}^\eta d\overline{z}^\beta \otimes w^{\mu_2} \cdots w^{\mu_r}$.
\end{remark}

The solution $\gamma$ deforms $\nabla$ to a flat connection
\begin{equation}
	D := \nabla + \frac{1}{\hbar} [\gamma, \quad]_\star = \nabla - \delta + \frac{1}{\hbar} [A + \hbar B, \quad]_\star
\end{equation}
on $\mathcal{W}$. Functions $f \in \mathcal{C}^\infty(M, \mathbb{C})$ are in bijection with $D$-flat sections $O_f$ of $\mathcal{W}$, where $O_f$ is uniquely determined by the condition that $\pi_0(O_f) = f$. Via this bijection, the fibrewise anti-Wick product induces a star product with separation of variables (in the sense of \cite{Kar1996}) on $\mathcal{C}^\infty(M, \mathbb{C})[[\hbar]]$, still denoted by $\star$ by abuse of notations.

\begin{example}
	\label{Example 5.4}
	We will state explicit formulae of $O_f$ for two important kinds of (local) functions.
	\begin{enumerate}
		\item Let $f$ be a holomorphic function on $M$. Then by Example 2.14 in \cite{ChaLeuLi2023},
		\begin{equation}
			O_f = \sum_{r=0}^\infty (\tilde{\nabla}^{1, 0})^r f.
		\end{equation}
		Hence, $O_f$ is a section of $\mathcal{W}_{\operatorname{cl}}^{1, 0} := \widehat{\operatorname{Sym}} T^{\vee (1, 0)}M$.
		\item Suppose $U$ is an open subset of $M$ on which local coordinates $(z^1, ..., z^n)$ are defined and $\omega$ admits a potential $\rho_0$, i.e. $\omega \vert_U = -\sqrt{-1} \partial \overline{\partial} \rho_0$. Define $\rho_1$ as in Remark \ref{Remark 5.2} and define
		\begin{align*}
			\rho = -\sqrt{-1} \left( \rho_0 + \tfrac{\hbar}{\sqrt{-1}} \rho_1 \right).
		\end{align*}
		Then by the proof of Proposition 2.15 in \cite{ChaLeuLi2023}, for a holomorphic vector field $\xi$ on $U$,
		\begin{equation}
			O_{\xi(\rho)} = \sum_{r=0}^\infty (\tilde{\nabla}^{1, 0})^r (\xi(\rho) + \widehat{\xi}),
		\end{equation}
		where $\widehat{\xi} := (\delta^{0, 1})^{-1}(\iota_\xi\omega))$. Hence, $O_{\xi(\rho)}$ is a local section of $\mathcal{W}_{\operatorname{cl}}^{1, 0} \otimes \operatorname{Sym}^{\leq 1} T^{\vee (0, 1)}M[\hbar]$.
	\end{enumerate}
\end{example}

Thus, in Example \ref{Example 5.4}, $O_f$ and $O_{\xi(\rho)}$ are (local) sections of the algebra subbundle
\begin{equation*}
	\underline{\mathcal{W}} := \mathcal{W}_{\operatorname{cl}}^{1, 0} \otimes \operatorname{Sym} T^{\vee (0, 1)}M [\hbar]
\end{equation*}
of $(\mathcal{W}, \star)$. Note that $\gamma \in \Omega^1(M, \underline{\mathcal{W}})$ and $[\quad, \quad]_\star$ preserves $\underline{\mathcal{W}}$. This enables us to evaluate $D$ at $\hbar = \tfrac{\sqrt{-1}}{k}$ for any non-zero complex number $k$ without convergence issues and obtain a non-formal flat connection $D_k = \nabla + \frac{k}{\sqrt{-1}} [\gamma_k, \quad]_{\star_k}$ on $\underline{\mathcal{W}}_{\operatorname{cl}} := \widehat{\operatorname{Sym}} T^{\vee (1, 0)}M \otimes \operatorname{Sym} T^{\vee (0, 1)}M$, where we add a subscript $k$ in a symbol to denote its evaluation at $\hbar = \tfrac{\sqrt{-1}}{k}$, so that $(\Omega^*(M, \underline{\mathcal{W}}_{\operatorname{cl}}), \star_k, D_k)$ forms a dga. We can now define quantizable functions.

\begin{definition}[Definition 2.20 in \cite{ChaLeuLi2023}]
	Let $k \in \mathbb{C}$ be non-zero. A \emph{(non-formal) quantizable function of level} $k$ is a $D_k$-flat section of $\underline{\mathcal{W}}_{\operatorname{cl}}$.
\end{definition}

The sheaf of quantizable functions of level $k$, denote by $\mathcal{O}_{\operatorname{qu}}^{(k)}$, is closed under $\star_k$. For a $D$-flat section $O$ of $\underline{\mathcal{W}}$, we denote by $\operatorname{ev}_k(O)$ its evaluation at $\hbar = \tfrac{\sqrt{-1}}{k}$. Clearly, $\operatorname{ev}_k(O)$ is $D_k$-flat.\par
For an open subset $U$ of $M$, let $\mathcal{O}_U$ be the sheaf of holomorphic functions on $U$, $\mathcal{D}_U$ be the sheaf of holomorphic differential operators on $U$ and $\iota_U: \mathcal{O}_U \hookrightarrow \mathcal{D}_U$ be the canonical embedding. We introduce the notion of sheaves of twisted differential operators, following Mili\v{c}i\'{c}'s notes \cite{Mil1994}. 

\begin{definition}
	\label{Definition 5.6}
	A \emph{sheaf of twisted differential operators} (\emph{TDO} in short) on $M$ is a pair $(\mathcal{D}, \iota)$, where $\mathcal{D}$ is a sheaf of $\mathbb{C}$-algebras with identity on $M$ and $\iota: \mathcal{O}_X \hookrightarrow \mathcal{D}$ is a morphism of sheaves of $\mathbb{C}$-algebras with identity, for which $M$ admits an open cover $\{U_i\}_{i \in I}$ such that for all $i \in I$, there is an isomorphism of sheaves of $\mathbb{C}$-algebras with identity $\psi_i: \mathcal{D} \vert_{U_i} \to \mathcal{D}_{U_i}$ such that $\psi_i \circ ( \iota \vert_{U_i} ) = \iota_{U_i}$.
\end{definition}

There is a morphism of $\mathbb{C}$-algebras with identity $\iota: \mathcal{O}_M \to \mathcal{O}_{\operatorname{qu}}^{(k)}$ given by $f \mapsto O_f$.

\begin{theorem}[Theorem 3.6 in \cite{ChaLeuLi2023}]
	\label{Theorem 5.7}
	Let $k \in \mathbb{C}$ be non-zero. Then the pair $(\mathcal{O}_{\operatorname{qu}}^{(k)}, \iota)$ forms a TDO on the underlying complex manifold $M$. 
\end{theorem}

Note that the definition of TDOs in \cite{ChaLeuLi2023}, which follows Ginzburg's note \cite{Gin1998}, is slightly different from Definition \ref{Definition 5.6}. Our notion of TDOs given by Definition \ref{Definition 5.6} is known as \emph{locally trivial} TDOs in \cite{Gin1998}. Theorem \ref{Theorem 5.7} is still valid in spite of this slight difference of definitions. We will not duplicate the complete proof of Theorem 3.6 from \cite{ChaLeuLi2023}, but will only outline how to construct local isomorphisms between $\mathcal{O}_{\operatorname{qu}}^{(k)}$ and $\mathcal{D}_X$ and state the characteristic class of $(\mathcal{O}_{\operatorname{qu}}^{(k)}, \iota)$.

\begin{lemma}
	Let $U$ be the open subset of $M$ and $\rho$ be the function in Example \ref{Example 5.4}.
	\begin{enumerate}
		\item Let $f$ be a holomorphic function and $\xi$ be a holomorphic vector field on $U$. Then
		\begin{equation*}
			\tfrac{1}{\hbar} \left[ O_{\xi(\rho)}, O_f \right]_\star = O_{\xi(f)}.
		\end{equation*}
		\item Let $\xi_1, \xi_2$ be holomorphic vector fields on $U$. Then
		\begin{equation*}
			\tfrac{1}{\hbar} [O_{\xi_1(\rho)}, O_{\xi_2(\rho)}]_\star = O_{[\xi_1, \xi_2](\rho)}.
		\end{equation*}
	\end{enumerate}
\end{lemma}
\begin{proof}
	We only need to compute $\pi_0\left( \tfrac{1}{\hbar} \left[ O_{\xi(\rho)}, O_f \right]_\star \right)$ and $\pi_0\left( \tfrac{1}{\hbar} [O_{\xi_1(\rho)}, O_{\xi_2(\rho)}]_\star \right)$.
	\begin{enumerate}
		\item Write $\xi = \xi^\mu \partial_{z^\mu}$. Then $\widehat{\xi} = \xi^\eta \omega_{\eta\overline{\nu}} \overline{w}^\nu$, whence
		\begin{align*}
			\pi_0\left( \tfrac{1}{\hbar} \left[ O_{\xi(\rho)}, O_f \right]_\star \right) = \omega^{\overline{\nu}\mu} \frac{\partial \widehat{\xi}}{\partial \overline{w}^\nu} \frac{\partial \tilde{\nabla}^{1, 0} f}{\partial w^\mu} = \omega^{\overline{\nu}\mu} \xi^\eta \omega_{\eta\overline{\nu}} \frac{\partial f}{\partial z^\mu} = \xi^\mu \frac{\partial f}{\partial z^\mu} = \xi(f).
		\end{align*}
		Therefore, $\tfrac{1}{\hbar} \left[ O_{\xi(\rho)}, O_f \right]_\star = O_{\xi(f)}$.
		\item We see that $\pi_0\left( O_{\xi_1(\rho)} \star O_{\xi_2(\rho)} \right)$ is equal to
		\begin{equation*}
			 \xi_1(\rho) \xi_2(\rho) + \hbar \omega^{\overline{\nu}\mu} \frac{\partial \widehat{\xi_1}}{\partial \overline{w}^\nu} \frac{\partial \tilde{\nabla}^{1, 0} ( \xi_2(\rho) )}{\partial w^\mu} = \xi_1(\rho) \xi_2(\rho) + \hbar \xi_1(\xi_2(\rho)).
		\end{equation*}
		Therefore, $[O_{\xi_1(\rho)}, O_{\xi_2(\rho)}]_\star = \hbar O_{[\xi_1, \xi_2](\rho)}$.
	\end{enumerate}
\end{proof}

Then we can well define a morphism of sheaves of $\mathbb{C}$-algebras with identity $\tilde{\psi}: \mathcal{D}_U \to \mathcal{O}_{\operatorname{qu}}^{(k)} \vert_U$ compatible with $\iota$ as follows. Denote by $\circ$ the multiplication on $\mathcal{D}_U$. For a holomorphic function $f$ on $U$ and holomorphic vector fields $\xi_1, ..., \xi_n$ on $U$, we define
\begin{equation*}
	\tilde{\psi}( f \circ \xi_1 \circ \cdots \circ \xi_n ) = \operatorname{ev}_k \left( \hbar^{-n} O_f \star O_{\xi_1(\rho)} \star \cdots \star O_{\xi_n(\rho)} \right).
\end{equation*}
It is indeed an isomorphism $\mathcal{D}_U \cong \mathcal{O}_{\operatorname{qu}}^{(k)} \vert_U$.\par
Recall from Theorem 1.3 in \cite{Mil1994} that there is a canonical bijection between isomorphism classes of TDOs on $M$ and elements of the first sheaf cohomology $H^1(M, \Omega_{M, \operatorname{cl}}^1)$ of $M$ with coefficients in the sheaf of $\partial$-closed holomorphic $(1, 0)$-forms on $M$. For a TDO $\mathcal{D}$ on $M$, the element in $H^1(M, \Omega_{M, \operatorname{cl}}^1)$ associated to the isomorphism class $[\mathcal{D}]$ is called the \emph{characteristic class} of $\mathcal{D}$. We can show that the characteristic class of the TDO $\mathcal{O}_{\operatorname{qu}}^{(k)}$ is represented by the \v{C}ech cocycle
\begin{equation}
	\omega_{ij}^{\mathcal{O}_{\operatorname{qu}}^{(k)}} = \partial (k \rho_{0, j} + \rho_{1, j} - k \rho_{0, i} - \rho_{1, i}),
\end{equation}
for an open cover $\{U_i\}_{i \in I}$ on $M$ with respective potentials $\rho_{0, i}, \rho_{1, i}$ of $\omega, \omega_1$ on $U_i$ for all $i \in I$.\par
We end this subsection by a discussion on Fedosov's characteristic class of $(\mathcal{C}^\infty(M, \mathbb{C})[[\hbar]], \star)$ and the extension problem for $(\mathcal{C}^\infty(M, \mathbb{C})[[\hbar]], \star)$. One way to find Fedosov's characteristic class of $(\mathcal{C}^\infty(M, \mathbb{C})[[\hbar]], \star)$ is to compute the corresponding Karabegov form \cite{Kar1998}. Alternatively, it was shown in \cite{BorWal1997, NeuWal2003} that the operator
\begin{equation}
	\tilde{\Delta} = \omega^{\overline{\nu}\mu} \frac{\partial^2}{\partial w^\mu \partial \overline{w}^\nu}
\end{equation}
is globally defined on $\Omega^*(M, \mathcal{W})$ and commutes with $\nabla$ and $\delta$. Moreover,
\begin{equation*}
	e^{-\frac{\hbar}{2} \tilde{\Delta}}: (\Omega^*(M, \mathcal{W}), \star) \to (\Omega^*(M, \mathcal{W}), \star_{\operatorname{M}})
\end{equation*}
is an isomorphism of algebra bundles. Observe that $\tilde{\Delta} R = - R_{\alpha\overline{\beta}\eta}^\eta dz^\alpha \wedge d\overline{z}^\beta = \tfrac{2}{\sqrt{-1}} \omega_1$ and for $r \geq 2$, $\tilde{\Delta} A_{(r)} = r \omega^{\overline{\nu}\mu_1} \omega_{\eta\overline{\nu}} F_{\mu_1, ..., \mu_r, \overline{\beta}}^\eta d\overline{z}^\beta \otimes w^{\mu_1} \cdots w^{\mu_r} = 2B_{(r-1)}$. Therefore,
\begin{equation*}
	e^{-\frac{\hbar}{2} \tilde{\Delta}} ( A + \hbar B ) = A \quad \text{and} \quad e^{-\frac{\hbar}{2} \tilde{\Delta}} R = R - \tfrac{\hbar}{\sqrt{-1}} \omega_1.
\end{equation*}
We then see that $\gamma_{\operatorname{M}} := -\tilde{\omega} + A$ is a solution to (\ref{Equation 3.2}) with $\omega_\hbar = 0$, and $\pi_0(s) = 0$, where $s := \delta^{-1} A \in \Omega^0(M, \mathcal{W}_{(4)})$ is in general non-zero. Denote by the same symbol $\star_{\operatorname{M}}$ the star product on $\mathcal{C}^\infty(M, \mathbb{C})[[\hbar]]$ induced by the fibrewise Moyal product with respect to the solution $\gamma_{\operatorname{M}}$ to (\ref{Equation 3.2}). Since $e^{-\frac{\hbar}{2} \tilde{\Delta}} \circ D = D_{\operatorname{M}} \circ e^{-\frac{\hbar}{2} \tilde{\Delta}}$, where $D_{\operatorname{M}} := \nabla + \tfrac{1}{\hbar}[\gamma_{\operatorname{M}}, \quad]_{\star_{\operatorname{M}}}$, $e^{-\frac{\hbar}{2} \tilde{\Delta}}$ induces an equivalence
\begin{equation*}
	(\mathcal{C}^\infty(M, \mathbb{C})[[\hbar]], \star) \to (\mathcal{C}^\infty(M, \mathbb{C})[[\hbar]], \star_{\operatorname{M}})
\end{equation*}
of deformation quantizations. In other words, the deformation quantization $(\mathcal{C}^\infty(M, \mathbb{C})[[\hbar]], \star)$ has Fedosov's characteristic class $-\tfrac{1}{\hbar}[\omega]$, and hence has an extension $(\mathcal{A}, \psi)$ on a complexification $(X, \Omega)$ of $(M, \omega)$ in the sense of Theorem \ref{Theorem 1.2}.

\subsection{Extension of $(\mathcal{C}^\infty(M, \mathbb{C})[[\hbar]], \star)$ in the real analytic K\"ahler case}
\label{Subsection 5.2}
\quad\par
Throughout this subsection, let $(M, \omega)$ be real analytic K\"ahler. As we have seen from Example \ref{Example 4.5}, we can pick a Stein complexification $\iota: (M, \omega) \hookrightarrow (X, \Omega)$ such that $X$ is a neighbourhood of the diagonal in $M \times \overline{M}$ and $\iota$ is the diagonal map. The kernel $\mathcal{P}^{\operatorname{c}}$ of $d\check{\pi}: T^{1, 0}X \to T^{0, 1}M$ (resp. the kernel $\check{\mathcal{P}}^{\operatorname{c}}: d\pi: T^{1, 0}X \to T^{1, 0}M$) is a holomorphic extension of $T^{1, 0}M$ (resp. $T^{0, 1}M$). Here, $\pi: X \to M$ and $\check{\pi}: X \to \overline{M}$ are the canonical projections.\par
We will show that we can extend $\star$ to a holomorphic star product $\widetilde{\star}$ which is `with separation of variables’ with respect to the decomposition $T^{1, 0}X = \mathcal{P}^{\operatorname{c}} \oplus \check{\mathcal{P}}^{\operatorname{c}}$. As the construction of $\widetilde{\star}$ follows almost verbatim from that of $\star$, we will only outline the main steps.\par
About every point in $X$, there are complex coordinates of the form
\begin{equation*}
	(z^1, ..., z^n, \check{z}^1, ..., \check{z}^n) = (\pi^*z_{(1)}^1, ..., \pi^*z_{(1)}^n, \check{\pi}^*\overline{z}_{(2)}^1, ..., \check{\pi}^*\overline{z}_{(2)}^n),
\end{equation*}
where $(z_{(1)}^1, ..., z_{(1)}^n)$ and $(z_{(2)}^1, ..., z_{(2)}^n)$ are complex coordinates on $M$. Then $( \tfrac{\partial}{\partial z^1}, ..., \tfrac{\partial}{\partial z^n} )$ and $( \tfrac{\partial}{\partial \check{z}^1}, ..., \tfrac{\partial}{\partial \check{z}^n} )$ are local holomorphic frames of $\mathcal{P}^{\operatorname{c}}$ and $\check{\mathcal{P}}^{\operatorname{c}}$ respectively. As $\mathcal{P}^{\operatorname{c}}, \check{\mathcal{P}}^{\operatorname{c}}$ are $\Omega$-Lagrangian and $T^{1, 0}X = \mathcal{P}^{\operatorname{c}} \oplus \check{\mathcal{P}}^{\operatorname{c}}$, we can write $\Omega = \Omega_{\alpha\check{\beta}} dz^\alpha \wedge d\check{z}^\beta$ locally and $(\Omega_{\alpha\check{\beta}})$ has an inverse $(\Omega^{\check{\alpha}\beta})$.\par
Define a fibrewise holomorphic star product $\widetilde{\star}$ on $\Omega_{\operatorname{hol}}^{*, 0}(X, \widetilde{\mathcal{W}})$ as follows. For $a, b \in \Omega_{\operatorname{hol}}^{*, 0}(X, \widetilde{\mathcal{W}})$,
\begin{equation*}
	a \widetilde{\star} b := \sum_{r=0}^\infty \frac{\hbar}{r!} \Omega^{\check{\nu}_1\mu_1} \cdots \Omega^{\check{\nu}_r\mu_r} \frac{\partial^r a}{\partial \check{w}^{\nu_1} \cdots \partial \check{w}^{\nu_r}} \wedge \frac{\partial^r b}{\partial w^{\mu_1} \cdots \partial w^{\mu_r}}.
\end{equation*}
Here, $w^\mu$ (resp. $\check{w}^\nu$) denotes $dz^\mu$ (resp. $d\check{z}^\nu$) as a local section of $\widetilde{\mathcal{W}} := \widehat{\operatorname{Sym}} T^{\vee (1, 0)}X[[\hbar]]$.\par
Note that $\nabla, R, \omega_1$ are all real analytic. Indeed, $X$ admits holomorphic extensions $\widetilde{\nabla}$, $\widetilde{R}$, $\Omega_1$ of $\nabla$, $R$, $\omega_1$ respectively such that $\widetilde{\nabla}$ is a holomorphic symplectic connection on $(X, \Omega)$ preserving $\mathcal{P}^{\operatorname{c}}$ and $\check{\mathcal{P}}^{\operatorname{c}}$ such that it restricts to holomorphic connections $\widetilde{\nabla}_{(1)}, \widetilde{\nabla}_{(2)}$ on $\mathcal{P}^{\operatorname{c}}, \check{\mathcal{P}}^{\operatorname{c}}$ respectively. Moreover, $\widetilde{\nabla}_{(1)}, \widetilde{\nabla}_{(2)}$ square to zero. In local coordinates, $\widetilde{\nabla}$ is given by
\begin{equation*}
	\widetilde{\nabla}\left( \frac{\partial}{\partial z^\beta} \right) = \widetilde{\Gamma}_{\alpha\beta}^\gamma dz^\alpha \otimes \frac{\partial}{\partial z^\gamma} \quad \text{and} \quad \widetilde{\nabla}\left( \frac{\partial}{\partial \check{z}^\beta} \right) = \widetilde{\Gamma}_{\check{\alpha}\check{\beta}}^{\check{\gamma}} d\check{z}^\alpha \otimes \frac{\partial}{\partial \check{z}^\gamma},
\end{equation*}
where $\widetilde{\Gamma}_{\alpha\beta}^\gamma = \Omega^{\check{\lambda}\gamma} \tfrac{\partial \Omega_{\beta\check{\lambda}}}{\partial z^\alpha}$ and $\widetilde{\Gamma}_{\check{\alpha}\check{\beta}}^{\check{\gamma}} = \Omega^{\check{\gamma}\lambda} \tfrac{\partial \Omega_{\lambda\check{\beta}}}{\partial \check{z}^\alpha}$. Also,
\begin{align*}
	\widetilde{R} = & -\Omega_{\eta\check{\nu}} \widetilde{R}_{\alpha\check{\beta}\mu}^\eta dz^\alpha \wedge d\check{z}^\beta \otimes w^\mu \check{w}^\nu,\\
	\Omega_1 = & -\tfrac{\sqrt{-1}}{2} \widetilde{R}_{\alpha\check{\beta}\eta}^\eta dz^\alpha \wedge d\check{z}^\beta,
\end{align*}
where $\widetilde{R}_{\alpha\check{\beta}\mu}^\nu = -\frac{\partial \widetilde{\Gamma}_{\alpha\mu}^\nu}{\partial \check{z}^\beta}$. We can easily check that $(\Omega_{\operatorname{hol}}^{*, 0}(X, \widetilde{\mathcal{W}}), \widetilde{R}, \widetilde{\nabla}, \widetilde{\star})$ is a curved dga. Recall that the decomposition $TM_\mathbb{C} = T^{1, 0}M \oplus T^{0, 1}M$ yields the operators $(\delta^{1, 0})^{-1}, (\delta^{0, 1})^{-1}$ on $\Omega^*(M, \mathcal{W})$. We can similarly define operators $\widetilde{\delta}_{(1)}^{-1}, \widetilde{\delta}_{(2)}^{-1}$ on $\Omega_{\operatorname{hol}}^{*, 0}(X, \widetilde{\mathcal{W}})$ with respect to the decomposition $T^{1, 0}X = \mathcal{P}^{\operatorname{c}} \oplus \check{\mathcal{P}}^{\operatorname{c}}$. From the formulae of $A$ and $B$ in Subsection \ref{Subsection 5.1}, it is not hard to see that we can construct their holomorphic extensions
\begin{equation*}
	\widetilde{A} \in \Gamma_{\operatorname{hol}}(X, (\mathcal{P}^{\operatorname{c}})^\vee \otimes \widehat{\operatorname{Sym}} (\check{\mathcal{P}}^{\operatorname{c}})^\vee \otimes (\mathcal{P}^{\operatorname{c}})^\vee) \quad \text{and} \quad \widetilde{B} \in \Gamma_{\operatorname{hol}}(X, (\mathcal{P}^{\operatorname{c}})^\vee \otimes \widehat{\operatorname{Sym}} (\check{\mathcal{P}}^{\operatorname{c}})^\vee)
\end{equation*}
via the operators $\widetilde{\delta}_{(1)}^{-1}, \widetilde{\nabla}_{(1)}$ and the holomorphic sections $\widetilde{R}, \Omega_1$. It implies that
\begin{equation*}
	\widetilde{\gamma} := \widetilde{\delta}_{(2)}^{-1}\Omega + \widetilde{\delta}_{(1)}^{-1}\Omega + \widetilde{A} + \hbar \widetilde{B} \in \Omega_{\operatorname{hol}}^{1, 0}(X, \widetilde{\mathcal{W}})
\end{equation*}
is a holomorphic extension of $\gamma$ given as in (\ref{Equation 5.4}) and $\widetilde{\gamma}$ has a polynomial growth in $\check{\mathcal{P}}^{\operatorname{c}}$. We then obtain a curved dga $(\Omega_{\operatorname{hol}}^{*, 0}(X, \widetilde{\mathcal{W}}), -(\Omega + \tfrac{\hbar}{\sqrt{-1}} \Omega_1), \widetilde{D}^{1, 0}, \widetilde{\star})$, where $\widetilde{D}^{1, 0} := \widetilde{\nabla} + \tfrac{1}{\hbar} [\widetilde{\gamma}, \quad]_{\widetilde{\star}}$. In particular, $(\widetilde{D}^{1, 0})^2 = 0$. Moreover, the following is a morphism of curved dgas:
\begin{equation*}
	\iota^*: (\Omega_{\operatorname{hol}}^{*, 0}(X, \widetilde{\mathcal{W}}), -(\Omega + \tfrac{\hbar}{\sqrt{-1}} \Omega_1), \widetilde{D}^{1, 0}, \widetilde{\star}) \to (\Omega^*(M, \mathcal{W}), -(\omega + \tfrac{\hbar}{\sqrt{-1}} \omega_1), D, \star).
\end{equation*}
For any open subset $U$ of $X$ and $f \in \mathcal{O}_X(U)[[\hbar]]$, there is a unique $\widetilde{D}^{1, 0}$-flat holomorphic section $\widetilde{O}_f$ of $\widetilde{\mathcal{W}}$ over $U$ such that $\widetilde{O}_f = f - \widetilde{\delta}^{-1} ( \widetilde{D}^{1, 0} + \widetilde{\delta} ) f$, where $\widetilde{\delta}, \widetilde{\delta}^{-1}$ are given as in Subsection \ref{Subsection 3.4}. Indeed, $f \mapsto \widetilde{O}_f$ defines an isomorphism from $\mathcal{O}_X[[\hbar]]$ to the sheaf of $\widetilde{D}^{1, 0}$-flat holomorphic sections of $\widetilde{\mathcal{W}}$. In addition, $\widetilde{O}_f$ is a holomorphic extension of $O_{f \vert_{M \cap U}}$. Then $\widetilde{\star}$ descends to a holomorphic star product on $\mathcal{O}_X[[\hbar]]$, denoted by the same symbol, and the following proposition holds.

\begin{proposition}
	The following is a morphism of sheaves of $\mathbb{C}[[\hbar]]$-algebras:
	\begin{equation*}
		\iota^*: (\mathcal{O}_X[[\hbar]], \widetilde{\star}) \to \iota_*(\mathcal{C}_M^\infty[[\hbar]], \star).
	\end{equation*}
\end{proposition}

For every non-zero $k \in \mathbb{C}$, denote by $\mathcal{A}_{\operatorname{qu}}^{(k)}$ the sheaf of $\widetilde{D}_k^{1, 0}$-flat holomorphic sections of $\widehat{\operatorname{Sym}} (\mathcal{P}^{\operatorname{c}})^\vee \otimes \operatorname{Sym} (\check{\mathcal{P}})^\vee$, where $\widetilde{D}_k^{1, 0}$ is the evaluation of $\widetilde{D}^{1, 0}$ at $\hbar = \tfrac{\sqrt{-1}}{k}$.

\begin{corollary}
	Let $k \in \mathbb{Z}^+$. Then the following is a morphism of sheaves of $\mathbb{C}$-algebras:
	\begin{equation*}
		\iota^*: \mathcal{A}_{\operatorname{qu}}^{(k)} \to \iota_*\mathcal{O}_{\operatorname{qu}}^{(k)}.
	\end{equation*}
\end{corollary}

\subsection{Sheaves of modules $\mathcal{L}^{(k)}$ and $\mathcal{L}^{(-k)}$ realizing $\operatorname{Hom}(\mathcal{B}^{(k)}, \mathcal{B}_{\operatorname{cc}}^{(k)})$ and $\operatorname{Hom}(\mathcal{B}_{\operatorname{cc}}^{(k)}, \mathcal{B}^{(k)})$}
\label{Subsection 5.3}
\quad\par
Let the K\"ahler manifold $(M, \omega)$ admit a prequantum line bundle $(L, \nabla^L)$ and a square root $\sqrt{K}$ of its canonical bundle. Let $k \in \mathbb{Z}^+$. In this subsection, we recall one of Chan-Leung-Li's theorems.

\begin{theorem}[a variant of Theorem 4.7 in \cite{ChaLeuLi2023}]
	\label{Theorem 5.9}
	The sheaf $\mathcal{L}^{(k)}$ of holomorphic sections of $L^{(k)} := L^{\otimes k} \otimes \sqrt{K}$ has a left $\mathcal{O}_{\operatorname{qu}}^{(k)}$-module structure giving an isomorphism of TDOs between $\mathcal{O}_{\operatorname{qu}}^{(k)}$ and the sheaf $\mathcal{D}_{L^{(k)}}$ of holomorphic differential operators of $L^{(k)}$. 
\end{theorem}

An immediate consequence of the above theorem is the following corollary.

\begin{corollary}
	\label{Corollary 5.10}
	The sheaf $\mathcal{L}^{(-k)}$ of holomorphic sections of $L^{(-k)} := (L^\vee)^{\otimes k} \otimes \sqrt{K}$ has a right $\mathcal{O}_{\operatorname{qu}}^{(k)}$-module structure giving an isomorphism  of TDOs between $\mathcal{O}_{\operatorname{qu}}^{(k)}$ and the opposite algebra $\mathcal{D}_{L^{(-k)}}^\circ$ of $\mathcal{D}_{L^{(-k)}}$, where $\mathcal{D}_{L^{(-k)}}$ is the sheaf of holomorphic differential operators of $L^{(-k)}$.
\end{corollary}

The proof of Theorem 4.7 in \cite{ChaLeuLi2023} essentially applies to its variant, Theorem \ref{Theorem 5.9}, in spite of the use of a star product of anti-Wick type and metaplectic correction adopted in this paper. For clarity, however, we will provide explicit constructions of the module structures and outline the arguments showing that these induce isomorphisms $\mathcal{O}_{\operatorname{qu}}^{(k)} \cong \mathcal{D}_{L^{(k)}}$ and $\mathcal{O}_{\operatorname{qu}}^{(k)} \cong \mathcal{D}_{L^{(-k)}}^\circ$. The starting point is to define fibrewise left and right module structures.

\begin{definition}
	Define two $\mathcal{C}^\infty(M, \mathbb{C})[[\hbar]]$-bilinear maps
	\begin{align*}
		\circledast^\pm: & \Omega^*(M, \underline{\mathcal{W}}) \times \Omega^*(M, \mathcal{W}) \to \Omega^*(M, \mathcal{W})
	\end{align*}
	as follows. For $\alpha, \sigma \in \Omega^*(M)$, $a = w^{\mu_1} \cdots w^{\mu_p} \overline{w}^{\nu_1} \cdots \overline{w}^{\nu_q} \in \Gamma(M, \underline{\mathcal{W}})$ and $s \in \Gamma(M, \mathcal{W})$,
	\begin{align*}
		(\alpha \otimes a) \circledast^+ (\sigma \otimes s) := & \alpha \wedge \sigma \otimes \hbar^q \omega^{\overline{\nu}_1\lambda_1} \cdots \omega^{\overline{\nu}_q\lambda_q} w^{\mu_1} \cdots w^{\mu_p} \frac{\partial^q}{\partial w^{\lambda_1} \cdots \partial w^{\lambda_q}} (s),\\
		(\alpha \otimes a) \circledast^- (\sigma \otimes s) := & \alpha \wedge \sigma \otimes (-\hbar)^q \omega^{\overline{\nu}_1\lambda_1} \cdots \omega^{\overline{\nu}_q\lambda_q} \frac{\partial^q}{\partial w^{\lambda_1} \cdots \partial w^{\lambda_q}} ( w^{\mu_1} \cdots w^{\mu_p} s).
	\end{align*}
\end{definition}

Note that for $s \in \Omega^*(M, \mathcal{W})$ and homogeneous $a, v \in \Omega^*(M, \underline{\mathcal{W}})$,
\begin{align*}
	a \circledast^+ (b \circledast^+ s) = & (a \star b) \circledast^+ s = [a, b]_\star \circledast^+ s + (-1)^{\lvert a \rvert \lvert b \rvert} b \circledast^+ (a \circledast^+ s),\\
	a \circledast^- (b \circledast^- s) = & (-1)^{\lvert a \rvert \lvert b \rvert} (b \star a) \circledast^- s = -[a, b]_\star \circledast^- s + (-1)^{\lvert a \rvert \lvert b \rvert} b \circledast^- (a \circledast^- s).
\end{align*}

Clearly, $\circledast^\pm$ restricts to an operation $\circledast^\pm: \Omega^*(M, \underline{\mathcal{W}}) \times \Omega^*(M, \mathcal{W}^{1, 0}) \to \Omega^*(M, \mathcal{W}^{1, 0})$. We can express the curvature of $\nabla$ on $\mathcal{W}^{1, 0}$ in terms of $\circledast^\pm$.

\begin{lemma}
	\label{Lemma 5.12}
	For $s \in \Omega^*(M, \mathcal{W}^{1, 0})$, we have
	\begin{equation*}
		\hbar \nabla^2 s = \pm \left( R - \sqrt{-1} \hbar \omega_1 \right) \circledast^\pm s + \sqrt{-1}\hbar \omega_1 \wedge s.
	\end{equation*}
\end{lemma}
\begin{proof}
	Note that $\omega_1 \circledast^\pm s = \omega_1 \wedge s$. Recall that $R = -\omega_{\eta\overline{\nu}} R_{\alpha\overline{\beta}\mu}^\eta dz^\alpha \wedge d\overline{z}^\beta \otimes w^\mu \overline{w}^\nu$. Then
	\begin{align*}
		R \circledast^+ s = & -\hbar \omega^{\overline{\nu}\lambda} \omega_{\eta\overline{\nu}} R_{\alpha\overline{\beta}\mu}^\eta dz^\alpha \wedge d\overline{z}^\beta \wedge w^\mu \frac{\partial}{\partial w^\lambda} s = -\hbar R_{\alpha\overline{\beta}\mu}^\nu dz^\alpha \wedge d\overline{z}^\beta \wedge w^\mu \frac{\partial}{\partial w^\nu} s = \hbar \nabla^2 s,\\
		R \circledast^- s = & \hbar \omega^{\overline{\nu}\lambda} \omega_{\eta\overline{\nu}} R_{\alpha\overline{\beta}\mu}^\eta dz^\alpha \wedge d\overline{z}^\beta \wedge \frac{\partial}{\partial w^\lambda} \left( w^\mu  \cdot s \right)\\
		= & \hbar \left( R_{\alpha\overline{\beta}\eta}^\eta dz^\alpha \wedge d\overline{z}^\beta \wedge s + R_{\alpha\overline{\beta}\mu}^\nu dz^\alpha \wedge d\overline{z}^\beta \wedge w^\mu \frac{\partial}{\partial w^\nu} s \right)\\
		= & \hbar \left( \sqrt{-1} \omega_1 \circledast^- s + \sqrt{-1} \omega_1 \wedge s - \nabla^2s \right).
	\end{align*}
\end{proof}
Let $\nabla^{L^{(\pm k)}}$ be the unitary connection on $L^{(\pm k)}$ induced by $\nabla^L$ and the Chern connection of $\sqrt{K}$. Define a connection $D_{\pm, k}^L$ on $\mathcal{W}_{\operatorname{cl}}^{1, 0} \otimes L^{(\pm k)}$ by
\begin{equation}
	D_{\pm, k}^L := (\nabla \pm \tfrac{k}{\sqrt{-1}} \gamma_k \circledast_k^\pm) \otimes \operatorname{Id} + \operatorname{Id} \otimes \nabla^{L^{(\pm k)}}.
\end{equation}

We first prove two basic properties of $D_{\pm, k}^L$.

\begin{proposition}
	\label{Proposition 5.13}
	The connection $D_{\pm, k}^L$ is flat.
\end{proposition}
\begin{proof}
	Set $\hbar = \tfrac{\sqrt{-1}}{k}$. For the ease of notations, simply write $\gamma$, $\star$, $\circledast^\pm$ in place of $\gamma_k$, $\star_k$ and $\circledast_k^\pm$ respectively. By Lemma \ref{Lemma 5.12},
	\begin{align*}
		( \nabla \pm \tfrac{1}{\hbar} \gamma \circledast^\pm )^2 = & \nabla^2 \pm \tfrac{1}{\hbar} (\nabla \gamma) \circledast^\pm \pm \tfrac{1}{2\hbar^2} [\gamma, \gamma]_\star \circledast^\pm\\
		= & \pm \tfrac{1}{\hbar} ( R + \nabla \gamma + \tfrac{1}{2\hbar^2} [\gamma, \gamma]_\star - \sqrt{-1}\hbar \omega_1) \circledast^\pm + \sqrt{-1} \omega_1\\
		= & \mp \tfrac{1}{\hbar} \omega + \sqrt{-1} \omega_1 = -\tfrac{1}{\sqrt{-1}}(\pm k\omega + \omega_1).
	\end{align*}
	Note that the curvature of $\nabla^{L^{(\pm k)}}$ is $\frac{1}{\sqrt{-1}} (\pm k\omega + \omega_1)$. We are done.
\end{proof}

\begin{proposition}
	\label{Proposition 5.14}
	Let $a \in \Omega^*(M, \underline{\mathcal{W}}_{\operatorname{cl}})$ and $s \in \Omega^*(M, \mathcal{W}_{\operatorname{cl}}^{1, 0} \otimes L^{(\pm k)})$. Then we have
	\begin{equation}
		\label{Equation 5.12}
		D_{\pm, k}^L(a \circledast_k^\pm s) = (D_ka) \circledast_k^\pm s + (-1)^{\lvert a \rvert} a \circledast_k^\pm D_{\pm, k}^L s,
	\end{equation}
\end{proposition}
\begin{proof}
	Again, set $\hbar = \tfrac{\sqrt{-1}}{k}$ and write $\gamma$, $\star$, $\circledast^\pm$ in place of $\gamma_k$, $\star_k$ and $\circledast_k^\pm$ respectively. Observe that for $s' \in \Omega^*(M, \mathcal{W}_{\operatorname{cl}}^{1, 0})$,
	\begin{align*}
		( \nabla \pm \tfrac{1}{\hbar} \gamma \circledast^\pm ) (a \circledast^\pm s') = & (\nabla a) \circledast^\pm s' \pm \tfrac{1}{\hbar} \gamma \circledast^\pm (a \circledast^\pm s') + (-1)^{\lvert a \rvert} a \circledast^\pm \nabla s'\\
		= & (\nabla a + \tfrac{1}{\hbar} [\gamma, a]_\star) \circledast^\pm s' + (-1)^{\lvert a \rvert} a \circledast^\pm (\nabla \pm \tfrac{1}{\hbar} \gamma \circledast^\pm)(s').
	\end{align*}
	Then (\ref{Equation 5.12}) follows from the above equality.
\end{proof}

Decompose $\nabla \otimes \operatorname{Id} + \operatorname{Id} \otimes \nabla^{L^{(\pm k)}}$ into its $(1, 0)$-part and $(0, 1)$-part as $\nabla \otimes \operatorname{Id} + \operatorname{Id} \otimes \nabla^{L^{(\pm k)}} = \partial + \overline{\partial}$. 

\begin{lemma}
	\label{Lemma 5.15}
	On $\Omega^*(M, \mathcal{W})$, we have $((\delta^{0, 1})^{-1}\omega) \circledast^\pm = \mp \hbar \delta^{1, 0}$.
\end{lemma}

\begin{lemma}
	The map
	\begin{equation*}
		\Psi_{\pm, k}: (\Omega^*(M, \mathcal{W}_{\operatorname{cl}}^{1, 0} \otimes L^{(\pm k)}), D_{\pm, k}^L) \to (\Omega^*(M, \mathcal{W}_{\operatorname{cl}}^{1, 0} \otimes L^{(\pm k)}), \overline{\partial} - \delta^{1, 0}),
	\end{equation*}
	defined by $s \mapsto s - (\delta^{1, 0})^{-1} (D_{\pm, k}^L - ( \overline{\partial} - \delta^{1, 0} )) s$, is a cochain isomorphism.
\end{lemma}
\begin{proof}
	Let $\underline{D}_{\pm, k}^L = D_{\pm, k}^L - ( \overline{\partial} - \delta^{1, 0} )$. Explicitly, by Lemma \ref{Lemma 5.15},
	\begin{equation*}
		\underline{D}_{\pm, k}^L = \partial \pm \tfrac{k}{\sqrt{-1}} \left( (\delta^{1, 0})^{-1}\omega + A + \tfrac{\sqrt{-1}}{k} B \right) \circledast_k^\pm.
	\end{equation*}
	Let $\mathcal{W}_{\operatorname{cl}, (r)}^{1, 0}$ be the subbundle of $\mathcal{W}_{\operatorname{cl}}^{1, 0}$ containing polynomials of degree at least $r$. We can check that $(\delta^{1, 0})^{-1}\underline{D}_{\pm, k}^L s \in \Omega^p(M, \mathcal{W}_{\operatorname{cl}, (r+1)}^{1, 0} \otimes L^{(\pm k)})$ for any $p, r \in \mathbb{N}$ and $s \in \Omega^p(M, \mathcal{W}_{\operatorname{cl}, (r)}^{1, 0} \otimes L^{(\pm k)})$. For any $s' \in \Omega^*(M, \mathcal{W}_{\operatorname{cl}}^{1, 0} \otimes L^{(\pm k)})$, we can show by iterative method that, there is a unique solution $s \in \Omega^*(M, \mathcal{W}_{\operatorname{cl}}^{1, 0} \otimes L^{(\pm k)})$ to the equation
	\begin{equation}
		s - (\delta^{1, 0})^{-1} \underline{D}_{\pm, k}^L s = s'.
	\end{equation}
	Thus, $\Psi_{\pm, k}$ is a linear isomorphism. We then prove that it is a cochain map. We have
	\begin{align*}
		\Psi_{\pm, k} D_{\pm, k}^L - (\overline{\partial} - \delta^{1, 0}) \Psi_{\pm, k} = \underline{D}_{\pm, k}^L - (\delta^{1, 0})^{-1}\underline{D}_{\pm, k}^L D_{\pm, k}^L + (\overline{\partial} - \delta^{1, 0})(\delta^{1, 0})^{-1} \underline{D}_{\pm, k}^L.
	\end{align*}
	Since $(\delta^{1, 0})^{-1} \overline{\partial} + \overline{\partial} (\delta^{1, 0})^{-1} = 0$ and $(D_{\pm, k}^L)^2 = 0$, 
	\begin{align*}
		(\delta^{1, 0})^{-1}\underline{D}_{\pm, k}^L D_{\pm, k}^L - (\overline{\partial} - \delta^{1, 0})(\delta^{1, 0})^{-1} \underline{D}_{\pm, k}^L = & (\delta^{1, 0})^{-1} (\underline{D}_{\pm, k}^L + \overline{\partial}) D_{\pm, k}^L + \delta^{1, 0} (\delta^{1, 0})^{-1} \underline{D}_{\pm, k}^L\\
		= & ((\delta^{1, 0})^{-1} \delta^{1, 0} + \delta^{1, 0} (\delta^{1, 0})^{-1}) \underline{D}_{\pm, k}^L.
	\end{align*}
	It remains to show that $\pi_{0, *}\underline{D}_{\pm, k}^L = 0$, which is straightforward.
\end{proof}

The inclusion $(\Omega^{0, *}(M, L^{(\pm k)}), \overline{\partial}) \hookrightarrow (\Omega^*(M, \mathcal{W}_{\operatorname{cl}}^{1, 0} \otimes L^{(\pm k)}), \overline{\partial} - \delta^{1, 0})$ is a quasi-isomorphism by weak Hodge decomposition. Therefore, we have a quasi-isomorphism of cochain complexes:
\begin{equation*}
	(\Omega^{0, *}(M, L^{(\pm k)}), \overline{\partial}) \to (\Omega^*(M, \mathcal{W}_{\operatorname{cl}}^{1, 0} \otimes L^{(\pm k)}), D_{\pm, k}^L), \quad s \mapsto \Psi_{\pm, k}^{-1}(s).
\end{equation*}
It implies that we have a homogeneous linear isomorphism
\begin{equation*}
	H^*(\Omega^*(M, \mathcal{W}_{\operatorname{cl}}^{1, 0} \otimes L^{(\pm k)}), D_{\pm, k}^L) \cong H_{\overline{\partial}}^{0, *}(M, L^{(\pm k)}).
\end{equation*}
In particular, $D_{\pm, k}^L$-flat sections of $\mathcal{W}_{\operatorname{cl}}^{1, 0} \otimes L^{(\pm k)}$ are in bijection with holomorphic sections of $L^{(\pm k)}$. Thus, $\circledast_k^\pm$ descend to operations $\mathcal{O}_{\operatorname{qu}}^{(k)} \times \mathcal{L}^{(\pm k)} \to \mathcal{L}^{(\pm k)}$, giving $\mathcal{L}^{(k)}$ a left $\mathcal{O}_{\operatorname{qu}}^{(k)}$-module structure and $\mathcal{L}^{(-k)}$ a right $\mathcal{O}_{\operatorname{qu}}^{(k)}$-module structure.\par
The locality of these constructions implies that the left and right $\mathcal{O}_{\operatorname{qu}}^{(k)}$-module structures induce morphisms of sheaves $\varphi_+: \mathcal{O}_{\operatorname{qu}}^{(k)} \to \mathcal{D}_{L^{(k)}}$ and $\varphi_-: \mathcal{O}_{\operatorname{qu}}^{(k)} \to \mathcal{D}_{L^{(-k)}}^\circ$. To show that $\varphi_+$ and $\varphi_-$ are isomorphisms, we introduce the following key lemma.

\begin{lemma}[Proposition 4.5 in \cite{ChaLeuLi2023}]
	\label{Lemma 5.17}
	Let $U$, $\rho_0$, $\rho_1$, $\rho$ be as in Example \ref{Example 5.4}. Let $\sigma_\pm$ be a local holomorphic frame of $L^{(\pm k)}$ over an open subset $U$ of $M$ such that $\nabla \sigma_\pm = \partial ( \pm k\rho_0 + \rho_1 ) \sigma_\pm$. Define $\Phi_\pm = \sum_{r=1}^\infty (\tilde{\nabla}^{1, 0})^r (\pm k \rho_0 + \rho_1)$ on $U$. Then
	\begin{enumerate}
		\item $D_{\pm, k}^L (e^{\Phi_\pm} \otimes \sigma_\pm ) = 0$;
		\item if $\xi$ is a local holomorphic vector field on $U$, then $\tfrac{k}{\sqrt{-1}} \operatorname{ev}_k( O_{\xi(\rho)} ) \circledast_k^\pm \left( e^{\Phi_\pm} \otimes \sigma_\pm \right) = 0$.
	\end{enumerate}
\end{lemma}
\begin{proof}
	Set $\hbar = \tfrac{\sqrt{-1}}{k}$. Write $\gamma$, $\circledast^\pm$, $\star$, $O_{\xi(\rho)}$ in place of $\gamma_k$, $\circledast_k^\pm$, $\star_k$, $\operatorname{ev}_k(O_{\xi(\rho)})$ respectively. Then
	\begin{equation*}
		\left( \nabla \pm \tfrac{1}{\hbar} \gamma \circledast^\pm \right) e^{\Phi_\pm} = (\nabla - \delta^{1, 0})\Phi_\pm \pm \tfrac{1}{\hbar}A \circledast^\pm e^{\Phi_\pm} \pm \tfrac{1}{\hbar} ((\delta^{1, 0})^{-1}\omega + \hbar B) e^{\Phi_\pm}.
	\end{equation*}
	First, we see that
	\begin{align*}
		A_{(r)} \circledast^+ e^{\Phi_\pm} = & \hbar \omega^{\overline{\mu}\lambda} \omega_{\eta\overline{\mu}} F_{\mu_1, ..., \mu_r, \overline{\beta}}^\eta d\overline{z}^\beta \otimes w^{\mu_1} \cdots w^{\mu_r} \frac{\partial e^{\Phi_\pm}}{\partial w^\lambda} = \left[ A_{(r)}, e^{\Phi_\pm} \right]_\star,\\
		A_{(r)} \circledast^- e^{\Phi_\pm} = & -\hbar \omega^{\overline{\mu}\lambda} \omega_{\eta\overline{\mu}} F_{\mu_1, ..., \mu_r, \overline{\beta}}^\eta d\overline{z}^\beta \otimes \frac{\partial}{\partial w^\lambda} \left( w^{\mu_1} \cdots w^{\mu_r} e^{\Phi_\pm} \right)\\
		= & -[A_{(r)} , e^{\Phi_\pm}]_\star - r\hbar F_{\eta, \mu_2, ..., \mu_r, \overline{\beta}}^\eta d\overline{z}^\beta \otimes w^{\mu_2} \cdots w^{\mu_r} e^{\Phi_\pm}\\
		= & -\left[A_{(r)}, e^{\Phi_\pm} \right]_\star - 2\hbar B_{(r-1)} e^{\Phi_\pm}.
	\end{align*}
	Next, note that $\delta^{0, 1} (e^{\Phi_\pm}) = 0$. We have
	\begin{align*}
		\left( \nabla \pm \tfrac{1}{\hbar} \gamma \circledast^\pm \right) e^{\Phi_\pm} = \left( (\nabla - \delta + \tfrac{1}{\hbar}[A, \quad]_\star) \Phi_\pm + \tfrac{1}{\hbar} (\pm (\delta^{1, 0})^{-1}\omega + \hbar B) \right) e^{\Phi_\pm} = -\partial (\pm k \rho_0 + \rho_1) e^{\Phi_\pm}.
	\end{align*}
	Therefore, $D_{\pm, k}^L (e^{\Phi_\pm} \otimes \sigma_\pm ) = 0$.\par
	Now suppose $\xi$ is a local holomorphic vector filed on $U$. If we can show that
	\begin{equation*}
		\pi_0 \left( \tfrac{1}{\hbar} O_{\xi(\rho)} \circledast^\pm \left( e^{\Phi_\pm} \otimes \sigma_\pm \right) \right) = \pi_0 \left( \tfrac{1}{\hbar} \left( \xi(\rho) + \widehat{\xi} + \tilde{\nabla}^{1, 0} \widehat{\xi} \right) \circledast^\pm \left( e^{\Phi_\pm} \otimes \sigma_\pm \right) \right)
	\end{equation*}
	vanishes, then the proof of the second statement in this lemma is complete. We see that
	\begin{align*}
		\pi_0 \left( \tfrac{1}{\hbar} \xi(\rho) \circledast^\pm e^{\Phi_\pm}  \right) = & \tfrac{1}{\hbar} \xi(\rho) = - \xi(k\rho_0 +  \rho_1),\\
		\pi_0 \left( \tfrac{1}{\hbar} \widehat{\xi} \circledast^\pm e^{\Phi_\pm} \right) = & \pm \omega^{\overline{\nu}\lambda} \frac{\partial \widehat{\xi}}{\partial \overline{w}^\nu} \frac{\partial}{\partial w^\lambda} \left( \tilde{\nabla}^{1, 0}(\pm k\rho_0 + \rho_1) \right) = \xi (k \rho_0 \pm \rho_1).
	\end{align*}
	Write $\xi = \xi^\mu \partial_{z^\mu}$. Since $\widehat{\xi} = \xi^\eta \omega_{\eta\overline{\nu}} \overline{w}^\nu$, $\tilde{\nabla}^{1, 0} \widehat{\xi} = \xi^\eta \frac{\partial \omega_{\eta\overline{\nu}}}{\partial z^\mu} w^\mu \overline{w}^\nu$. Thus,
	\begin{align*}
		\pi_0 \left( \tfrac{1}{\hbar} \tilde{\nabla}^{1, 0} \widehat{\xi} \circledast^+ e^{\Phi_+} \right) = & \pi_0 \left( \xi^\eta \omega^{\overline{\nu}\lambda} \frac{\partial \omega_{\eta\overline{\nu}}}{\partial z^\mu} w^\mu \frac{\partial e^{\Phi_+}}{\partial w^\lambda} \right) = 0,\\
		\pi_0 \left( \tfrac{1}{\hbar} \tilde{\nabla}^{1, 0} \widehat{\xi} \circledast^- e^{\Phi_-} \right) = & \pi_0 \left( -\xi^\eta \omega^{\overline{\nu}\lambda} \frac{\partial \omega_{\eta\overline{\nu}}}{\partial z^\mu} \frac{\partial}{\partial w^\lambda} \left( w^\mu e^{\Phi_-} \right) \right) = -\xi^\eta \omega^{\overline{\nu}\mu} \frac{\partial \omega_{\eta\overline{\nu}}}{\partial z^\mu} = 2\xi(\rho_1).
	\end{align*}
	The last equality is by (\ref{Equation 5.3}). To conclude, $\tfrac{1}{\hbar} O_{\xi(\rho)} \circledast^\pm \left( e^{\Phi_\pm} \otimes \sigma_\pm \right) = 0$.
\end{proof}

Lemma \ref{Lemma 5.17} implies that, if $f$ is a holomorphic function on an open subset $U$ of $M$, then $O_f e^{\Phi_\pm} \otimes \sigma_\pm$ is the unique $D_{\pm, k}^L$-flat section of $\mathcal{W}_{\operatorname{cl}}^{1, 0} \otimes L^{(\pm k)}$ such that $\pi_0(O_f e^{\Phi_\pm} \otimes \sigma_\pm) = f \sigma_\pm$. Therefore, if $s_\pm$ is a local holomorphic section of $L^{(\pm k)}$, then $\varphi_\pm(f)(s_\pm) = fs_\pm$. Now, one can easily see that the induction argument in the proof of Theorem 4.7 in \cite{ChaLeuLi2023} works for the injectivity of both $\varphi_+$ and $\varphi_-$. We also see that
\begin{equation*}
	\operatorname{ev}_k(\tfrac{1}{\hbar} O_{\xi(\rho)}) \circledast_k^\pm \left( O_f e^{\Phi_\pm} \otimes \sigma_\pm \right) = \pm \left[ \operatorname{ev}_k (\tfrac{1}{\hbar} O_{\xi(\rho)}), O_f \right]_{\star_k} e^{\Phi_\pm} \otimes \sigma_\pm = \pm O_{\xi(f)} e^{\Phi_\pm} \otimes \sigma_\pm,
\end{equation*}
verifying the surjectivity of both $\varphi_+$ and $\varphi_-$. Therefore, $\varphi_+$ and $\varphi_-$ are isomorphisms of TDOs.

\subsection{A sheaf of algebras $\overline{\mathcal{O}}_{\operatorname{qu}}^{(k)}$ realizing $\operatorname{Hom}(\overline{\mathcal{B}}_{\operatorname{cc}}^{(k)}, \overline{\mathcal{B}}_{\operatorname{cc}}^{(k)})$ in relation to $\mathcal{O}_{\operatorname{qu}}^{(k)}$}
\label{Subsection 5.4}
\quad\par
Consider another K\"ahler manifold $(\overline{M}, -\omega)$, where $\overline{M}$ is the conjugate complex manifold of $M$, which shares the same K\"ahler metric with $(M, \omega)$. We can apply the constructions in Subsection \ref{Subsection 5.1} to obtain the $(\overline{M}, -\omega)$-counterparts $\check{A}$, $\check{B}$, $\check{\gamma}$, $\check{D}$, $\underline{\check{\mathcal{W}}}$, $\underline{\check{\mathcal{W}}}_{\operatorname{cl}}$, $\check{\star}$, $\check{\circledast}^\pm$ of $A$, $B$, $\gamma$, $D$, $\underline{\mathcal{W}}$, $\underline{\mathcal{W}}_{\operatorname{cl}}$, $\star$, $\circledast^\pm$ respectively. It is intriguing to compare the deformation quantizations $(\mathcal{C}^\infty(M, \mathbb{C})[[\hbar]], \star)$ and $(\mathcal{C}^\infty(M, \mathbb{C})[[\hbar]], \check{\star})$, which have Fedosov's characteristic classes $-\tfrac{1}{\hbar}[\omega]$ and $-\tfrac{1}{\hbar}[-\omega]$ respectively.\par
Clearly, $(\mathcal{C}^\infty(M, \mathbb{C})[[\hbar]], \star)$ and $(\mathcal{C}^\infty(M, \mathbb{C})[[\hbar]], \check{\star})$ are equivalent if and only if $\omega$ is exact. By \cite{Bur2002} (see also \cite{BurDolWal2012, BurWal2002}), $(\mathcal{C}^\infty(M, \mathbb{C})[[\hbar]], \star)$ and $(\mathcal{C}^\infty(M, \mathbb{C})[[\hbar]], \check{\star})$ are Morita equivalent if and only if there is a diffeomorphism $\phi: M \to M$ and a complex line bundle $L_0$ over $M$ such that
\begin{equation*}
	-\tfrac{1}{\hbar}[\phi^*\omega] + 2\pi\sqrt{-1}c_1(L_0) = -\tfrac{1}{\hbar}[-\omega],
\end{equation*}
or equivalently, $[\phi^*\omega] = [-\omega]$ and $c_1(L_0) = 0$. When $M$ is compact, by Moser's trick, we can see that $(\mathcal{C}^\infty(M, \mathbb{C})[[\hbar]], \star)$ and $(\mathcal{C}^\infty(M, \mathbb{C})[[\hbar]], \check{\star})$ are Morita equivalent if and only if there is a diffeomorphism $\phi: M \to M$ such that $\phi^*\omega = -\omega$.\par
Indeed, $(\mathcal{C}^\infty(M, \mathbb{C})[[\hbar]], \check{\star})$ is the `conjugate' of $(\mathcal{C}^\infty(M, \mathbb{C})[[\hbar]], \star)$ in the following sense. Extend the usual complex conjugation on $\mathbb{C}$ to $\mathbb{C}[[\hbar]]$ by requiring that $\overline{\hbar} = -\hbar$. By a \emph{conjugate} $\mathbb{F}$\emph{-algebra isomorphism}, where $\mathbb{F}$ is either $\mathbb{C}$ or $\mathbb{C}[[\hbar]]$, we mean an $\mathbb{R}$-algebra isomorphism $\phi: \mathcal{A}_1 \to \mathcal{A}_2$ between two $\mathbb{F}$-algebras $\mathcal{A}_1, \mathcal{A}_2$ such that for all $\lambda \in \mathbb{F}$ and $a \in \mathcal{A}_1$, $\phi(\lambda a) = \overline{\lambda} \phi(a)$.

\begin{proposition}
	The following map is a conjugate $\mathbb{C}[[\hbar]]$-algebra isomorphism:
	\begin{equation*}
		(\mathcal{C}^\infty(M, \mathbb{C})[[\hbar]], \star) \to (\mathcal{C}^\infty(M, \mathbb{C})[[\hbar]], \check{\star}), \quad f \mapsto \overline{f},
	\end{equation*}
\end{proposition}
\begin{proof}
	By direct computations, we have $\overline{A} = -\check{A}$ and $\overline{B} = \check{B}$, whence $\overline{\gamma} = -\check{\gamma}$. It is easy to check that, for all $a, b \in \Omega^*(M, \mathcal{W})$, $\overline{a \star b} = \overline{a} \check{\star} \overline{b}$. These imply that for all $a \in \Omega^*(M, \mathcal{W})$, $\overline{Da} = \check{D}\overline{a}$. To complete the proof, it remains to pass the above observation to the cohomological level.
\end{proof}

For non-zero $k \in \mathbb{C}$, let $\overline{\mathcal{O}}_{\operatorname{qu}}^{(k)}$ be the sheaf of quantizable functions of level $k$ on $(\overline{M}, -\omega)$.

\begin{proposition}
	For all $k \in \mathbb{C} \backslash \{0\}$, the following map is a conjugate $\mathbb{C}$-algebra isomorphism:
	\begin{equation}
		\label{Equation 5.014}
		(\mathcal{O}_{\operatorname{qu}}^{(k)}, \star_k) \to (\overline{\mathcal{O}}_{\operatorname{qu}}^{(\overline{k})}, \check{\star}_{\overline{k}}), \quad f \mapsto \overline{f}.
	\end{equation}
\end{proposition}
\begin{proof}
	It suffices to observe that for all $a \in \Gamma(M, \underline{\mathcal{W}}_{\operatorname{cl}})$, $\overline{a} \in \Gamma(M, \underline{\check{\mathcal{W}}}_{\operatorname{cl}})$.
\end{proof}

Suppose that $(M, \omega)$ admits a prequantum line bundle $(L, \nabla^L)$ and a square root $\sqrt{K}$ of its canonical bundle. Then the dual of $(L, \nabla^L)$ is a prequantum line bundle of $(M, -\omega)$ and $\sqrt{K}^\vee$ is a square root of the canonical bundle of $\overline{M}$. Now, we fix $k \in \mathbb{Z}^+$. By Theorem \ref{Theorem 5.9} and Corollary \ref{Corollary 5.10}, the sheaf $\overline{\mathcal{L}}^{(k)}$ of antiholomorphic sections of $\overline{L}^{(k)} := (L^\vee)^{\otimes k} \otimes \sqrt{K}^\vee$ is a left $\overline{\mathcal{O}}_{\operatorname{qu}}^{(k)}$-module and the sheaf $\overline{\mathcal{L}}^{(-k)}$ of antiholomorphic sections of $\overline{L}^{(-k)} := L^{\otimes k} \otimes \sqrt{K}^\vee$ is a right $\overline{\mathcal{O}}_{\operatorname{qu}}^{(k)}$-module.\par
Note that the $(\overline{M}, -\omega)$-counterpart of $\mathcal{W}_{\operatorname{cl}}^{1, 0}$ is $\mathcal{W}_{\operatorname{cl}}^{0, 1} := \widehat{\operatorname{Sym}} T^{\vee (0, 1)}M$. Let $h^{(k)}$ be the Hermitian metric on $L^{(k)}$. Define a complex conjugate linear isomorphism
\begin{equation*}
	\Omega^*(M, \mathcal{W}_{\operatorname{cl}}^{1, 0} \otimes L^{(k)}) \to \Omega^*(M, \mathcal{W}_{\operatorname{cl}}^{0, 1} \otimes \overline{L}^{(k)}), \quad s \mapsto s^*,
\end{equation*}
as follows. For $a \in \Omega^*(M, \mathcal{W}_{\operatorname{cl}}^{1, 0})$ and $s \in \Gamma(M, L^{(k)})$, $(a \otimes s)^* := \overline{a} \otimes h^{(k)}(\quad, s)$. It satisfies the property that, if $s \in \Gamma(M, L^{(k)})$ is holomorphic, then $s^* \in \Gamma(M, {\overline{L}}^{(k)})$ is anti-holomorphic.

\begin{proposition}
	The following diagram is commutative:
	\begin{center}
		\begin{tikzcd}
			\mathcal{O}_{\operatorname{qu}}^{(k)} \ar[r] \ar[d] & \overline{\mathcal{O}}_{\operatorname{qu}}^{(k)} \ar[d]\\
			\mathcal{D}_{L^{(k)}} \ar[r] & \mathcal{D}_{\overline{L}^{(k)}}
		\end{tikzcd}
	\end{center}
	Here, the upper horizontal arrow is (\ref{Equation 5.014}), the lower horizontal arrow is the conjugation action by the map $H^0(M, L^{(k)}) \to H^0(\overline{M}, \overline{L}^{(k)})$ given by $s \mapsto s^*$, and the vertical arrows are the isomorphisms given as in Theorem \ref{Theorem 5.9}.
\end{proposition}
\begin{proof}
	It suffices to observe that for all $a \in \Gamma(M, \underline{\mathcal{W}}_{\operatorname{cl}})$ and $s \in \Omega^*(M, \mathcal{W}_{\operatorname{cl}}^{1, 0} \otimes L^{(k)})$,
	\begin{equation*}
		(a \circledast_k^+ s)^* = \overline{a} \check{\circledast}_k^+ s^*,
	\end{equation*}
	and $(D_{+, k}^L s)^* = \check{D}_{+, k}^L (s^*)$, where $\check{D}_{+, k}^L$ is the $(\overline{M}, -\omega)$-counterpart of $D_{+, k}^L$.
\end{proof}

\subsection{A sheaf of bimodules $\mathcal{L}_{\operatorname{sm}}^{(2k)}$ realizing $\operatorname{Hom}(\overline{\mathcal{B}}_{\operatorname{cc}}^{(k)}, \mathcal{B}_{\operatorname{cc}}^{(k)})$}
\label{Subsection 5.5}
\quad\par
In this subsection, assume that $(M, \omega)$ admits a prequantum line bundle $(L, \nabla^L)$ and $k \in \mathbb{Z}^+$. We will prove the following theorem.

\begin{theorem}[$=$ Theorem \ref{Theorem 1.1}]
	The sheaf $\mathcal{L}_{\operatorname{sm}}^{(2k)}$ of smooth sections of $L^{\otimes 2k}$ has an $\mathcal{O}_{\operatorname{qu}}^{(k)}$-$\overline{\mathcal{O}}_{\operatorname{qu}}^{(k)}$ bimodule structure such that if $M$ is spin (with a chosen square root of its canonical bundle), then the natural pairing
	\begin{equation}
		\label{Equation 5.14}
		\mathcal{L}^{(k)} \otimes \overline{\mathcal{L}}^{(-k)} \to \mathcal{L}_{\operatorname{sm}}^{(2k)}
	\end{equation}
	is a morphism of $\mathcal{O}_{\operatorname{qu}}^{(k)}$-$\overline{\mathcal{O}}_{\operatorname{qu}}^{(k)}$ bimodules.
\end{theorem}

\begin{remark}
	\label{Remark 5.24}
	Our construction of the bimodule structure on $\mathcal{L}_{\operatorname{sm}}^{(2k)}$ looks very similar to Neumaier-Waldmann's construction \cite{NeuWal2003} of a Morita equivalence bimodule for Wick type star products, but they are not the same. If there were a Morita equivalence bimodule structure on $\Gamma(M, L_0)[[\hbar]]$ for the two star products $\star$ and $\check{\star}$, then Fedosov's characteristic classes of $(\mathcal{C}^\infty(M, \mathbb{C})[[\hbar]], \star)$ and $(\mathcal{C}^\infty(M, \mathbb{C})[[\hbar]], \check{\star})$ would differ by the constant term $2\pi\sqrt{-1}c_1(L_0)$ in $\hbar$ (see (67) in \cite{NeuWal2003}). However, the difference is actually $-\tfrac{1}{\hbar}[2\omega]$, not a constant term in $\hbar$ unless $\omega$ is exact. Upon evaluation at $\hbar = \tfrac{\sqrt{-1}}{k}$, this difference is equal to, up to a complex scalar multiple, $c_1(L^{\otimes 2k})$.
\end{remark}

To prove Theorem \ref{Theorem 1.1}, we first define a connection $D_k^L$ on $\mathcal{W}_{\operatorname{cl}} \otimes L^{\otimes 2k}$ by
\begin{equation}
	D_k^L := \left( \nabla + \tfrac{k}{\sqrt{-1}} \gamma_k \circledast_k^+ - \tfrac{k}{\sqrt{-1}} \check{\gamma}_k \check{\circledast}_k^- \right) \otimes \operatorname{Id} + \operatorname{Id} \otimes \nabla^{L^{\otimes 2k}},
\end{equation}
where $\mathcal{W}_{\operatorname{cl}} = \widehat{\operatorname{Sym}} T^\vee M_\mathbb{C}$ and $\nabla^{L^{\otimes 2k}}$ is the Chern connection on $L^{\otimes 2k}$. Again, we will show two basic properties of $D_k^L$ in the following propositions.

\begin{proposition}
	The connection $D_k^L$ is flat.
\end{proposition}
\begin{proof}
	Note that for $s \in \Omega^*(M, \mathcal{W})$, $a \in \Omega^*(M, \underline{\mathcal{W}})$ and $b \in \Omega^*(M, \underline{\check{\mathcal{W}}})$ with $a, b$ homogeneous,
	\begin{equation}
		\label{Equation 5.16}
		a \circledast^+ (b \check{\circledast}^- s) = (-1)^{\lvert a \rvert \lvert b \rvert} b \check{\circledast}^- (a \circledast^+ s).
	\end{equation}
	Also note that the $(\overline{M}, -\omega)$-counterpart of $\omega_1$ is $-\omega_1$. Mimicking the proof of Proposition \ref{Proposition 5.13} and using (\ref{Equation 5.16}), we can show that on $\Omega^*(M, \mathcal{W}_{\operatorname{cl}})$,
	\begin{equation*}
		\left( \nabla + \tfrac{k}{\sqrt{-1}} \gamma_k \circledast_k^+ - \tfrac{k}{\sqrt{-1}} \check{\gamma}_k \circledast_k^- \right)^2 = -\tfrac{1}{\sqrt{-1}}(k\omega + \omega_1 -k(-\omega) + (-\omega_1)) = -\tfrac{2k}{\sqrt{-1}} \omega.
	\end{equation*}
	The curvature of $\nabla^{L^{\otimes 2k}}$ is $\tfrac{2k}{\sqrt{-1}} \omega$. Therefore, $(D_k^L)^2 = 0$.
\end{proof}

\begin{proposition}
	We have
	\begin{align*}
		D_k^L (a \circledast_k^+ s) = & (D_ka) \circledast_k^+ s + (-1)^{\lvert a \rvert} a \circledast_k^+ D_k^L s,\\
		D_k^L (a \check{\circledast}_k^- s) = & (\check{D}_k a) \check{\circledast}_k^- s + (-1)^{\lvert a \rvert} a \check{\circledast}_k^- D_k^L s.
	\end{align*}
\end{proposition}
\begin{proof}
	The proof is similar to that of Proposition \ref{Proposition 5.14} and we skip it.
\end{proof}

\begin{lemma}
	The map
	\begin{equation*}
		\Psi_k: (\Omega^*(M, \mathcal{W}_{\operatorname{cl}} \otimes L^{\otimes 2k}), D_k^L) \to (\Omega^*(M, \mathcal{W}_{\operatorname{cl}} \otimes L^{\otimes 2k}), -\delta),
	\end{equation*}
	defined by $s \mapsto s - \delta^{-1} (D_k^L + \delta) s$, is a cochain isomorphism.
\end{lemma}
\begin{proof}
	Let $\underline{D}_k^L = D_k^L + \delta$ and $k' = \tfrac{k}{\sqrt{-1}}$. The $(\overline{M}, -\omega)$-counterpart of $\delta^{1, 0}$ is $\delta^{0, 1}$. By Lemma \ref{Lemma 5.15},
	\begin{equation*}
		\underline{D}_k^L = \left(\nabla + \left(k' \left( (\delta^{1, 0})^{-1}\omega + A \right) + B \right) \circledast_k^+ - \left( k' \left( (\delta^{0, 1})^{-1} (-\omega) + \check{A} \right)+ \check{B} \right) \check{\circledast}_k^-\right) \otimes \operatorname{Id} + \operatorname{Id} \otimes \nabla^{L^{\otimes 2k}}.
	\end{equation*}
	Let $\mathcal{W}_{\operatorname{cl}, (r)}$ be the subbundle of $\mathcal{W}_{\operatorname{cl}}$ containing polynomials of degree at least $r$. We can check that $\delta^{-1}\underline{D}_k^L s \in \Omega^p(M, \mathcal{W}_{\operatorname{cl}, (r+1)} \otimes L^{\otimes 2k})$ for any $p, r \in \mathbb{N}$ and $s \in \Omega^p(M, \mathcal{W}_{\operatorname{cl}, (r)} \otimes L^{\otimes 2k})$. By iterative method, we can show that, for any $s' \in \Omega^*(M, \mathcal{W}_{\operatorname{cl}} \otimes L^{\otimes 2k})$, there is a unique solution $s \in \Omega^*(M, \mathcal{W}_{\operatorname{cl}} \otimes L^{\otimes 2k})$ to the equation
	\begin{equation}
		s - \delta^{-1} \underline{D}_k^L s = s'.
	\end{equation}
	In other words, $\Psi_k$ is a linear isomorphism. On the other hand, as $(D_k^L)^2 = \delta^2 = 0$,
	\begin{align*}
		\Psi_k D_k^L + \delta \Psi_k = D_k^L + \delta - \delta^{-1}\underline{D}_k^L D_k^L - \delta\delta^{-1} \underline{D}_k^L = \underline{D}_k^L - \delta^{-1}\delta \underline{D}_k^L - \delta\delta^{-1} \underline{D}_k^L = \pi_0\underline{D}_k^L = 0.
	\end{align*}
\end{proof}

\begin{proof}[\myproof{Theorem}{\ref{Theorem 1.1}}]
The inclusion $\Gamma(M, L^{\otimes 2k}) \hookrightarrow \Omega^*(M, \mathcal{W}_{\operatorname{cl}} \otimes L^{\otimes 2k})$ is a quasi-isomorphism from $(\Gamma(M, L^{\otimes 2k}), 0)$ to $(\Omega^*(M, \mathcal{W}_{\operatorname{cl}} \otimes L^{\otimes 2k}), -\delta)$ by weak Hodge decomposition. Therefore, we have a quasi-isomorphism of cochain complexes:
\begin{equation*}
	(\Gamma(M, L^{\otimes 2k}), 0) \to (\Omega^*(M, \mathcal{W}_{\operatorname{cl}} \otimes L^{\otimes 2k}), D_k^L), \quad s \mapsto \Psi_k^{-1}(s).
\end{equation*}
It implies that for all $p > 0$, $H^p(\Omega^*(M, \mathcal{W}_{\operatorname{cl}} \otimes L^{\otimes 2k}), D_k^L) = 0$ and we have a linear isomorphism
\begin{equation*}
	H^0(\Omega^*(M, \mathcal{W}_{\operatorname{cl}} \otimes L^{\otimes 2k}), D_k^L) \cong \Gamma(M, L^{\otimes 2k}).
\end{equation*}
In other words, $D_k^L$-flat sections of $\mathcal{W}_{\operatorname{cl}} \otimes L^{\otimes 2k}$ are in bijection with smooth sections of $L^{\otimes 2k}$. Thus, $\circledast_k^+$ and $\check{\circledast}_k^-$ descend to operations $\mathcal{O}_{\operatorname{qu}}^{(k)} \times \mathcal{L}_{\operatorname{sm}}^{(2k)} \to \mathcal{L}_{\operatorname{sm}}^{(2k)}$ and $\overline{\mathcal{O}}_{\operatorname{qu}}^{(k)} \times \mathcal{L}_{\operatorname{sm}}^{(2k)} \to \mathcal{L}_{\operatorname{sm}}^{(2k)}$ respectively. These give an $\mathcal{O}_{\operatorname{qu}}^{(k)}$-$\overline{\mathcal{O}}_{\operatorname{qu}}^{(k)}$ bimodule structure on $\mathcal{L}_{\operatorname{sm}}^{(2k)}$.\par
Now, assume $M$ is spin and we choose a square root $\sqrt{K}$ of the canonical bundle of $M$. Then there is a natural pairing
\begin{equation}
	\label{Equation 5.18}
	\Omega^*(M, \mathcal{W}_{\operatorname{cl}}^{1, 0} \otimes L^{(k)}) \times \Omega^*(M, \mathcal{W}_{\operatorname{cl}}^{0, 1} \otimes \overline{L}^{(-k)}) \to \Omega^*(M, \mathcal{W}_{\operatorname{cl}} \otimes L^{\otimes 2k}), \quad (s, \check{s}) \mapsto s \check{s}.
\end{equation}
We can easily show that for all $s \in \Omega^*(M, \mathcal{W}_{\operatorname{cl}}^{1, 0} \otimes L^{(k)})$ and $\check{s} \in \Omega^*(M, \mathcal{W}_{\operatorname{cl}}^{0, 1} \otimes \overline{L}^{(-k)})$,
\begin{equation*}
	D_k^L(s \check{s}) = (D_{+, k}^L s) \check{s} + (-1)^{\lvert s \rvert} s (\check{D}_{-, k}^L \check{s}),
\end{equation*}
where $\check{D}_{-, k}^L$ is the $(\overline{M}, -\omega)$-counterpart of $D_{-, k}^L$. Hence, (\ref{Equation 5.18}) induces the morphism of sheaves in (\ref{Equation 5.14}). It remains to verify that, if $a \in \Omega^*(M, \underline{\mathcal{W}}_{\operatorname{cl}})$, $\check{a} \in \Omega^*(M, \underline{\check{\mathcal{W}}}_{\operatorname{cl}})$, $s \in \Omega^*(M, \mathcal{W}_{\operatorname{cl}}^{1, 0} \otimes L^{(k)})$ and $\check{s} \in \Omega^*(M, \mathcal{W}_{\operatorname{cl}}^{0, 1} \otimes \overline{L}^{(-k)})$, then $a \circledast_k^+ (s \check{s}) = (a \circledast_k^+ s) \check{s}$ and $a \check{\circledast}_k^- (s \check{s}) = (-1)^{\lvert a \rvert \lvert s \rvert} s (a \check{\circledast}_k^- \check{s})$. The verifications are straightforward.
\end{proof}

\appendix

\section{Complexification and holomorphic extensions}
\label{Appendix A}
We adopt the definition of a complexification of a real analytic manifold in \cite{LouSan2017}.

\begin{definition}
	A \emph{complexification} of a real analytic manifold $M$ is a complex manifold $X$ together with a closed real analytic embedding $\iota: M \hookrightarrow X$ such that $\iota(M)$ is a maximally totally real submanifold  of $X$.
\end{definition}

\begin{proposition}
	Let $M$ be a real analytic manifold. Then the following statements hold.
	\begin{enumerate}
		\item $M$ admits a complexification $\iota: M \hookrightarrow X$ and $X$ can be chosen to be a Stein manifold.
		\item If $\iota_1: M \hookrightarrow X_1$ and $\iota_2: M \hookrightarrow X_2$ are complexifications of $M$, then there exists a unique germ of biholomorphism $\phi: U_1 \to U_2$, where $U_i$ is a neighbourhood of $\iota_i(M)$ in $X_i$ for $i \in \{1, 2\}$ such that $\iota_2 = \phi \circ \iota_1$.
	\end{enumerate}
\end{proposition}

Now we fix a complexification $\iota: M \hookrightarrow X$ of a real analytic manifold $M$. We call a holomorphic vector bundle $E^{\operatorname{c}}$ over $X$ a \emph{holomorphic extension} of a real analytic complex vector bundle $E$ over $M$ if $E$ and $E^{\operatorname{c}}$ are of the same rank and there exists an injective morphism of real analytic complex vector bundles $\iota_E: E \hookrightarrow E^{\operatorname{c}}$ covering $\iota: M \hookrightarrow X$ (which implies $E \cong \iota^*E^{\operatorname{c}}$). It can be shown that every real analytic complex vector bundle $E$ over $M$ has a holomorphic extension on a neighbourhood of $\iota(M)$ in $X$ and the germ of holomorphic extensions of $E$ is unique.

\begin{example}
	As $\iota(M)$ is maximally totally real in $X$, we have a vector bundle isomorphism
	\begin{equation*}
		\iota^*: T^{\vee (1, 0)}X \vert_M \to T^\vee M_\mathbb{C}.
	\end{equation*}
	Its inverse $T^\vee M_\mathbb{C} \to T^{\vee (1, 0)}X \vert_M$ and its dual map $TM_\mathbb{C} \to T^{1, 0}X \vert_M$, which is the restriction of the canonical projection $\pi^{1, 0}: TX_\mathbb{C} \to T^{1, 0}X$, realize $T^{\vee (1, 0)}X$ and $T^{(1, 0)}X$ as holomorphic extensions of $T^\vee M_\mathbb{C}$ and $TM_\mathbb{C}$ respectively.
\end{example}

For a smooth section $\sigma$ of a holomorphic extension $E^{\operatorname{c}}$ of $E$, we denote by $\iota^*\sigma$ the unique smooth section of $E$ such that for all $p \in M$, $\sigma(\iota(p)) = \iota_E( (\iota^*\sigma)(p) )$. If $\sigma$ is holomorphic, then $\iota^*\sigma$ is real analytic and $\sigma$ is called a \emph{holomorphic extension} of $\iota^*\sigma$. Similarly, a holomorphic connection $\nabla^{\operatorname{c}}$ on $E^{\operatorname{c}}$ is called a \emph{holomorphic extension} of a real analytic connection $\nabla$ on $E$ if for any open subset $U$ of $X$ with $U \cap M \neq \emptyset$ and local holomorphic section $\sigma$ of $E^{\operatorname{c}}$ over $U$, $\nabla^{\operatorname{c}} \sigma$ is a holomorphic extension of $\nabla \iota^*\sigma$. Again, a real analytic section of (resp. connection on) $E$ always admits a holomorphic extension on a neighbourhood of $\iota(M)$ in $X$ and the germ of its holomorphic extension is unique. All of the above existence and uniqueness results of holomorphic extensions can be proved by analytic continuation of real analytic functions. A final remark is that,
\begin{enumerate}
	\item if $\xi_0, \xi_1$ are local holomorphic vector fields on $X$, then $\iota^*[\xi_0, \xi_1] = [ \iota^*\xi_0, \iota^*\xi_1 ]$;
	\item if $\alpha \in \Omega^{*, 0}(X)$ is holomorphic, then $\iota^*(\partial \alpha) = d\iota^*\alpha$; and
	\item the curvature of $\nabla^{\operatorname{c}}$ is a holomorphic extension of that of $\nabla$.
\end{enumerate}

\section{Holomorphic deformation quantization}
\label{Appendix B}
Here we will give a brief review on holomorphic deformation quantizations of a holomorphic symplectic manifold $(X, \Omega)$, mainly following the framework of Nest-Tsygan \cite{NesTsy2001}. A few notions are borrowed from Bezrukavnikov-Kaledin \cite{BezKal2004} and Yekutieli \cite{Yek2008} in the algebraic case.\par
Let $U$ be an open subset of $X$ and $\mathcal{O}_U$ be the sheaf of holomorphic functions on $U$. A \emph{holomorphic star product} on $\mathcal{O}_U[[\hbar]]$ is a $\mathbb{C}[[\hbar]]$-bilinear sheaf morphism $\star: \mathcal{O}_U[[\hbar]] \times \mathcal{O}_U[[\hbar]] \to \mathcal{O}_U[[\hbar]]$ such that $(\mathcal{O}_U[[\hbar]], \star)$ forms a sheaf of $\mathbb{C}[[\hbar]]$-algebras with unit $1$ and there is a sequence $\{ C_r \}_{r=1}^\infty$ of holomorphic bi-differential operators on $U$ such that for all open subset $V$ of $U$ and $f, g \in \mathcal{O}_U(V)$,
\begin{equation*}
	f \star g = fg + \sum_{r=1}^\infty \hbar^r C_r(f, g).
\end{equation*}
A \emph{gauge equivalence} of $\mathcal{O}_U[[\hbar]]$ is a $\mathbb{C}[[\hbar]]$-linear automorphism of sheaves $\gamma: \mathcal{O}_U[[\hbar]] \to \mathcal{O}_U[[\hbar]]$ such that $\gamma(1) = 1$ and there is a sequence $\{ D_r \}_{r=1}^\infty$ of holomorphic differential operators on $U$ such that for all open subset $V$ of $U$ and $f \in \mathcal{O}_U(V)$,
\begin{align*}
	\gamma(f) = f + \sum_{r=1}^\infty \hbar^r D_r(f).
\end{align*}
Now suppose that $\mathcal{A}$ is a sheaf of $\hbar$-adically complete flat $\mathbb{C}[[\hbar]]$-algebras on $X$ with multiplication $\star$ and $\psi: \mathcal{A}/\hbar \mathcal{A} \to \mathcal{O}_X$ is an isomorphism of sheaves of $\mathbb{C}$-algebras. A \emph{differential trivialization} of $(\mathcal{A}, \psi)$ on $U$ is an isomorphism of sheaves of $\mathbb{C}[[\hbar]]$-modules
\begin{align*}
	\tau: \mathcal{O}_U[[\hbar]] \to \mathcal{A} \vert_U
\end{align*}
such that the multiplication $\star_\tau$ on $\mathcal{O}_U[[\hbar]]$ identified with $\star$ via the morphism $\tau$ is a holomorphic star product on $\mathcal{O}_U[[\hbar]]$, and for all open subset $V$ of $U$ and $f \in \mathcal{O}_U(V)$, $(\psi \circ \tau)(f) = f$. Finally, a \emph{differential structure} of $(\mathcal{A}, \psi)$ is an open cover $\{ U_i \}_{i \in I}$ of $X$ with a differential trivialization $\tau_i$ of $(\mathcal{A}, \psi)$ on $U_i$ for each $i \in I$ such that for all $i, j \in I$ with $U_{ij} = U_i \cap U_j$, the transition automorphism $\tau_j^{-1} \circ \tau_i: \mathcal{O}_{U_{ij}}[[\hbar]] \to \mathcal{O}_{U_{ij}}[[\hbar]]$ is a gauge equivalence.

\begin{definition}[Definition 5.1 in \cite{NesTsy2001}]
	The pair $(\mathcal{A}, \psi)$ is called a \emph{holomorphic deformation quantization} of $(X, \Omega)$ if it admits a differential structure and
	\begin{equation*}
		\{f, g\} = \psi\left( \frac{1}{\hbar} [\tilde{f}, \tilde{g}] \right),
	\end{equation*}
	for all open subset $U$ of $X$ and $f, g \in \mathcal{O}_X(U)$, where $\{\quad, \quad\}$ is the holomorphic Poisson bracket induced by $\Omega$, $\tilde{f}, \tilde{g} \in \mathcal{A}(U)$ are lifts of $f, g$ respectively and $[\tilde{f}, \tilde{g}] := \tilde{f} \tilde{g} - \tilde{g} \tilde{f}$.
\end{definition}

Now we introduce a crucial geometric construction of holomorphic deformation quantizations of $(X, \Omega)$. Define $\mathcal{W} = \widehat{\operatorname{Sym}} T^{\vee (1, 0)}X[[\hbar]]$, where $T^{\vee (1, 0)}X$ is the holomorphic cotangent bundle of the complex manifold $X$. In local complex coordinates $(z^1, ..., z^{2n})$ on $X$, a smooth (resp. holomorphic) section of $\mathcal{W}$ is given by a formal power series $\sum_{r, l \geq 0} \sum_{i_1, ..., i_l \geq 0} \hbar^r a_{r, i_1, ..., i_l} w^{i_1} \cdots w^{i_l}$, where $a_{r, i_1, ..., i_l}$ are local smooth (resp. holomorphic) functions and $w^i$ denotes the covector $dz^i$ regarded as a section of $\mathcal{W}$. Similar to the real symplectic case, we assign weights on $\mathcal{W}$ by setting the weight of $w^i$ to be $1$ and the weight of $\hbar$ to be $2$. For $k \in \mathbb{N}$, we denote by $\mathcal{W}_{(k)}$ the subbundle of $\mathcal{W}$ of weight at least $k$. We then have a decreasing filtration
\begin{align*}
	\mathcal{W} = \mathcal{W}_{(0)} \supset \mathcal{W}_{(1)} \supset \cdots \supset  \mathcal{W}_{(k)} \supset \cdots.
\end{align*}
There are $\mathcal{C}^\infty(X, \mathbb{C})[[\hbar]]$-linear operators $\delta, \delta^{-1}, \pi_0$ on $\Omega^*(X, \mathcal{W})$ defined as follows: for a local section $a = w^{i_1} \cdots w^{i_l} dz^{j_1} \wedge \cdots \wedge dz^{j_p} \wedge d\overline{z}^{j'_1} \wedge \cdots \wedge d\overline{z}^{j'_q}$,
\begin{align*}
	\delta a = dz^k \wedge \frac{\partial a}{\partial w^k},\quad
	\delta^{-1} a =
	\begin{cases}
		\frac{1}{l+p} w^k \iota_{\partial_{z^k}} a & \text{ if } l + p > 0;\\
		0 & \text{ if } l + p = 0,
	\end{cases}
	\quad \pi_0(a) =
	\begin{cases}
		0 & \text{ if } l + p > 0;\\
		a & \text{ if } l + p = 0,
	\end{cases}
\end{align*}
and the equality $\operatorname{Id} - \pi_0 = \delta \circ \delta^{-1} + \delta^{-1} \circ \delta$ holds on $\Omega^*(X, \mathcal{W})$.\par
Write $\Omega = \tfrac{1}{2} \Omega_{ij} dz^i \wedge dz^j$ and let $(\Omega^{ij})$ be the inverse of $(\Omega_{ij})$. There is a fibrewise holomorphic star product $\star_{\operatorname{M}}$ on $\Omega^*(X, \mathcal{W})$ known as the \emph{fibrewise holomorphic Moyal product} and defined by
\begin{equation*}
	\alpha \star_{\operatorname{M}} \beta = \sum_{r=0}^\infty \frac{1}{r!} \left( \frac{\hbar}{2} \right)^r \Omega^{i_1j_1} \cdots \Omega^{i_rj_r} \frac{\partial^r \alpha}{\partial w^{i_1} \cdots \partial w^{i_r}} \wedge \frac{\partial^r \beta}{\partial w^{j_1} \cdots \partial w^{j_r}},
\end{equation*}
preserving the weight filtration of $\mathcal{W}$, i.e. $\mathcal{W}_{(k_1)} \star_{\operatorname{M}} \mathcal{W}_{(k_2)} \subset \mathcal{W}_{(k_1 + k_2)}$.\par
The bundle $T^{1, 0}X$ equipped with $\Omega$ is a complex Lie algebroid. We pick a $T^{1, 0}X$-connection $\nabla$ on $T^{1, 0}X$ such that $\nabla \Omega = 0$ (indeed we can choose $\nabla$ so that it is torsion-free). There exists a unique element $R \in \Omega^2(X, \operatorname{Sym}^2 T^{\vee (1, 0)}X)$ such that $(\nabla + \overline{\partial})^2 = \tfrac{1}{\hbar}[R, \quad]_{\star_{\operatorname{M}}}$. Then
\begin{equation*}
	(\Omega^*(X, \mathcal{W}), R, \nabla + \overline{\partial}, \tfrac{1}{\hbar} [\quad, \quad]_{\star_{\operatorname{M}}})
\end{equation*}
forms a curved dgla. We can deform this structure by an element $\gamma \in \Omega^1(X, \mathcal{W})$ to form a new curved dgla $(\Omega^*(X, \mathcal{W}), R_\gamma, D_\gamma, \tfrac{1}{\hbar} [\quad, \quad]_{\star_{\operatorname{M}}})$, where
\begin{align*}
	D_\gamma :=  \nabla + \overline{\partial} + \tfrac{1}{\hbar} [\gamma, \quad]_{\star_{\operatorname{M}}} \quad \text{and} \quad R_\gamma := R + (\nabla + \overline{\partial}) \gamma + \tfrac{1}{2\hbar} [\gamma, \gamma]_{\star_{\operatorname{M}}}.
\end{align*}

\begin{definition}
	\label{Definition B.2}
	A complex manifold $X$ is said to be \emph{admissible} if for $i \in \{1, 2\}$, the canonical map $H^i(X, \mathbb{C}) \to H^i(X, \mathcal{O}_X)$ is surjective. 
\end{definition}

Admissibility of $X$ implies that the following is a short exact sequence of $\mathbb{C}$-vector spaces:
\begin{center}
	\begin{tikzcd}
		0 \ar[r] & \mathbb{H}^2(X, F^1\Omega_X^*) \ar[r] & H^2(X, \mathbb{C}) \ar[r] & H^2(X, \mathcal{O}_X) \ar[r] & 0
	\end{tikzcd}
\end{center}
where $F^1\Omega_X^*$ is the Hodge filtration of the sheaf $(\Omega_X^*, \partial)$ of holomorphic de-Rham complex on $X$. We now use deRham or Dolbeault models to describe the above sheaf cohomologies:
\begin{align*}
	\mathbb{H}^2(X, F^1\Omega_X^*) \cong H_d^2(\Omega^{*\geq 1, *}(X)), \quad H^2(X, \mathbb{C}) \cong H_d^2(\Omega^*(X, \mathbb{C})), \quad H^2(X, \mathcal{O}_X) \cong H_{\overline{\partial}}^2(\Omega^{0, 2}(X)).
\end{align*}
Assuming $X$ is admissible, we are able to construct a (not necessarily $\mathbb{C}$-linear) map
\begin{align*}
	\Omega_{\overline{\partial}\text{-closed}}^{2, 0}(X) \to \Omega_{d\text{-closed}}^2(X, \mathbb{C}), \quad \tau \mapsto \Lambda^\tau + \tau,
\end{align*}
where $\Lambda^\tau \in \Omega^{2, 0}(X) \oplus \Omega^{1, 1}(X)$, which descends to a splitting $H^2(X, \mathcal{O}_X) \to H^2(X, \mathbb{C})$ of the above short exact sequence. Let us fix a choice of such a map $\tau \mapsto \Lambda^\tau + \tau$. Also, let $\tilde{\Omega} \in \Omega^{1, 0}(X, T^{\vee (1, 0)}X)$ be the unique element such that $\delta = \tfrac{1}{\hbar} [\tilde{\Omega}, \quad]_{\star_{\operatorname{M}}}$.

\begin{theorem}[Theorem 5.9 in \cite{NesTsy2001}]
	\label{Theorem B.3}
	Suppose that $(X, \Omega)$ is an admissible holomorphic symplectic manifold and that $\Lambda \in \hbar (\Omega^{2, 0}(X) \oplus \Omega^{1, 1}(X))[[\hbar]]$ is $d$-closed. Then there exist $A \in \Omega^{1, 0}(X, \mathcal{W}_{(2)})$, $B \in \Omega^{0, 1}(X, \mathcal{W}_{(3)})$, and $\overline{\partial}$-closed $\tau \in \hbar^2 \Omega^{0, 2}(X)[[\hbar]]$, whose cohomology class $[\tau] \in \hbar H^2(X, \mathcal{O}_X)[[\hbar]]$ is uniquely determined by $[\Lambda] \in \hbar \mathbb{H}^2(X, F^1\Omega_X^*)[[\hbar]]$, such that
	\begin{equation}
		R_\gamma = -\Omega + \Lambda + \Lambda^\tau + \tau,
	\end{equation}
	where $\gamma = -\tilde{\Omega} + A + B$. In particular, $D_\gamma$ is a flat connection.
\end{theorem}

In this case, the sheaf of $D_\gamma$-flat sections of $\mathcal{W}$ is a holomorphic deformation quantization of $(X, \Omega)$. Indeed, every holomorphic deformation quantization of $(X, \Omega)$ is isomorphic to the one constructed in the above way \cite{NesTsy2001}. In consequence, we have the following classification theorem.

\begin{theorem}[Theorem 5.11 in \cite{NesTsy2001}]
	\label{Theorem B.4}
	Suppose $(X, \Omega)$ is an admissible holomorphic symplectic manifold. Then equivalence classes of holomorphic deformation quantizations of $(X, \Omega)$ are in bijection with elements in $-\tfrac{1}{\hbar}[\Omega] + \mathbb{H}^2(X, F^1\Omega_X^*)[[\hbar]]$.
\end{theorem}

\section{Existence of real analytic or holomorphic symplectic connections}

\begin{proposition}
	\label{Proposition C.1}
	Every symplectic manifold $(M, \omega)$ admits a real analytic symplectic connection.
\end{proposition}
\begin{proof}
	Pick a torsion-free real analytic connection $\nabla^{(0)}$ on $M$ (e.g. the Levi-Civita connection of a real analytic metric on $M$). Then a usual proof of the existence of a symplectic connection states that the connection $\nabla$ on $M$ defined by
	\begin{align*}
		\nabla_XY = \nabla_X^{(0)}Y + \tfrac{1}{3} (N(X, Y) + N(Y, X))
	\end{align*}
	is a symplectic connection, where $N$ is a smooth section of $T^\vee M \otimes T^\vee M \otimes TM$ defined by $(\nabla_X^{(0)}\omega)(Y, Z) = \omega(N(X, Y), Z)$. As $\nabla^{(0)}$ and $\omega$ are real analytic, so are $N$ and $\nabla$.
\end{proof}

\begin{proposition}
	\label{Proposition C.2}
	Every Stein holomorphic symplectic manifold $(X, \Omega)$ admits a holomorphic symplectic connection.
\end{proposition}

The space of torsion-free $T^{1, 0}X$-connections on $X$ preserving $\Omega$ is an affine space modelled by the vector space of smooth sections of $E$, where $E$ is the holomorphic vector subbundle of $\operatorname{End}(T^{1, 0}X, \mathfrak{sp}(T^{1, 0}X))$ consisting of $\mathbb{C}$-linear maps $\phi: T_p^{1, 0}X \to \mathfrak{sp}(T_p^{1, 0}X)$ (for $p \in X$) such that $\phi(u)v = \phi(v)u$, and $\mathfrak{sp}(T^{1, 0}X)$ is the bundle of symplectic Lie algebras with respect to $\Omega$. One way to obtain a holomorphic symplectic connection is to deform a torsion-free $\Omega$-preserving $T^{1, 0}X$-connection on $T^{1, 0}X$ in the above affine space. The following lemma is needed.

\begin{lemma}
	\label{Lemma C.3}
	Let $\nabla$ be a torsion-free $T^{1, 0}X$-connection on $X$. The $(1, 1)$-part $R^{1, 1}$ of $( \nabla + \overline{\partial} )^2$ satisfies the following condition: for all smooth sections $u$ of $T^{0, 1}X$ and $v, w$ of $T^{1, 0}X$,
	\begin{align*}
		R^{1, 1}(u, v)w = R^{1, 1}(u, w)v.
	\end{align*}
\end{lemma}
\begin{proof}
	Note that $R^{1, 1}(u, v)w = [u, \nabla_vw]^{1, 0} - \nabla_v([u, w]^{1, 0}) - \nabla_{[u, v]^{1, 0}} w - [[u, v]^{0, 1}, w]^{1, 0}$, where for any smooth complex vector field $\xi$ on $X$, $\xi^{1, 0}, \xi^{0, 1}$ are its $(1, 0)$-part and $(0, 1)$-part respectively. Since $\nabla$ is torsion-free,
	\begin{align*}
		& R^{1, 1}(u, v)w - R^{1, 1}(u, w)v\\
		= & [u, [v, w]]^{1, 0} - [v, [u, w]^{1, 0}] - [[u, v]^{1, 0}, w] - [[u, v]^{0, 1}, w]^{1, 0} - [[u, w]^{0, 1}, v]^{1, 0}\\
		= & -[u, [v, w]]^{0, 1} + [v, [u, w]^{0, 1}] + [[u, v]^{0, 1}, w] - [[u, v]^{0, 1}, w]^{1, 0} - [[u, w]^{0, 1}, v]^{1, 0}\\
		= & -[u, [v, w]]^{0, 1} + [v, [u, w]^{0, 1}]^{0, 1} + [[u, v]^{0, 1}, w]^{0, 1},\\
		= & -( \partial_v\partial_wu - \partial_w\partial_vu - \partial_{[v, w]}u ) = 0.
	\end{align*}
	In the third line, the Jacobi identity is used.
\end{proof}

\begin{proof}[\myproof{Proposition}{\ref{Proposition C.2}}]
	Pick a torsion-free $T^{1, 0}X$-connection $\nabla$ on $X$ such that $\nabla \Omega = 0$. By Lemma \ref{Lemma C.3} and the property that $\nabla\Omega = 0$, the $(1, 1)$-part $R^{1, 1}$ of the curvature of $\nabla + \overline{\partial}$ is a $\overline{\partial}$-closed element in $\Omega^{0, 1}(X, E)$. As the higher cohomology of any analytic coherent sheaf on a Stein manifold vanishes, there exists $\alpha \in \Gamma(X, E) \subset \Omega^{1, 0}(X, \mathfrak{sp}(T^{1, 0}X))$ such that $R^{1, 1} = \overline{\partial}\alpha$. Then $\nabla - \alpha$ is a torsion-free holomorphic connection on $X$ preserving $\Omega$.
\end{proof}

\subsection*{Acknowledgement}
\quad\par
The author thanks Naichung Conan Leung for helpful discussions during the visit in the Institute of Mathematical Sciences at the Chinese University of Hong Kong in November 2023. He also thanks Daniel Burns for insightful comments on complexifications of real analytic manifolds, as well as Kwokwai Chan and Zhiming Ma for some advice on extensions of a deformation quantization.

\subsection*{Conflict of interest}
\quad\par
The author declares that there is no conflict of interest.

\bibliographystyle{amsplain}
\bibliography{References}

\end{document}